\documentclass[12p]{amsart}
\usepackage{amssymb}
\usepackage{amsmath}
\usepackage{amsfonts}
\usepackage{geometry}
\usepackage{graphicx}
\usepackage{mathrsfs,amssymb}

\usepackage{hyperref}
\usepackage{cleveref}

\theoremstyle{plain}

\newtheorem{corollary}{Corollary}

\newtheorem{definition}{Definition}

\newtheorem{lemma}{Lemma}

\newtheorem{remark}{Remark}

\newtheorem{theorem}{Theorem}
\numberwithin{equation}{section}

\begin{document}
\title[]{A determination of the blowup solutions to the focusing, quintic NLS with mass equal to the mass of the soliton}

\author{Benjamin Dodson}

\begin{abstract}
In this paper we prove that the only blowup solutions to the focusing, quintic nonlinear Schr{\"o}dinger equation with mass equal to the mass of the soliton are rescaled solitons or the pseudoconformal transformation of those solitons.
\end{abstract}
\maketitle

\section{Introduction}
The one dimensional, focusing, mass-critical nonlinear Schr{\"o}dinger equation is given by
\begin{equation}\label{1.1}
i u_{t} + u_{xx} + |u|^{4} u = 0, \qquad u(0,x) = u_{0}(x) \in L^{2}(\mathbb{R}).
\end{equation}
This equation is a special case of the Hamiltonian equation
\begin{equation}\label{1.1.1}
i u_{t} + u_{xx} + |u|^{p - 1} u = 0, \qquad u(0,x) = u_{0}(x), \qquad p > 1.
\end{equation}
If $u(t,x)$ is a solution to $(\ref{1.1.1})$, then
\begin{equation}\label{1.2}
v(t,x) = \lambda^{\frac{2}{p - 1}} u(\lambda^{2} t, \lambda x)
\end{equation}
is a solution to $(\ref{1.1.1})$ with appropriately rescaled initial data. Furthermore,
\begin{equation}\label{1.2.1}
\| \lambda^{\frac{2}{p - 1}} u(0, \lambda x) \|_{\dot{H}^{s}(\mathbb{R})} = \lambda^{\frac{2}{p - 1} + s - \frac{1}{2}} \| u_{0} \|_{\dot{H}^{s}(\mathbb{R})},
\end{equation}
so for $s_{p} = \frac{1}{2} - \frac{2}{p - 1}$, the $\dot{H}^{s_{p}}(\mathbb{R})$ norm of the initial data is invariant under the scaling symmetry $(\ref{1.2})$.\medskip

The scaling symmetry in $(\ref{1.2})$ controls the local well-posedness theory of $(\ref{1.1})$. In that case, $p = 5$ and $s_{p} = 0$.
\begin{theorem}\label{t1.1}
The initial value problem $(\ref{1.1})$ is locally well-posed for any $u_{0} \in L^{2}$.
\begin{enumerate}
\item For any $u_{0} \in L^{2}$ there exists $T(u_{0}) > 0$ such that $(\ref{1.1})$ is locally well-posed on the interval $(-T, T)$.

\item If $\| u_{0} \|_{L^{2}}$ is small then $(\ref{1.1})$ is globally well-posed, and the solution scatters both forward and backward in time. That is, there exist $u_{-}$, $u_{+} \in L^{2}(\mathbb{R})$ such that
\begin{equation}\label{1.5}
\lim_{t \nearrow +\infty} \| u(t) - e^{it \Delta} u_{+}  \|_{L^{2}} = 0, \qquad \text{and} \qquad \lim_{t \searrow -\infty} \| u(t) - e^{it \Delta} u_{-} \|_{L^{2}} = 0.
\end{equation}

\item If $I$ is the maximal interval of existence for a solution to $(\ref{1.1})$ with initial data $u_{0}$, $u$ is said to blow up forward in time if
\begin{equation}\label{1.6}
\lim_{T \nearrow \sup(I)} \| u \|_{L_{t,x}^{6}([0, T] \times \mathbb{R})} = +\infty.
\end{equation}
If $u$ does not blow up forward in time, then $\sup(I) = +\infty$ and $u$ scatters forward in time.

\item If $\sup(I) < \infty$ then for any $s > 0$,
\begin{equation}\label{1.7}
\lim_{t \nearrow \sup(I)} \| u(t) \|_{H^{s}} = +\infty.
\end{equation}

\item Time reversal symmetry implies that the results corresponding to $(3)$ and $(4)$ also hold going backward in time.
\end{enumerate}
\end{theorem}
\begin{remark}
It is very important to emphasize that throughout this paper, blow up in positive time may be in finite time or infinite time, unless specified otherwise. The same is true for blow up in negative time.
\end{remark}

\begin{proof}
Theorem $\ref{t1.1}$ was proved in \cite{cazenave1990cauchy}. See also \cite{ginibre1979class}, \cite{ginibre1979class}, \cite{ginibre1985global}, and \cite{kato1987nonlinear}. The proof uses the Strichartz estimates
\begin{equation}\label{1.7.1}
\| u \|_{L_{t}^{\infty} L_{x}^{2} \cap L_{t}^{4} L_{x}^{\infty}(I \times \mathbb{R})} \lesssim \| u_{0} \|_{L^{2}(\mathbb{R})} + \| F \|_{L_{t}^{1} L_{x}^{2} + L_{t}^{4/3} L_{x}^{1}(I \times \mathbb{R})},
\end{equation}
where $u$ is the solution to
\begin{equation}\label{1.7.2}
i u_{t} + u_{xx} = F, \qquad u(0,x) = u_{0},
\end{equation}
on the interval $I$, where $0 \in I$. The Strichartz estimates were proved in \cite{strichartz1977restrictions}, \cite{ginibre1992smoothing}, \cite{yajima1987existence}. Theorem $\ref{t1.1}$ was proved using Picard iteration, so $u$ is a strong solution to $(\ref{1.1})$. For all $t \in I$, where $I$ is the open interval on which local well-posedness of $(\ref{1.1})$ holds,
\begin{equation}\label{1.7.3}
u(t) = e^{it \partial_{xx}} u_{0} + i \int_{0}^{t} e^{i(t - \tau) \partial_{xx}} (|u(\tau)|^{4} u(\tau)) d\tau.
\end{equation}
See \cite{tao2006nonlinear} for different notions of a solution.
\end{proof}

Furthermore, a solution to $(\ref{1.1})$ has the conserved quantities mass,
\begin{equation}\label{1.3}
M(u(t)) = \int |u(t,x)|^{2} dx = M(u(0)),
\end{equation}
and energy,
\begin{equation}\label{1.4}
E(u(t)) = \frac{1}{2} \int |u_{x}(t,x)|^{2} dx - \frac{1}{6} \int |u(t,x)|^{6} dx = E(u(0)).
\end{equation}
For the more general equation $(\ref{1.1.1})$, the Hamiltonian is given by
\begin{equation}\label{1.4.1}
E(u(t)) = \frac{1}{2} \int |u_{x}(t,x)|^{2} dx - \frac{1}{p + 1} \int |u(t,x)|^{p + 1} dx = E(u(0)).
\end{equation}

For $p < 5$, \cite{ginibre1979class} and \cite{ginibre1979class} proved global well-posedness of $(\ref{1.1.1})$ with initial data $u_{0} \in H^{1}(\mathbb{R})$. Indeed, by a straightforward application of the fundamental theorem of calculus and H{\"o}lder's inequality, if $u(t,x) \in L^{2} \cap \dot{H}^{1}$,
\begin{equation}\label{1.8}
|u(t,x)|^{2} \leq \int_{x}^{\infty} |\partial_{y} |u(t,y)|^{2}| dy \leq 2 \int_{x}^{\infty} |\partial_{y} u(t,y)| |u(t,y)| dy \leq 2 \| u \|_{\dot{H}^{1}(\mathbb{R})} \| u \|_{L^{2}(\mathbb{R})}.
\end{equation}
Therefore,
\begin{equation}\label{1.8.1}
\| u(t) \|_{L^{p + 1}(\mathbb{R})}^{p + 1} \lesssim \| u(t) \|_{\dot{H}^{1}(\mathbb{R})}^{\frac{p - 1}{2}} \| u(t) \|_{L^{2}(\mathbb{R})}^{\frac{p + 3}{2}},
\end{equation}
so $(\ref{1.4.1})$ implies the existence of a uniform upper bound on $\| u(t) \|_{\dot{H}^{1}}$ when $p < 5$.

For $p > 5$, there exist singular solutions of $(\ref{1.1.1})$, that is, solutions on the finite interval $[0, T)$, $T < \infty$, for which
\begin{equation}
\lim_{t \rightarrow T} \| u(t) \|_{H^{1}(\mathbb{R})} = \infty.
\end{equation}
See \cite{glassey1977blowing} and \cite{weinstein1986structure}.

When $p = 5$, $(\ref{1.8})$ implies
\begin{equation}\label{1.9}
\int |u(t,x)|^{6} dx \lesssim \| u(t) \|_{\dot{H}^{1}(\mathbb{R})}^{2} \| u(t) \|_{L^{2}(\mathbb{R})}^{4},
\end{equation}
which implies the existence of a threshhold mass $M_{0}$ for which, if $\| u_{0} \|_{L^{2}} < M_{0}$,
\begin{equation}\label{1.10}
E(u(t)) \gtrsim_{M_{0}} \| u(t) \|_{\dot{H}^{1}(\mathbb{R})}^{2},
\end{equation}
with implicit constant $\searrow 0$ as $\| u_{0} \|_{L^{2}} \nearrow M_{0}$. 

From \cite{weinstein1983nonlinear}, the optimal constant in $(\ref{1.9})$ is given by the Gagliardo-Nirenberg inequality,
\begin{equation}\label{1.11}
\| u \|_{L^{6}(\mathbb{R})}^{6} \leq 3 (\frac{\| u \|_{L^{2}}^{2}}{\| Q \|_{L^{2}}^{2}})^{2} \| u_{x} \|_{L^{2}}^{2},
\end{equation}
where
\begin{equation}\label{1.12}
Q(x) = (\frac{3}{\cosh(2x)^{2}})^{1/4}.
\end{equation}
Therefore, if $\| u_{0} \|_{L^{2}} < \| Q \|_{L^{2}}$, then $(\ref{1.11})$ implies
\begin{equation}\label{1.13}
E(u(t)) \gtrsim_{\| u_{0} \|_{L^{2}}} \| u(t) \|_{\dot{H}^{1}}^{2},
\end{equation}
which implies global well-posedness of $(\ref{1.1})$ with initial data $u_{0} \in H^{1}$ and $\| u_{0} \|_{L^{2}} < \| Q \|_{L^{2}}$. Furthermore, the identities
\begin{equation}
\frac{d}{dt} \int x^{2} |u(t,x)|^{2} dx = 4 Im \int x \overline{u(t,x)} u_{x}(t,x) dx,
\end{equation}
and
\begin{equation}
\frac{d^{2}}{dt^{2}} \int x^{2} |u(t,x)|^{2} dx = 16 E(u(t)),
\end{equation}
imply scattering for $(\ref{1.1})$ with initial data in $u_{0} \in H^{1}(\mathbb{R}) \cap \Sigma = \{ u : \int x^{2} |u(x)|^{2} dx < \infty \}$, $\| u_{0} \|_{L^{2}} < \| Q \|_{L^{2}}$.

More recently, \cite{dodson2015global} and \cite{dodson2016global} proved that $(\ref{1.1})$ is globally well-posed and scattering for any initial data $u_{0} \in L^{2}$, $\| u_{0} \|_{L^{2}} < \| Q \|_{L^{2}}$. The proof used the concentration compactness result of \cite{keraani2006blow} and \cite{tao2008minimal} which states that if $u(t)$ is a blowup solution to $(\ref{1.1})$ of minimal mass, and if $t_{n}$ is a sequence of times approaching $\sup(I)$, and if $u$ blows up forward in time on the maximal interval of existence $I$, then $u(t_{n}, x)$ has a subsequence that converges in $L^{2}$, up to the symmetries of $(\ref{1.1})$. Using this fact, \cite{dodson2015global} proved that if $u$ is a minimal mass blowup solution to $(\ref{1.1})$, then there exists a sequence $t_{n}' \rightarrow \sup(I)$, for which $E(v_{n}) \searrow 0$, where $v_{n}$ is a good approximation of $u(t_{n}', x)$, acted on by appropriate symmetries. Since $(\ref{1.13})$ implies that the only $u$ with mass less than $\| Q \|_{L^{2}}^{2}$ and zero energy is $u \equiv 0$, and the small data scattering result implies that the zero solution is stable under small perturbations, there cannot exist a minimal mass blowup solution to $(\ref{1.1})$ with mass less than $\| Q \|_{L^{2}}^{2}$.\medskip

When $\| u \|_{L^{2}} = \| Q \|_{L^{2}}$, $(\ref{1.11})$ only implies that $E(u) \geq 0$. The $Q(x)$ in $(\ref{1.12})$ is the unique, positive solution to
\begin{equation}\label{1.14}
Q_{xx} + Q^{5} = Q.
\end{equation}
See \cite{berestycki1981ode}, \cite{MR512091}, \cite{strauss1977existence}, and \cite{kwong1989uniqueness} for existence and uniqueness of a ground state solution in general dimensions. Also observe that by the Pohozaev identity,
\begin{equation}\label{1.16}
E(Q) = \frac{1}{2} \int (Q - Q_{xx} - Q^{5})(\frac{Q}{2} + x Q_{x}) dx = 0.
\end{equation} 
Up to the scaling $(\ref{1.2})$, multiplication by a modulus one constant, and translation in space, $Q$ is the unique minimizer of the energy functional with mass $\| Q \|_{L^{2}}$. See \cite{cazenave1982orbital} and \cite{weinstein1986structure}.\medskip

It is straightforward to verify that $(\ref{1.12})$ solves $(\ref{1.14})$, and that $e^{it} Q$ solves $(\ref{1.1})$. Since $\| e^{it} Q \|_{L^{6}}$ is constant for all $t \in \mathbb{R}$, $e^{it} Q$ blows up both forward and backward in time. Furthermore, the pseudoconformal transformation of $e^{it} Q(x)$,
\begin{equation}\label{1.15}
u(t,x) = \frac{1}{t^{1/2}} e^{\frac{-i}{t} + \frac{i x^{2}}{4t}} Q(\frac{x}{t}), \qquad t > 0,
\end{equation}
is a solution to $(\ref{1.1})$ that blows up as $t \searrow 0$, and scatters as $t \rightarrow \infty$. Note that the mass is preserved under the pseudoconformal transformation of $e^{it} Q$.\medskip

It has long been conjecture that, up to symmetries of equation $(\ref{1.1})$, the only non-scattering solutions to $(\ref{1.1})$ are the soliton $e^{it} Q$ and the pseudoconformal transformation of the soliton, $(\ref{1.15})$. Partial progress has been made in this direction.\begin{theorem}
If $u_{0} \in H^{1}$, $\| u_{0} \|_{L^{2}} = \| Q \|_{L^{2}}$ and the solution $u(t)$ to $(\ref{1.1})$ blows up in finite time $T > 0$, then $u(t,x)$ is equal to $(\ref{1.15})$, up to symmetries of $(\ref{1.1})$. 
\end{theorem}
\begin{proof} 
This result was proved in \cite{merle1992uniqueness} and \cite{merle1993determination}, and was proved for the focusing, mass-critical nonlinear Schr{\"o}dinger equation in every dimension.
\end{proof}

For the mass-critical nonlinear Schr{\"o}dinger equation in higher dimensions with radially symmetric initial data, \cite{killip2009characterization} proved
\begin{theorem}
If $\| u_{0} \|_{L^{2}} = \| Q \|_{L^{2}}$ is radially symmetric, and $u$ is the solution to the focusing, mass-critical nonlinear Schr{\"o}dinger equation with initial data $u_{0}$, and $u$ blows up both forward and backward in time, then $u$ is equal to the soliton, up to symmetries of the mass-critical nonlinear Schr{\"o}dinger equation.
\end{theorem}

In this paper we completely resolve this conjecture in one dimension, showing that the only blowup solutions to $(\ref{1.1})$ with mass $\| u_{0} \|_{L^{2}}^{2} = \| Q \|_{L^{2}}^{2}$ are the soliton and the pseudoconformal transformation of the soliton. This result should also hold in higher dimensions, which will be addressed in a forthcoming paper.

It is convenient to begin by considering solutions symmetric in $x$ first.
\begin{theorem}\label{t1.1}
The only symmetric solutions to $(\ref{1.1})$ with mass $\| u_{0} \|_{L^{2}} = \| Q \|_{L^{2}}$ that blow up forward in time are the family of soliton solutions
\begin{equation}\label{1.16}
e^{-i \theta} e^{i \lambda^{2} t} \lambda^{1/2} Q(\lambda x), \qquad \lambda > 0, \qquad \theta \in \mathbb{R},
\end{equation}
and the pseudoconformal transformation of the soliton solution,
\begin{equation}\label{1.17}
\frac{1}{(T - t)^{1/2}} e^{i \theta} e^{\frac{i x^{2}}{4(t - T)}} e^{i \frac{\lambda^{2}}{t - T}} Q(\frac{\lambda x}{T - t}), \qquad \lambda > 0, \qquad \theta \in \mathbb{R}, \qquad T \in \mathbb{R}, \qquad t < T.
\end{equation}
\end{theorem}

The proof of Theorem $\ref{t1.1}$ will occupy most of the paper. Once we have proved Theorem $\ref{t1.1}$, we will remove the symmetry assumption on $u_{0}$, proving
\begin{theorem}\label{t1.2}
The only solutions to $(\ref{1.1})$ with mass $\| u_{0} \|_{L^{2}} = \| Q \|_{L^{2}}$ that blow up forward in time are the family of soliton solutions
\begin{equation}\label{1.16.1}
e^{-i \theta - it \xi_{0}^{2}} e^{i \lambda^{2} t} e^{ix \xi_{0}} \lambda^{1/2} Q(\lambda (x - 2 t \xi_{0}) + x_{0}), \qquad \lambda > 0, \qquad \theta \in \mathbb{R}, \qquad x_{0} \in \mathbb{R}, \qquad \xi_{0} \in \mathbb{R}, \qquad \xi_{0} \in \mathbb{R},
\end{equation}
and the pseudoconformal transformation of the family of solitons,
\begin{equation}\label{1.17.1}
\aligned
\frac{1}{(T - t)^{1/2}} e^{i \theta} e^{\frac{i (x - \xi_{0})^{2}}{4(t - T)}} e^{i \frac{\lambda^{2}}{t - T}} Q(\frac{\lambda (x - \xi_{0}) - (T - t) x_{0}}{T - t}), \\ \text{where} \qquad \lambda > 0, \qquad \theta \in \mathbb{R}, \qquad x_{0} \in \mathbb{R}, \qquad \xi_{0} \in \mathbb{R}, \qquad T \in \mathbb{R}, \qquad t < T.
\endaligned
\end{equation}
Applying time reversal symmetry to $(\ref{1.1})$, this theorem completely settles the question of qualitative behavior of solutions to $(\ref{1.1})$ for initial data satisfying $\| u_{0} \|_{L^{2}} = \| Q \|_{L^{2}}$.
\end{theorem}

The reader should see \cite{nakanishi2011invariant} and the references therein for this result for the Klein--Gordon equation.

\section{Reductions of a symmetric blowup solution}
Let $u$ be a symmetric blowup solution to $(\ref{1.1})$ with mass $\| u_{0} \|_{L^{2}} = \| Q \|_{L^{2}}$. Defining the distance to the two dimensional manifold of symmetries acting on the soliton $(\ref{1.12})$ by
\begin{equation}\label{3.0.1}
\inf_{\lambda > 0, \gamma \in \mathbb{R}} \| u_{0}(x) - e^{i \gamma} \lambda^{1/2} Q(\lambda x) \|_{L^{2}},
\end{equation}
there exist $\lambda_{0} > 0$ and $\gamma_{0} \in \mathbb{R}$ where this infimum is attained. Indeed,

\begin{lemma}\label{l3.3}
There exist $\lambda_{0} > 0$ and $\gamma_{0} \in \mathbb{R}$ such that
\begin{equation}\label{3.17}
\| u_{0}(x) - e^{-i \gamma_{0}} \lambda_{0}^{-1/2} Q(\lambda_{0}^{-1} x) \|_{L^{2}(\mathbb{R})} = \inf_{\gamma, \lambda} \| u_{0}(x) - e^{-i \gamma} \lambda^{-1/2} Q(\lambda^{-1} x) \|_{L^{2}}.
\end{equation}
\end{lemma}
\begin{proof}
Since $Q$, along with all its derivatives, is rapidly decreasing,
\begin{equation}\label{3.18}
\| u_{0}(x) - e^{-i \gamma} \lambda^{-1/2} Q(\lambda^{-1} x) \|_{L^{2}}^{2},
\end{equation}
is differentiable, and hence continuous as a function of $\lambda$ and $\gamma$.\medskip

Next, by the dominated convergence theorem,
\begin{equation}\label{3.19}
\aligned
\lim_{\lambda \nearrow \infty} \inf_{\gamma \in [0, 2 \pi]} \| u_{0}(x) - e^{-i \gamma} \lambda^{-1/2} Q(\lambda^{-1} x) \|_{L^{2}}^{2} \\ = \| u_{0} \|_{L^{2}}^{2} + \| Q \|_{L^{2}}^{2} - 2 \lim_{\lambda \nearrow \infty} \sup_{\gamma \in [0, 2 \pi]} |(e^{-i \gamma} \lambda^{-1/2} Q(\lambda^{-1} x), u_{0}(x))_{L^{2}}| = 2 \| Q \|_{L^{2}}^{2}.
\endaligned
\end{equation}
Here, $(f, g)_{L^{2}}$ denotes the $L^{2}$-inner product
\begin{equation}\label{3.19.1}
(f, g)_{L^{2}} = Re \int f(x) \overline{g(x)} dx.
\end{equation}
Meanwhile, rescaling $(\ref{3.19})$,
\begin{equation}\label{3.19.2}
(e^{-i \gamma} \lambda^{-1/2} Q(\lambda^{-1} x), u_{0}(x))_{L^{2}} = (Q(x), e^{i \gamma} \lambda^{1/2} u_{0}(\lambda x))_{L^{2}},
\end{equation}
and therefore,
\begin{equation}\label{3.20}
\aligned
\lim_{\lambda \searrow 0} \inf_{\gamma \in [0, 2 \pi]} \| u_{0}(x) - e^{-i \gamma} \lambda^{-1/2} Q(\lambda^{-1} x) \|_{L^{2}}^{2} = 2 \| Q \|_{L^{2}}^{2}.
\endaligned
\end{equation}

Finally, the polarization identity,
\begin{equation}\label{3.21}
\| u_{0}(x) - \lambda^{-1/2} Q(\lambda^{-1} x) \|_{L^{2}}^{2} + \| u_{0}(x) + \lambda^{-1/2} Q(\lambda^{-1} x) \|_{L^{2}}^{2} = 4 \| Q \|_{L^{2}}^{2},
\end{equation}
implies that
\begin{equation}\label{3.22}
\frac{1}{2 \pi} \int_{0}^{2 \pi} \| u_{0}(x) - e^{-i \gamma} \lambda^{-1/2} Q(\lambda^{-1} x) \|_{L^{2}}^{2} d\gamma = 2 \| Q \|_{L^{2}}^{2}.
\end{equation}
If, for all $\lambda > 0$,
\begin{equation}\label{3.23}
\inf_{\gamma \in [0, 2\pi]} \| u_{0}(x) - e^{-i \gamma} \lambda^{-1/2} Q(\lambda^{-1} x) \|_{L^{2}}^{2} = 2 \| Q \|_{L^{2}}^{2},
\end{equation}
then $(\ref{3.22})$ implies
\begin{equation}\label{3.24}
\| u_{0}(x) - e^{-i \gamma} \lambda^{-1/2} Q(\lambda^{-1} x) \|_{L^{2}}^{2} = 2 \| Q \|_{L^{2}}^{2}, \qquad \forall \lambda > 0, \qquad \gamma \in [0, 2 \pi].
\end{equation}
In this case simply take $\lambda_{0} = 1$ and $\gamma_{0} = 0$.
\begin{remark}
This seems very unlikely to the author, since $(\ref{3.24})$ is equivalent to the statement that there exists $\| u_{0} \|_{L^{2}} = \| Q \|_{L^{2}}$ that satisfies
\begin{equation}\label{3.24.1}
(u_{0}(x), e^{-i \gamma} \lambda^{-1/2} Q(\lambda^{-1} x))_{L^{2}} = 0, \qquad \forall \gamma \in [0, 2\pi], \qquad \forall \lambda > 0.
\end{equation}
Since it is unnecessary to the proof of Theorem $\ref{t1.1}$ to show that $(\ref{3.24.1})$ cannot happen, this question will remain unconsidered in this paper.
\end{remark}

On the other hand, if
\begin{equation}\label{3.25}
\inf_{\lambda > 0} \inf_{\gamma \in \mathbb{R}} \| u_{0}(x) - e^{-i \gamma} \lambda^{-1/2} Q(\lambda^{-1} x) \|_{L^{2}}^{2} < 2 \| Q \|_{L^{2}}^{2},
\end{equation}
then $(\ref{3.19})$ and $(\ref{3.20})$ imply that there exist $0 < \lambda_{1} < \lambda_{2} < \infty$, such that
\begin{equation}\label{3.25.1}
\inf_{\lambda > 0} \inf_{\gamma \in [0, 2\pi]} \| u_{0}(x) - e^{-i \gamma} \lambda^{-1/2} Q(\lambda^{-1} x) \|_{L^{2}}^{2} = \inf_{\lambda \in [\lambda_{1}, \lambda_{2}]} \inf_{\gamma \in [0, 2\pi]} \| u_{0}(x) - e^{-i \gamma} \lambda^{-1/2} Q(\lambda^{-1} x) \|_{L^{2}}^{2}.
\end{equation}
Since $(\ref{3.18})$ is continuous as a function of $\lambda > 0$, $\gamma \in [0, 2\pi]$, and $[\lambda_{1}, \lambda_{2}] \times [0, 2 \pi]$ is a compact set, there exists $\lambda_{0} > 0$ and $\gamma_{0} \in [0, 2\pi]$ such that
\begin{equation}\label{3.26}
\| u_{0}(x) - e^{-i \gamma_{0}} \lambda_{0}^{-1/2} Q(\lambda_{0}^{-1} x) \|_{L^{2}(\mathbb{R})} = \inf_{\gamma \in [0, 2 \pi], \lambda > 0} \| u_{0}(x) - e^{-i \gamma} \lambda^{-1/2} Q(\lambda^{-1} x) \|_{L^{2}}.
\end{equation}
This proves the lemma.
\end{proof}

Using the weak sequential convergence result of \cite{fan20182}, Theorem $\ref{t1.1}$ may be reduced to considering solutions that blow up in positive time for which $(\ref{3.0.1})$ is small for all $t > 0$.


\begin{theorem}\label{t3.1}
Let $0 < \eta_{\ast} \ll 1$ be a small, fixed constant to be defined later. If $u$ is a symmetric solution to $(\ref{1.1})$ on the maximal interval of existence $I \subset \mathbb{R}$, $\| u_{0} \|_{L^{2}} = \| Q \|_{L^{2}}$, $u$ blows up forward in time, and
\begin{equation}\label{3.13}
\sup_{t \in [0, \sup(I))} \inf_{\lambda, \gamma} \| e^{i \gamma} \lambda^{1/2} u(t, \lambda x) - Q(x) \|_{L^{2}} \leq \eta_{\ast},
\end{equation}
then $u$ is a soliton solution of the form $(\ref{1.16})$ or a pseudoconformal transformation of a soliton of the form $(\ref{1.17})$.

\end{theorem}
\begin{remark}
Scaling symmetries imply that $(\ref{3.0.1})$ and the left hand side of $(\ref{3.13})$ at a fixed time are equal.
\end{remark}
\begin{proof}[Proof that Theorem $\ref{t3.1}$ implies Theorem $\ref{t1.1}$]
Let $u(t)$ be the solution to $(\ref{1.1})$ with symmetric initial data $u_{0}$ that satisfies $\| u_{0} \|_{L^{2}} = \| Q \|_{L^{2}}$. If 
\begin{equation}\label{3.15}
\lim_{t \nearrow \sup(I)} \inf_{\lambda > 0, \gamma \in \mathbb{R}} \| \lambda^{1/2} e^{i \gamma} u(t, \lambda x) - Q \|_{L^{2}} = 0,
\end{equation}
then $(\ref{3.13})$ holds for all $t > t_{0}$, for some $t_{0} \in I$. After translating in time so that $t_{0} = 0$, Theorem $\ref{t3.1}$ easily implies Theorem $\ref{t1.1}$ in this case.

However, the convergence theorem of \cite{fan20182} only implies $u(t)$ must converge to $Q$ along a subsequence after rescaling and multiplying by a complex number of modulus one.
\begin{theorem}\label{t2.1}
Let $u$ be a symmetric solution to $(\ref{1.1})$ that satisfies $\| u_{0} \|_{L^{2}} = \| Q \|_{L^{2}}$ and blows up forward in time. Let $(T_{-}(u), T_{+}(u))$ be the maximal lifespan of the solution $u$. Then there exists a sequence $t_{n} \rightarrow T_{+}(u)$ and a family of parameters $\lambda_{n} > 0$, $\gamma_{n} \in \mathbb{R}$ such that
\begin{equation}\label{2.3}
e^{i \gamma_{n}} \lambda_{n}^{1/2} u(t_{n}, \lambda_{n} x) \rightarrow Q, \qquad \text{in} \qquad L^{2}.
\end{equation}
\end{theorem}

If $(\ref{3.15})$ does not hold, but there exists some $t_{0} > 0$ such that
\begin{equation}
\sup_{t \in [t_{0}, \sup(I))} \inf_{\lambda, \gamma} \| e^{i \gamma} \lambda^{1/2} u(t, \lambda x) - Q(x) \|_{L^{2}} \leq \eta_{\ast},
\end{equation}
then after translating in time so that $t_{0} = 0$, $(\ref{3.13})$ holds.

Now suppose $(\ref{3.15})$ does not hold, and furthermore that there exists a sequence $t_{n}^{-} \nearrow \sup(I)$ such that
\begin{equation}\label{3.3.1}
\inf_{\gamma \in \mathbb{R}, \lambda > 0} \| e^{i \gamma} \lambda^{1/2} u(t_{n}^{-}, \lambda x) - Q \|_{L^{2}} > \eta_{\ast},
\end{equation}
for every $n$. After passing to a subsequence, suppose that for every $n$, $t_{n}^{-} < t_{n} < t_{n + 1}^{-}$, where $t_{n}$ is the sequence in $(\ref{2.3})$ and $t_{n}^{-}$ is the sequence in $(\ref{3.3.1})$. The fact that
\begin{equation}\label{3.3.2}
\inf_{\gamma \in \mathbb{R}, \lambda > 0} \| e^{i \gamma} \lambda^{1/2} u(t, \lambda x) - Q \|_{L^{2}}
\end{equation}
is upper semicontinuous as a function of $t$, and is continuous for every $t$ such that $(\ref{3.3.2})$ is small guarantees that there exists a small, fixed $0 < \eta_{\ast} \ll 1$ such that the sequence $t_{n}^{+}$, defined by,
\begin{equation}\label{3.3}
t_{n}^{+} = \inf \{ t \in I : \sup_{\tau \in [t, t_{n}]} \inf_{\lambda, \gamma} \| \lambda^{1/2} e^{i \gamma} u(\tau, \lambda x) - Q \|_{L^{2}}  < \eta_{\ast} \},
\end{equation}
satisfies $t_{n}^{+} \nearrow \sup(I)$ and
\begin{equation}
\inf_{\lambda > 0, \gamma \in \mathbb{R}} \| e^{i \gamma} \lambda^{1/2} u(t_{n}^{+}, \lambda x) - Q(x) \|_{L^{2}} = \eta_{\ast}.
\end{equation}
Indeed, the fact that $(\ref{3.3.2})$ is upper semicontinuous as a function of $t$ implies that
\begin{equation}
\{ 0 \leq t < t_{n} : \inf_{\lambda > 0, \gamma \in \mathbb{R}} \| e^{i \gamma} \lambda^{1/2} u(t, \lambda x) - Q(x) \|_{L^{2}} \geq \eta_{\ast} \}
\end{equation}
is a closed set. Since this set is also contained in a bounded set, it has a maximal element $t_{n}^{+}$, and $t_{n}^{+} \geq t_{n}^{-}$. The fact that $(\ref{3.3.2})$ is upper semicontinuous in time also implies that
\begin{equation}
\inf_{\lambda > 0, \gamma \in \mathbb{R}} \| e^{i \gamma} \lambda^{1/2} u(t_{n}^{+}, \lambda x) - Q \|_{L^{2}} \geq \eta_{\ast}.
\end{equation}
On the other hand, since
\begin{equation}
\inf_{\lambda > 0, \gamma \in \mathbb{R}} \| e^{i \gamma} \lambda^{1/2} u(t, \lambda x) - Q \|_{L^{2}} < \eta_{\ast} \qquad \text{for all} \qquad t_{n}^{+} < t < t_{n},
\end{equation}
and $(\ref{3.3.2})$ is continuous at times $t \in I$ where $(\ref{3.3.2})$ is small,
\begin{equation}\label{3.16.1}
\inf_{\lambda > 0, \gamma \in \mathbb{R}} \| e^{i \gamma} \lambda^{1/2} u(t_{n}^{+}, \lambda x) - Q \|_{L^{2}} = \eta_{\ast}.
\end{equation}
\begin{remark}
The constant $0 < \eta_{\ast} \ll 1$ will be chosen to be a small fixed quantity that is sufficiently small to satisfy the hypotheses of Theorem $\ref{t2.3}$, sufficiently small such that $(\ref{3.3.2})$ is continuous in time when $(\ref{3.3.2})$ is bounded by $\eta_{\ast}$, sufficiently small such that $\eta_{\ast} \leq \eta_{0}$, where $\eta_{0}$ is the constant in the induction on frequency arguments in Theorem $\ref{t6.2}$, and so that $T_{\ast} = \frac{1}{\eta_{\ast}}$ is sufficiently large to satisfy the hypotheses of Theorem $\ref{t10.1}$. 
\end{remark}

\begin{theorem}[Upper semicontinuity of the distance to a soliton]\label{t3.2}
The quantity
\begin{equation}\label{3.16}
\inf_{\lambda, \gamma} \| e^{i \gamma} \lambda^{1/2} u(t, \lambda x) - Q(x) \|_{L^{2}(\mathbb{R})},
\end{equation}
is upper semicontinuous as a function of time for any $t \in I$, where $I$ is the maximal interval of existence for $u$. The quantity $(\ref{3.16})$ is also continuous in time when $(\ref{3.16})$ is small.
\end{theorem}
\begin{proof}
Choose some $t_{0} \in I$ and suppose without loss of generality that
\begin{equation}\label{3.27}
\|  u(t_{0},  x) - Q(x) \|_{L^{2}} = \inf_{\lambda, \gamma} \| e^{i \gamma} \lambda^{1/2} u(t_{0}, \lambda x) - Q(x) \|_{L^{2}}.
\end{equation}
For $t$ close to $t_{0}$, let
\begin{equation}\label{3.28}
\epsilon(t, x) = u(t, x) - e^{i(t - t_{0})} Q(x).
\end{equation}
Since $e^{i(t - t_{0})} Q$ solves $(\ref{1.1})$,
\begin{equation}\label{3.29}
i \epsilon_{t} + \epsilon_{xx} + |u|^{4} u - e^{it} |Q|^{4} Q = i \epsilon_{t} + \epsilon_{xx} + 3 |Q|^{4} \epsilon + 2 e^{2i (t - t_{0})} |Q|^{2} Q^{2} \bar{\epsilon} + O(\sum_{j = 2}^{5} \epsilon^{2 + j} Q^{5 - j}) = 0.
\end{equation}
Equations $(\ref{3.28})$, $(\ref{3.29})$, and Strichartz estimates imply that for $J \subset \mathbb{R}$, $t_{0} \in J$,
\begin{equation}\label{3.30}
\| \epsilon \|_{L_{t}^{\infty} L_{x}^{2} \cap L_{t}^{4} L_{x}^{\infty}(J \times \mathbb{R})} \lesssim \| \epsilon(t_{0}) \|_{L^{2}} + \| \epsilon \|_{L_{t}^{\infty} L_{x}^{2}(J \times \mathbb{R})} \| u \|_{L_{t}^{4} L_{x}^{\infty}(J \times \mathbb{R})}^{4} + \| \epsilon \|_{L_{t,x}^{6}(J \times \mathbb{R})}^{5}.
\end{equation}
Local well-posedness of $(\ref{1.1})$ combined with Strichartz estimates implies that $\| u \|_{L_{t}^{4} L_{x}^{\infty}(J \times \mathbb{R})} = 1$ on some open neighborhood $J$ of $t_{0}$. Therefore, for $\| \epsilon(t_{0}) \|_{L^{2}}$ small, partitioning $J$ into finitely many pieces,
\begin{equation}\label{3.31}
\sup_{t \in J} \| \epsilon(t) \|_{L^{2}} \lesssim \| \epsilon(t_{0}) \|_{L^{2}},
\end{equation}
and
\begin{equation}\label{3.32}
\lim_{t \rightarrow t_{0}} \| \epsilon(t) \|_{L^{2}} = \| \epsilon(t_{0}) \|_{L^{2}}.
\end{equation}

Therefore,
\begin{equation}\label{3.33}
\lim_{t \rightarrow t_{0}} \inf_{\lambda, \gamma} \| \lambda^{1/2} e^{i \gamma} u(t, \lambda x) - Q \|_{L^{2}} \leq \| u(t_{0},x) - Q \|_{L^{2}} = \inf_{\lambda > 0, \gamma \in \mathbb{R}} \| \lambda^{1/2} e^{i \gamma} u(t_{0}, \lambda x) - Q \|_{L^{2}}.
\end{equation}
Furthermore, if
\begin{equation}\label{3.34}
\lim_{t \rightarrow t_{0}} \inf_{\lambda, \gamma} \| \lambda^{1/2} e^{i \gamma} u(t, \lambda x) - Q \|_{L^{2}} <  \| u(t_{0},x) - Q \|_{L^{2}}.
\end{equation}
Then there exists a sequence $t_{n}' \rightarrow t_{0}$, $\lambda_{n}' > 0$, $\gamma_{n}' \in \mathbb{R}$ such that
\begin{equation}\label{3.35}
\lim_{n \rightarrow \infty} \| \lambda_{n}'^{1/2} e^{i \gamma_{n}'} u(t_{n}', \lambda_{n}' x) - Q \|_{L^{2}} < \inf_{\lambda, \gamma} \| \lambda^{1/2} e^{i \gamma} u(t_{0}, \lambda x) \|_{L^{2}}.
\end{equation}
For $t_{n}'$ sufficiently close to $t_{0}$, repeating the arguments giving $(\ref{3.31})$ and $(\ref{3.32})$ with $t_{n}'$ as the initial data gives a contradiction.

When $\| \epsilon(t_{0}) \|_{L^{2}}$ is large, $(\ref{3.29})$ implies
\begin{equation}\label{3.36}
\frac{d}{dt} \| \epsilon(t) \|_{L^{2}}^{2} \lesssim \| Q \|_{L^{\infty}}^{4} \| \epsilon \|_{L^{2}}^{2} + \| u \|_{L^{\infty}}^{4} \| \epsilon \|_{L^{2}}^{2}.
\end{equation}
Therefore, Gronwall's inequality and the fact that $u \in L_{t, loc}^{4} L_{x}^{\infty}$ imply
\begin{equation}\label{3.37}
\lim_{t \rightarrow t_{0}} \inf_{\lambda > 0, \gamma \in \mathbb{R}} \| e^{i \gamma} \lambda^{1/2} u(t, \lambda x) - Q \|_{L^{2}} \leq \inf_{\lambda > 0, \gamma \in \mathbb{R}} \| e^{i \gamma} \lambda^{1/2} u(t_{0}, \lambda x) - Q \|_{L^{2}},
\end{equation}
which implies upper semicontinuity.
\end{proof}

Making a profile decomposition of $u(t_{n}^{+}, x)$, the fact that $u$ is a minimal mass blowup solution that blows up forward in time and $t_{n}^{+} \nearrow \sup(I)$ implies that there exist $\lambda(t_{n}^{+}) > 0$ and $\gamma(t_{n}^{+}) \in \mathbb{R}$ such that
\begin{equation}\label{3.5}
\lambda(t_{n}^{+})^{1/2} e^{i \gamma(t_{n}^{+})} u(t_{n}^{+}, \lambda(t_{n}^{+}) x) \rightarrow \tilde{u}_{0},
\end{equation}
in $L^{2}$. Also, $t_{n}^{+} \nearrow \sup(I)$ implies $\| \tilde{u}_{0} \|_{L^{2}} = \| Q \|_{L^{2}}$ is the initial data for a solution to $(\ref{1.1})$ that blows up forward and backward in time, and by $(\ref{3.16.1})$,
\begin{equation}\label{3.6}
\inf_{\lambda > 0, \gamma \in \mathbb{R}} \| \lambda^{1/2} e^{i \gamma} \tilde{u}_{0}(\lambda x) - Q \|_{L^{2}} = \eta_{\ast}.
\end{equation}
Moreover, observe that $(\ref{2.3})$ and $(\ref{3.31})$ directly imply that
\begin{equation}\label{3.36}
\lim_{n \rightarrow \infty} \| u \|_{L_{t,x}^{6}([t_{n}^{+}, t_{n}] \times \mathbb{R})} = \infty,
\end{equation}
so if $\tilde{u}$ is the solution to $(\ref{1.1})$ with initial data $\tilde{u}_{0}$,
\begin{equation}\label{3.7}
\inf_{\lambda > 0, \gamma \in \mathbb{R}} \| \lambda^{1/2} e^{i \gamma} \tilde{u}(t, \lambda x) - Q \|_{L^{2}} \leq \eta_{\ast},
\end{equation}
for all $t \in [0, \sup(\tilde{I}))$, where $\tilde{I}$ is the interval of existence of the solution $\tilde{u}$ to $(\ref{1.1})$ with initial data $\tilde{u}_{0}$, and $\tilde{u}$ blows up both forward and backward in time. However, Theorem $\ref{t3.1}$ and $(\ref{3.6})$ imply that $\tilde{u}$ must be of the form $(\ref{1.17})$. Such a solution scatters backward in time, which contradicts the fact that $\tilde{u}$ blows up both forward and backward in time.


Therefore, Theorem $\ref{t3.1}$ implies that $(\ref{3.3.1})$ cannot hold for any symmetric solution to $(\ref{1.1})$ with mass $\| u_{0} \|_{L^{2}} = \| Q \|_{L^{2}}$, so by Theorem $\ref{t3.1}$, any symmetric solution to $(\ref{1.1})$ that blows up forward in time must be of the form $(\ref{1.16})$ or $(\ref{1.17})$.

\end{proof}

\section{Decomposition of the solution near $Q$}
Turning now to the proof of Theorem $\ref{t3.1}$, make a decomposition of a symmetric solution close to $Q$, up to rescaling and multiplication by a modulus one constant. This result is classical, see for example \cite{martel2002stability}, although here there is an additional technical complication due to the fact that $u$ need not lie in $H^{1}$.

\begin{theorem}\label{t2.3}
Take $u \in L^{2}$. There exists $\alpha > 0$ sufficiently small such that if there exist $\lambda_{0} > 0$, $\gamma_{0} \in \mathbb{R}$ that satisfy
\begin{equation}\label{2.12}
\| e^{i \gamma_{0}} \lambda_{0}^{1/2} u(\lambda_{0} x) - Q \|_{L^{2}} \leq \alpha,
\end{equation}
then there exist unique $\lambda > 0$, $\gamma \in \mathbb{R}$ which satisfy
\begin{equation}\label{2.13}
(\epsilon, Q^{3})_{L^{2}} = (\epsilon, i Q^{3})_{L^{2}} = 0,
\end{equation}
where
\begin{equation}\label{2.14}
\epsilon(x) = e^{i \gamma} \lambda^{1/2} u(\lambda x) - Q.
\end{equation}
Furthermore,
\begin{equation}\label{2.14.1}
\| \epsilon \|_{L^{2}} + |\frac{\lambda}{\lambda_{0}} - 1| + |\gamma - \gamma_{0}| \lesssim \| e^{i \gamma_{0}} \lambda_{0}^{1/2} u(\lambda_{0} x) - Q \|_{L^{2}}.
\end{equation}
\end{theorem}
\begin{remark}
Since $e^{i \gamma}$ is $2\pi$-periodic, the $\gamma$ in $(\ref{2.14})$ is unique up to translations by $2 \pi k$ for some integer $k$.
\end{remark}

\begin{proof}
By H{\"o}lder's inequality,
\begin{equation}\label{2.16}
|(e^{i \gamma_{0}} \lambda_{0}^{1/2} u(\lambda_{0} x) - Q(x), Q^{3})_{L^{2}}| \lesssim \| e^{i \gamma_{0}} \lambda_{0}^{1/2} u(\lambda_{0} x) - Q \|_{L^{2}},
\end{equation}
and
\begin{equation}\label{2.17}
|(e^{i \gamma_{0}} \lambda_{0}^{1/2} u(\lambda_{0} x) - Q(x), i Q^{3})_{L^{2}}| \lesssim \| e^{i \gamma_{0}} \lambda_{0}^{1/2} u(\lambda_{0} x) - Q \|_{L^{2}}.
\end{equation}

First suppose that $\lambda_{0} = 1$ and $\gamma_{0} = 0$. The inner products
\begin{equation}\label{2.18}
( e^{i \gamma} \lambda^{1/2} u( \lambda x) - Q(x), Q^{3} )_{L^{2}},
\end{equation}
and
\begin{equation}\label{2.19}
( e^{i \gamma} \lambda^{1/2} u(\lambda x) - Q(x), i Q^{3} )_{L^{2}},
\end{equation}
are $C^{1}$ as functions of $\lambda$ and $\gamma$. Indeed,
\begin{equation}\label{2.20}
\frac{\partial}{\partial \gamma} ( e^{i \gamma} \lambda^{1/2} u(\lambda x) - Q(x), Q^{3} )_{L^{2}} = ( i e^{i \gamma} \lambda^{1/2} u(\lambda x), Q^{3} )_{L^{2}} \lesssim \| u \|_{L^{2}} \| Q \|_{L^{6}}^{3},
\end{equation}
and
\begin{equation}\label{2.21}
\frac{\partial}{\partial \gamma} ( e^{i \gamma} \lambda^{1/2} u(\lambda x) - Q(x), i Q^{3} )_{L^{2}} = ( i e^{i \gamma} \lambda^{1/2} u(\lambda x), i Q^{3} )_{L^{2}} \lesssim \| u \|_{L^{2}} \| Q \|_{L^{6}}^{3}.
\end{equation}
Next, integrating by parts,
\begin{equation}\label{2.22}
\aligned
\frac{\partial}{\partial \lambda} ( e^{i \gamma} \lambda^{1/2} u(\lambda x) - Q(x), Q^{3} )_{L^{2}} = ( \frac{e^{i \gamma}}{2 \lambda^{1/2}} u(\lambda x) + x e^{i \gamma} \lambda^{1/2} u_{x}(\lambda x), Q^{3} )_{L^{2}} \\
= ( \frac{e^{i \gamma}}{2 \lambda^{1/2}} u(\lambda x) - \frac{1}{\lambda^{1/2}} e^{i \gamma} u(\lambda x), Q^{3} )_{L^{2}} - \frac{3}{\lambda^{1/2}} ( e^{i \gamma} u(\lambda x), Q^{2} Q_{x} )_{L^{2}} \\ \lesssim \frac{1}{\lambda} \| u \|_{L^{2}} \| Q \|_{L^{6}}^{3} + \frac{1}{\lambda} \| u \|_{L^{2}} \| x Q_{x} \|_{L^{2}} \| Q \|_{L^{\infty}}^{2},
\endaligned
\end{equation}
and
\begin{equation}\label{2.23}
\aligned
\frac{\partial}{\partial \lambda} ( e^{i \gamma} \lambda^{1/2} u(\lambda x) - Q(x), i Q^{3} )_{L^{2}} = ( \frac{e^{i \gamma}}{2 \lambda^{1/2}} u(\lambda x) + x e^{i \gamma} \lambda^{1/2} u_{x}(\lambda x), i Q^{3} )_{L^{2}} \\
= ( \frac{e^{i \gamma}}{2 \lambda^{1/2}} u(\lambda x) - \frac{1}{\lambda^{1/2}} e^{i \gamma} u(\lambda x), i Q^{3} )_{L^{2}} - \frac{3}{\lambda^{1/2}} ( e^{i \gamma} u(\lambda x), i Q^{2} Q_{x} )_{L^{2}} \\ \lesssim \frac{1}{\lambda} \| u \|_{L^{2}} \| Q \|_{L^{6}}^{3} + \frac{1}{\lambda} \| u \|_{L^{2}} \| x Q_{x} \|_{L^{2}} \| Q \|_{L^{\infty}}^{2}.
\endaligned
\end{equation}
Similar calculations prove uniform bounds on the Hessians of $(\ref{2.18})$ and $(\ref{2.19})$.

Next, compute
\begin{equation}\label{2.24}
\frac{\partial}{\partial \gamma} ( e^{i \gamma} \lambda^{1/2} u(\lambda x) - Q(x), Q^{3} )_{L^{2}}|_{\lambda = 1, \gamma = 0, u = Q} = ( i Q, Q^{3} )_{L^{2}} = 0,
\end{equation}
\begin{equation}\label{2.25}
\frac{\partial}{\partial \gamma} ( e^{i \gamma} \lambda^{1/2} u(\lambda x) - Q(x), i Q^{3} )_{L^{2}}|_{\lambda = 1, \gamma = 0, u = Q} = ( i Q, i Q^{3} )_{L^{2}} = \| Q \|_{L^{4}}^{4},
\end{equation}
\begin{equation}\label{2.26}
\frac{\partial}{\partial \lambda} ( e^{i \gamma} \lambda^{1/2} u(\lambda x) - Q(x), Q^{3} )_{L^{2}}|_{\lambda = 1, \gamma = 0, u = Q} = ( \frac{Q}{2} + x Q_{x}, Q^{3} )_{L^{2}} = \frac{1}{4} \| Q \|_{L^{4}}^{4},
\end{equation}
and
\begin{equation}\label{2.27}
\frac{\partial}{\partial \lambda} ( e^{i \gamma} \lambda^{1/2} u(\lambda x),i Q )_{L^{2}}|_{\lambda = 1, \gamma = 0, u = Q} = ( \frac{Q}{2} + x Q_{x}, i Q )_{L^{2}} = 0.
\end{equation}

Therefore, by the inverse function theorem, if $\lambda_{0} = 1$ and $\gamma_{0} = 0$, there exists $\lambda$ and $\gamma$ satisfying
\begin{equation}\label{2.28}
|\lambda - 1| + |\gamma| \lesssim \| e^{i \gamma_{0}} u(x) - Q(x) \|_{L^{2}},
\end{equation}
such that
\begin{equation}\label{2.29}
( e^{i \gamma} \lambda^{1/2} u(t, \lambda x) - Q(x), Q^{3} )_{L^{2}} = ( e^{i \gamma} \lambda^{1/2} u(t, \lambda x) - Q(x), i Q^{3} )_{L^{2}} = 0.
\end{equation}
The inverse function theorem also guarantees that $\lambda$ and $\gamma$ are unique for all $\lambda, \gamma \in [1 - \delta, 1 + \delta] \times [-\delta, \delta]$ for some $\delta > 0$, up to $2\pi$ periodicity.

For $\lambda$ outside $[1 - \delta, 1 + \delta]$, observe that
\begin{equation}\label{2.29.1}
\| e^{i \gamma} \lambda^{1/2} u(\lambda x) - Q \|_{L^{2}}^{2} = \| u \|_{L^{2}}^{2} + \| Q \|_{L^{2}}^{2} - 2 (e^{i \gamma} \lambda^{1/2} Q(\lambda x), Q)_{L^{2}} - 2 (e^{i \gamma} \lambda^{1/2} [u - Q](\lambda x), Q)_{L^{2}} \gtrsim \delta^{2} - O(\alpha).
\end{equation}
Similarly, for $\gamma$ outside $[-\delta, \delta]$, up to $2\pi$-multiplicity,
\begin{equation}\label{2.29.2}
\| e^{i \gamma} \lambda^{1/2} u(\lambda x) - Q \|_{L^{2}}^{2} = \| u \|_{L^{2}}^{2} + \| Q \|_{L^{2}}^{2} - 2 (e^{i \gamma} \lambda^{1/2} Q(\lambda x), Q)_{L^{2}} - 2 (e^{i \gamma} \lambda^{1/2} [u - Q](\lambda x), Q)_{L^{2}} \gtrsim \delta^{2} - O(\alpha),
\end{equation}
which implies uniqueness for $\alpha > 0$ sufficiently small.

For general $\lambda_{0}$ and $\gamma_{0}$, after rescaling,
\begin{equation}\label{2.30}
|\frac{\lambda}{\lambda_{0}} - 1| + |\gamma - \gamma_{0}| \lesssim \| e^{i \gamma_{0}} \lambda_{0}^{1/2} u(t, \lambda_{0} x) - Q(x) \|_{L^{2}}.
\end{equation}

Finally, using scaling symmetries, the triangle inequality, and $(\ref{2.30})$,
\begin{equation}\label{2.31}
\aligned
\| e^{i \gamma} \lambda^{1/2} u(t, \lambda x) - Q(x) \|_{L^{2}} = \| u(x) - e^{-i \gamma} \lambda^{-1/2} Q(\frac{x}{\lambda}) \|_{L^{2}} \\
\leq \| u(x) - e^{-i \gamma_{0}} \lambda_{0}^{-1/2} Q(\frac{x}{\lambda_{0}}) \|_{L^{2}} + \| e^{-i \gamma_{0}} \lambda_{0}^{-1/2} Q(\frac{x}{\lambda_{0}}) - e^{-i \gamma} \lambda_{0}^{-1/2} Q(\frac{x}{\lambda_{0}}) \|_{L^{2}} \\
+ \| e^{-i \gamma} \lambda_{0}^{-1/2} Q(\frac{x}{\lambda_{0}}) - e^{-i \gamma} \lambda^{-1/2} Q(\frac{x}{\lambda}) \|_{L^{2}}
\lesssim  \| e^{i \gamma_{0}} u(x) - Q(x) \|_{L^{2}}.
\endaligned
\end{equation}
This proves $(\ref{2.14.1})$.
\end{proof}

Therefore, in Theorem $\ref{t3.1}$, there exist functions
\begin{equation}
\lambda : I \rightarrow (0, \infty), \qquad \text{and} \qquad \gamma : I \rightarrow \mathbb{R},
\end{equation}
such that $(\ref{2.13})$ holds for all $t \in [0, \sup(I))$. 


\begin{theorem}\label{t7.2}
Under the hypotheses of Theorem $\ref{t3.1}$, the functions $\lambda(t)$, $\gamma(t)$ are continuous as functions of time on $[0, \sup(I))$. Additionally, $\lambda(t)$ and $\gamma(t)$ are differentiable in time almost everywhere on $[0, \sup(I))$.
\end{theorem}
\begin{proof}
Suppose $J = [a, b]$ is an interval that satisfies
\begin{equation}
\| u \|_{L_{t}^{4} L_{x}^{\infty}(J \times \mathbb{R})} \leq 1,
\end{equation}
and $J \subset [0, \sup(I))$. Suppose without loss of generality that $\lambda(a) = 1$ and $\gamma(a) = 0$. Also, suppose for now that $\| u(a) \|_{\dot{H}^{1}} < \infty$. Strichartz estimates and local well-posedness theory imply that
\begin{equation}\label{7.9.1}
\| u \|_{L_{t}^{\infty} \dot{H}^{1}(J \times \mathbb{R})} \lesssim \| u(a) \|_{\dot{H}^{1}}.
\end{equation}
Since $\lambda(a) = 1$ and $\gamma(a) = 0$,
\begin{equation}
(u(a, x) - Q(x), Q^{3})_{L^{2}} = (u(a, x) - Q(x), i Q^{3})_{L^{2}} = 0.
\end{equation}
Then, by direct calculation and the fact that $Q$ is smooth and rapidly decreasing,
\begin{equation}
\aligned
\frac{d}{dt} (u(t,x) - Q, Q^{3})_{L^{2}} = (i u_{xx}, Q^{3})_{L^{2}} + (i |u|^{4} u, Q^{3})_{L^{2}} \\ = (i u, \partial_{xx}(Q^{3}))_{L^{2}} + i(|u|^{4} u, Q^{3})_{L^{2}} \lesssim \| u \|_{L^{2}} + \| u \|_{L^{\infty}}^{3} \| u \|_{L^{2}}^{2}.
\endaligned
\end{equation}
Therefore, $(\ref{7.9.1})$ implies that $(u(t,x) - Q(x), Q^{3})_{L^{2}}$ is Lipschitz in time on $J$, as is $(u(t,x) - Q(x), i Q^{3})_{L^{2}}$ by an identical calculation. Then by the proof of Theorem $\ref{t2.3}$, $\lambda(t)$ and $\gamma(t)$ are Lipschitz as a function of time for $t$ close to $a$, and by the Lebesgue differentiation theorem, $\lambda$ and $\gamma$ are differentiable almost everywhere for $t$ near $a$.\medskip

Recall from $(\ref{2.13})$ that
\begin{equation}\label{7.9}
\epsilon(t,x) = e^{i \gamma(t)} \lambda(t)^{1/2} u(t, \lambda(t) x) - Q(x).
\end{equation}
By direct computation, for almost every $t$ near $a$,
\begin{equation}\label{7.10}
\aligned
\epsilon_{t} = i \dot{\gamma}(t) (Q + \epsilon) + \frac{\dot{\lambda}(t)}{\lambda(t)} (\frac{Q}{2} + x Q_{x} + \frac{\epsilon}{2} + x \epsilon_{x}) + i \lambda(t)^{-2} (Q_{xx} + \epsilon_{xx}) + i \lambda(t)^{1/2} e^{i \gamma(t)} |u(t, \lambda(t) x)|^{4} u(t, \lambda(x)) \\
= i (\dot{\gamma}(t) + \lambda(t)^{-2}) Q + \frac{\dot{\lambda}(t)}{\lambda(t)} (\frac{Q}{2} + x Q_{x}) + i \lambda(t)^{-2} (\epsilon_{xx} + 5 Q^{4} Re(\epsilon) + i Q^{4} Im(\epsilon) - \epsilon) \\ +  i (\dot{\gamma}(t) + \lambda(t)^{-2}) \epsilon + \frac{\dot{\lambda}(t)}{\lambda(t)} (\frac{\epsilon}{2} + x \epsilon_{x}) + \lambda(t)^{-2} O (|Q|^{3} |\epsilon|^{2} + |\epsilon|^{5}).
\endaligned
\end{equation}
Since $a$ is arbitrary, $\lambda$ and $\gamma$ are differentiable at almost every $t \in [0, \sup(I))$. 

Next, define the monotone function $s : [0, \sup(I)) \rightarrow \mathbb{R}$,
\begin{equation}\label{7.5}
s(t) = \int_{0}^{t} \lambda(\tau)^{-2} d\tau.
\end{equation}
Making a change of variables, $\epsilon_{s} = \lambda^{2} \epsilon_{t}$, by $(\ref{7.10})$,
\begin{equation}\label{7.11}
\aligned
\epsilon_{s} = i (\gamma_{s} + 1) Q + \frac{\lambda_{s}}{\lambda} (\frac{Q}{2} + x Q_{x}) + i (\epsilon_{xx} + 5 Q^{4} Re(\epsilon) + i Q^{4} Im(\epsilon) - \epsilon) \\ +  i (\gamma_{s} + 1) \epsilon + \frac{\lambda_{s}}{\lambda} (\frac{\epsilon}{2} + x \epsilon_{x}) + O (|Q|^{3} |\epsilon|^{2} + |\epsilon|^{2} |u|^{3}).
\endaligned
\end{equation}

Plugging $(\ref{7.11})$ into $(\ref{2.14})$ and integrating by parts,
\begin{equation}\label{7.12}
\aligned
\frac{d}{ds} (\epsilon, Q^{3}) = (\epsilon_{s}, Q^{3}) = 0 = \frac{\lambda_{s}}{4 \lambda} \| Q \|_{L^{4}}^{4} - (Im(\epsilon), \mathcal L_{-} Q^{3})_{L^{2}} + O(|\gamma_{s} + 1| \| \epsilon \|_{L^{2}}) + O(\frac{\lambda_{s}}{\lambda} \| \epsilon \|_{L^{2}}) \\ + O(\| \epsilon \|_{L^{2}}^{2}) + O(\| \epsilon \|_{L^{2}}^{2} \| u \|_{L^{\infty}}^{3}),
\endaligned
\end{equation}
and
\begin{equation}\label{7.13}
\aligned
\frac{d}{ds} (\epsilon, i Q^{3}) = (\epsilon_{s}, i Q^{3}) = 0 = (\gamma_{s} + 1) \| Q \|_{L^{4}}^{4} + (\epsilon, \mathcal L Q^{3})_{L^{2}} + O(|\gamma_{s} + 1| \| \epsilon \|_{L^{2}}) + O(\frac{\lambda_{s}}{\lambda} \| \epsilon \|_{L^{2}}) \\ + O(\| \epsilon \|_{L^{2}}^{2}) + O(\| \epsilon \|_{L^{2}}^{2} \| u \|_{L^{\infty}}^{3}),
\endaligned
\end{equation}
where $\mathcal L_{-}$ and $\mathcal L$ are the linear operators
\begin{equation}\label{2.39}
\mathcal L_{-} f = -f_{xx} - Q^{4} f + f, \qquad \text{and} \qquad \mathcal L f = -f_{xx} - 5 Q^{4} f + f.
\end{equation}
Since $\mathcal L Q^{3} = -8 Q^{3}$ and $(\epsilon, Q^{3})_{L^{2}} = 0$,
\begin{equation}\label{7.15.3}
\frac{\| Q \|_{L^{4}}^{4}}{4} \frac{\lambda_{s}}{\lambda} = (Im(\epsilon), \mathcal L_{-} Q^{3})_{L^{2}} + O(|\gamma_{s} + 1| \| \epsilon \|_{L^{2}}) + O(\frac{\lambda_{s}}{\lambda} \| \epsilon \|_{L^{2}}) + O(\| \epsilon \|_{L^{2}}^{2}) + O(\| \epsilon \|_{L^{2}}^{2} \| u \|_{L^{\infty}}^{3}),
\end{equation}
and
\begin{equation}\label{7.15.4}
\| Q \|_{L^{4}}^{4} (\gamma_{s} + 1) = O(|\gamma_{s} + 1| \| \epsilon \|_{L^{2}}) + O(\frac{\lambda_{s}}{\lambda} \| \epsilon \|_{L^{2}}) + O(\| \epsilon \|_{L^{2}}^{2}) + O(\| \epsilon \|_{L^{2}}^{2} \| u \|_{L^{\infty}}^{3}).
\end{equation}

Doing some algebra, $(\ref{7.15.3})$, $(\ref{7.15.4})$, and the computations proving $(\ref{3.31})$ imply that for any $a \in \mathbb{Z}_{\geq 0}$,
\begin{equation}\label{7.14}
\int_{a}^{a + 1} |\frac{\lambda_{s}}{\lambda}| ds \lesssim \int_{a}^{a + 1} \| \epsilon \|_{L^{2}} ds,
\end{equation}
and
\begin{equation}\label{7.15}
\int_{a}^{a + 1} |\gamma_{s} + 1| ds \lesssim \int_{a}^{a + 1} \| \epsilon \|_{L^{2}} ds.
\end{equation}
Indeed, the computations proving $(\ref{3.31})$ imply that
\begin{equation}\label{7.15.1}
\sup_{s \in [a, a + 1]} \| \epsilon(s) \|_{L^{2}} \lesssim \int_{a}^{a + 1} \| \epsilon(s) \|_{L^{2}} ds,
\end{equation}
so
\begin{equation}\label{7.15.2}
\int_{a}^{a + 1} \| \epsilon \|_{L^{2}}^{2} \| u \|_{L^{\infty}}^{3} ds \lesssim \int_{a}^{a + 1} \| \epsilon(s) \|_{L^{2}}^{2} ds \cdot \int_{a}^{a + 1} \| u \|_{L^{\infty}}^{3} ds.
\end{equation}
Furthermore, Strichartz estimates and the computations proving $(\ref{3.31})$ imply that $\int_{a}^{a + 1} \| u \|_{L^{\infty}}^{4} ds \lesssim 1$, and crucially, the bound is independent of $\| u(a) \|_{\dot{H}^{1}}$.

For a general $u(a) \in L^{2}$, let $u^{N}(a) = P_{\leq N} u(a)$. Taking $N$ sufficiently large so that
\begin{equation}\label{7.15.5}
\| e^{i \gamma(a)} (\lambda(a))^{1/2} u^{N}(a, \lambda(a) x) - Q \|_{L^{2}} \leq 2 \eta_{\ast},
\end{equation}
Theorem $\ref{t2.3}$ implies that there exists $\gamma^{N}(s)$, $\lambda^{N}(s)$ for any $s \in [a, a + 1]$ such that $(\ref{2.13})$ holds. Furthermore, $\lambda^{N}(s)$ and $\gamma^{N}(s)$ satisfy $(\ref{7.15.3})$ and $(\ref{7.15.4})$, and $\gamma^{N}(s)$ and $\lambda^{N}(s)$ converge to $\gamma(s)$ and $\lambda(s)$ uniformly on $[a, a + 1]$, so $\gamma(s)$ and $\lambda(s)$ are continuous as functions of $s$. Furthermore, $\epsilon^{N} \rightarrow \epsilon$ in $L^{2}$ uniformly in $s$, and $u^{N} \rightarrow u$ in $L_{s}^{4} L_{x}^{\infty}$.

Therefore, plugging $\lambda^{N}(s)$, $\gamma^{N}(s)$, $\epsilon^{N}$, $u^{N}$ into $(\ref{7.15.3})$ and $(\ref{7.15.4})$ and doing some algebra implies, by the dominated convergence theorem,
\begin{equation}\label{7.15.6}
\aligned
\frac{\| Q \|_{L^{4}}^{4}}{4} [\ln(\lambda(s)) - \ln(\lambda(a))] = \int_{a}^{s} O((Im(\epsilon), \mathcal L_{-} Q^{3})_{L^{2}}) + O(\| \epsilon \|_{L^{2}}^{2}) + O(\| \epsilon \|_{L^{2}}^{2} \| u \|_{L^{\infty}}^{3}) ds,
\endaligned
\end{equation}
and
\begin{equation}\label{7.15.4=7}
\| Q \|_{L^{4}}^{4} [\gamma(s) - \gamma(a) + (s - a)] = \int_{a}^{s} O(\| \epsilon \|_{L^{2}}^{2}) + O(\| \epsilon \|_{L^{2}}^{2} \| u \|_{L^{\infty}}^{3}) ds.
\end{equation}
Therefore, by the Lebesgue differentiation theorem, $\frac{\lambda_{s}}{\lambda}$ and $\gamma_{s}$ exist for almost every $s \in [a, a + 1]$, and satisfy $(\ref{7.15.3})$ and $(\ref{7.15.4})$.
\end{proof}

Following \cite{merle2001existence}, the decomposition in Theorem $\ref{t2.3}$ gives a positivity result. 
\begin{theorem}\label{t2.4}
If $\epsilon(t,x)$ is a symmetric function, $\epsilon \perp Q^{3}$, $\epsilon \perp i Q^{3}$, $\| \epsilon(t,x) \|_{L^{2}} \ll 1$, and $\| Q + \epsilon \|_{L^{2}} = \| Q \|_{L^{2}}$, then
\begin{equation}\label{2.32}
E(Q + \epsilon) \gtrsim \| \epsilon(t) \|_{H^{1}(\mathbb{R})}^{2} = \int |\epsilon_{x}(t,x)|^{2} dx + \int |\epsilon(t,x)|^{2} dx.
\end{equation}
\end{theorem}
\begin{proof}
Decomposing the energy and integrating by parts, since $Q$ is a real-valued function,
\begin{equation}\label{2.33}
\aligned
E(Q + \epsilon) = \frac{1}{2} \int Q_{x}^{2} dx + Re \int Q_{x}(x) \epsilon_{x}(t,x) dx + \frac{1}{2} \| \epsilon_{x} \|_{L^{2}}^{2} - \frac{1}{6} \int Q(x)^{6} dx - Re \int Q(x)^{5} \epsilon(t,x) dx \\
- \frac{3}{2} \int Q(x)^{4} |\epsilon(t,x)|^{2} dx - Re \int Q(x)^{4} \epsilon(t,x)^{2} dx - \int O(|\epsilon(t,x)|^{3} Q^{3} + |\epsilon(t,x)|^{6}) dx.
\endaligned
\end{equation}
First observe that since $E(Q) = 0$,
\begin{equation}\label{2.34}
\frac{1}{2} \int Q_{x}^{2} dx - \frac{1}{6} \int Q^{6} dx = 0.
\end{equation}
Next, by $(\ref{1.14})$ and integrating by parts,
\begin{equation}\label{2.35}
Re \int Q_{x}(x) \epsilon_{x}(t,x) - Re \int Q(x)^{5} \epsilon(t,x) = -Re \int (Q_{xx}(x) + Q(x)^{5}) \epsilon(t,x) dx = -Re \int Q(x) \epsilon(t,x) dx.\end{equation}
Using the fact that $\| Q + \epsilon \|_{L^{2}} = \| Q \|_{L^{2}}$,
\begin{equation}\label{2.36}
\frac{1}{2} \| Q \|_{L^{2}}^{2} - \frac{1}{2} \| Q + \epsilon \|_{L^{2}}^{2} + \frac{1}{2} \| \epsilon \|_{L^{2}}^{2}  = -(Q, \epsilon)_{L^{2}} =  -Re \int Q(x) \epsilon(t,x) dx = \frac{1}{2} \| \epsilon \|_{L^{2}}^{2}.
\end{equation}
Therefore,
\begin{equation}
\aligned
E(Q + \epsilon) =  \frac{1}{2} \| \epsilon \|_{L^{2}}^{2} + \frac{1}{2} \| \epsilon_{x} \|_{L^{2}}^{2}
- \frac{3}{2} \int Q(x)^{4} |\epsilon(t,x)|^{2} dx - Re \int Q(x)^{4} \epsilon(t,x)^{2} dx \\ - \int O(|\epsilon(t,x)|^{3} Q^{3} + |\epsilon(t,x)|^{6}) dx.
\endaligned
\end{equation}
Decomposing the terms of order $\epsilon^{2}$ into real and imaginary parts,
\begin{equation}\label{2.37}
\aligned
\frac{1}{2} \| \epsilon_{x} \|_{L^{2}}^{2} + \frac{1}{2} \| \epsilon \|_{L^{2}}^{2} - \frac{3}{2} \int Q(x)^{4} |\epsilon(t,x)|^{2} dx - Re \int Q(x)^{4} \epsilon(t,x)^{2} dx \\
= \frac{1}{2} \int Re(\epsilon)_{x}^{2} dx + \frac{1}{2} \int Re(\epsilon)^{2} dx - \frac{5}{2} \int Q(x)^{4} Re(\epsilon)^{2} dx \\
+ \frac{1}{2} \int Im(\epsilon)_{x}^{2} dx + \frac{1}{2} \int Im(\epsilon)^{2} dx - \frac{1}{2} \int Q(x)^{4} Im(\epsilon)^{2} dx.
\endaligned
\end{equation}

Recalling $(\ref{2.39})$,
\begin{equation}\label{2.38}
\frac{1}{2} \int Re(\epsilon)_{x}^{2} dx + \frac{1}{2} \int Re(\epsilon)^{2} dx - \frac{5}{2} \int Q(x)^{4} Re(\epsilon)^{2} dx = \frac{1}{2} (\mathcal L Re(\epsilon), Re(\epsilon))_{L^{2}}.
\end{equation}

It is well-known, see for example \cite{merle2001existence}, that $\mathcal L$ has one negative eigenvector, $\mathcal L(Q^{3}) = -8 Q^{3}$, and one zero eigenvector, $\mathcal L(Q_{x}) = 0$. Since $Re(\epsilon) \perp Q^{3}$, and $Re(\epsilon)$ symmetric guarantees that $Re(\epsilon) \perp Q_{x}$,
\begin{equation}\label{2.40}
 \frac{1}{2} \int Re(\epsilon)_{x}^{2} dx + \frac{1}{2} \int Re(\epsilon)^{2} dx - \frac{5}{2} \int Q(x)^{4} Re(\epsilon)^{2} dx \geq \frac{1}{2} \int Re(\epsilon)^{2} dx.
 \end{equation}
Next, doing some algebra,
\begin{equation}\label{2.41}
\aligned
\frac{1}{2} \int Re(\epsilon)_{x}^{2} dx = \frac{1}{2} (\mathcal L Re(\epsilon), Re(\epsilon)) - \frac{1}{2} \int Re(\epsilon)^{2} dx + \frac{5}{2} \int Q(x)^{4} Re(\epsilon)^{2} dx
\leq C (\mathcal L Re(\epsilon), Re(\epsilon)).
\endaligned
\end{equation}
By similar calculations, since $Im(\epsilon) \perp Q^{3}$ and $Im(\epsilon) \perp Q_{x}$,
\begin{equation}\label{2.41.1}
\aligned
\frac{1}{2} \int Im(\epsilon)_{x}^{2} dx + \frac{1}{2} \int Im(\epsilon)^{2} dx - \frac{1}{2} \int Q(x)^{4} Im(\epsilon)^{2} dx = \frac{1}{2} (\mathcal L Im(\epsilon), Im(\epsilon)) + 2 \int Q(x)^{4} Im(\epsilon)^{2} \\ \geq \frac{1}{2} (\mathcal L Im(\epsilon), Im(\epsilon))
\geq \frac{1}{2} \int Im(\epsilon)^{2} dx + \frac{1}{2C} \int Im(\epsilon)_{x}^{2} dx.
\endaligned
\end{equation}

Finally, by the Sobolev embedding theorem and $\| \epsilon \|_{L^{2}} \ll 1$,
\begin{equation}\label{2.42}
\int |\epsilon|^{6} dx \lesssim \| \epsilon \|_{\dot{H}^{1}}^{2} \| \epsilon \|_{L^{2}}^{4} \ll \| \epsilon \|_{\dot{H}^{1}}^{2},
\end{equation}
and
\begin{equation}\label{2.43}
\int Q(x)^{3} |\epsilon(t,x)|^{3} dx \lesssim \| \epsilon \|_{L^{2}}^{3/2} \| \epsilon \|_{L^{6}}^{3/2} \lesssim \| \epsilon \|_{L^{2}}^{5/2} \| \epsilon \|_{\dot{H}^{1}}^{1/2} \lesssim \| \epsilon \|_{L^{2}}^{5/2} + \| \epsilon \|_{L^{2}}^{5/2} \| \epsilon \|_{\dot{H}^{1}}^{2} \ll \| \epsilon \|_{L^{2}}^{2} +  \| \epsilon \|_{\dot{H}^{1}}^{2},
\end{equation}
which completes the proof of Theorem $\ref{t2.4}$.
\end{proof}

\section{A long time Strichartz estimate}
Having shown that it is enough to consider solutions to $(\ref{1.1})$ that are close to the family of solitons, and that there is a good decomposition of solutions that are close to the family of solitons, the next task is to obtain a good frequency localized Morawetz estimate. The proof of the frequency localized Morawetz estimate will occupy sections four, five, and six.

The proof of scattering in \cite{dodson2015global} for $(\ref{1.1})$ when $\| u_{0} \|_{L^{2}} < \| Q \|_{L^{2}}$ utilized a frequency localized Morawetz estimate. There, the Morawetz estimate was used to show that $E(P_{n} u(t_{n})) \rightarrow 0$ along a subsequence, where $P_{n}$ is a Fourier truncation operator that converges to the identity in the strong $L^{2}$-operator topology. Then the Gagliardo--Nirenberg inequality, $(\ref{1.11})$, and the stability of the zero solution to $(\ref{1.1})$ implies that $u \equiv 0$. In the case that $\| u_{0} \|_{L^{2}} = \| Q \|_{L^{2}}$, \cite{fan20182} and \cite{dodson20202} proved that $E(P_{n} u(t_{n})) \rightarrow 0$ along a subsequence, so the almost periodicity of $u$ implies that $u(t_{n})$ converges to a rescaled version of $Q$.

In fact, \cite{fan20182} and \cite{dodson20202} proved more, that $E(Pu(t)) \rightarrow 0$ in an averaged sense on an interval $[0, T] \subset I$. The operator $P$ is fixed on a fixed time interval, but $P$ converges to the identity in the strong $L^{2}$-operator topology as $T \rightarrow \sup(I)$. The proof of Theorem $\ref{t3.1}$ will argue that if $E(P u(t))$ goes to zero in a time averaged sense, then $u$ must be equal to the soliton, if the solution is global. If the solution blows up in finite time, then $u$ must equal a pseudoconformal transformation of the soliton. 

An essential ingredient in this proof is an improved version of the long time Strichartz estimates in \cite{dodson2016global}. The proof will make use of the bilinear estimates of \cite{planchon2009bilinear}, which were also used in the two dimensional problem, \cite{dodson2016global2}.

Eventually, the proof of Theorem $\ref{t3.1}$ will make use of long time Strichartz estimates on an interval $J = [a, b]$, for
\begin{equation}\label{6.0}
1 \leq \lambda(t) \leq T^{1/100},
\end{equation}
where $T = s(b) - s(a)$ and $s(t) : [0, \sup(I)) \rightarrow [0, \infty)$ is the function given by $(\ref{7.5})$. However, to avoid obscuring the main idea, it will be convenient to consider the case when $\lambda(t) = 1$ first, since the generalization to the case $(\ref{6.0})$ is fairly straightforward.\medskip

Suppose without loss of generality that $a = 0$ and $b = T$. Choose
\begin{equation}\label{6.1}
0 < \eta_{1} \ll \eta_{0} \ll 1,
\end{equation}
to be small constants, suppose
\begin{equation}\label{6.1.1}
\| \epsilon(t,x) \|_{L^{2}} \leq \eta_{0},
\end{equation}
for all $t \in J$, and choose $\eta_{1} \ll \eta_{0}$ sufficiently small so that
\begin{equation}\label{6.2}
 \int_{|\xi| \geq \eta_{1}^{-1/2}} |\hat{Q}(\xi)|^{2} d\xi \leq \eta_{0}^{2},
\end{equation}
and therefore,
\begin{equation}\label{6.2.1}
\sup_{t \in J} \int_{|\xi| \geq \eta_{1}^{-1/2}} |\hat{u}(t,\xi)|^{2} d\xi \leq 4 \eta_{0}^{2}.
\end{equation}
Then rescale from $\lambda(t) = 1$ to $\lambda(t) = \frac{1}{\eta_{1}}$, and $[0, T] \mapsto [0, \eta_{1}^{-2} T]$.

When $i \in \mathbb{Z}$, $i > 0$, let $P_{i}$ denote the standard Littlewood-Paley projection operator. When $i = 0$, let $P_{i}$ denote the projection operator $P_{\leq 0}$, and when $i < 0$, let $P_{i}$ denote the zero operator.
\begin{definition}\label{d6.1}
Suppose $\eta_{1}^{-2} T = 2^{3k}$ for some $k \in \mathbb{Z}_{\geq 0}$. Then define the norm
\begin{equation}\label{6.3}
\| u \|_{X([0, \eta_{1}^{-2} T] \times \mathbb{R})}^{2} = \sup_{0 \leq i \leq k} \sup_{1 \leq a < 2^{3k - 3i}} \| P_{i} u \|_{U_{\Delta}^{2}([(a - 1) 2^{3i}, a 2^{3i}] \times \mathbb{R})}^{2} + 2^{i} \| (P_{\geq i} u)(P_{\leq i - 3} u) \|_{L_{t,x}^{2}([(a - 1) 2^{3i}, a 2^{3i}] \times \mathbb{R})}^{2}.
\end{equation}
Also, for any $0 \leq j \leq k$, let
\begin{equation}\label{6.4}
\| u \|_{X_{j}([0, \eta_{1}^{-2} T] \times \mathbb{R})}^{2} = \sup_{0 \leq i \leq j} \sup_{1 \leq a < 2^{3k - 3i}} \| P_{i} u \|_{U_{\Delta}^{2}([(a - 1) 2^{3i}, a 2^{3i}] \times \mathbb{R})}^{2} + 2^{i} \| (P_{\geq i} u)(P_{\leq i - 3} u) \|_{L_{t,x}^{2}([(a - 1) 2^{3i}, a 2^{3i}] \times \mathbb{R})}^{2}.
\end{equation}
\end{definition}

See \cite{koch2007priori} for a definition of the $U_{\Delta}^{p}$ and $V_{\Delta}^{p}$ norms, and the references therein. See also \cite{dodson2016global2}, \cite{dodson2016global}, and \cite{dodson2019defocusing}.

\begin{theorem}\label{t6.2}
The long time Strichartz estimate,
\begin{equation}\label{6.5}
\| u \|_{X([0, \eta_{1}^{-2} T] \times \mathbb{R})} \lesssim 1,
\end{equation}
holds with implicit constant independent of $T$.
\end{theorem}
\begin{proof}
This estimate is proved by induction on $j$. Local well-posedness arguments combined with the fact that $\lambda(t) = \eta_{1}^{-1}$ for any $t \in [0, \eta_{1}^{-2} T]$ imply that
\begin{equation}\label{6.6}
\| u \|_{U_{\Delta}^{2}([a, a + 1] \times \mathbb{R})} \lesssim 1,
\end{equation}
and when $i = 0$,
\begin{equation}\label{6.7}
(P_{\geq i} u)(P_{\leq i - 3} u) = 0.
\end{equation}
Therefore,
\begin{equation}\label{6.8}
\| u \|_{X_{0}([0, \eta_{1}^{-2} T] \times \mathbb{R})} \lesssim 1.
\end{equation}
This is the base case.\medskip

\begin{remark}
The implicit constant in $(\ref{6.8})$ does not depend on $T$ or $\eta_{1}$.
\end{remark}

To prove the inductive step, recall that by Duhamel's principle that if $J = [(a - 1) 2^{3k - 3i}, a 2^{3k - 3i}]$, then for any $t_{0} \in J$,
\begin{equation}\label{6.9}
u(t) = e^{i(t - t_{0}) \Delta} u(t_{0}) + i \int_{t_{0}}^{t} e^{i(t - \tau) \Delta} (|u|^{4} u) d\tau,
\end{equation}
and
\begin{equation}\label{6.10}
\| P_{\geq i} u \|_{U_{\Delta}^{2}(J \times \mathbb{R})} \lesssim \| P_{\geq i}(u(t_{0})) \|_{L^{2}} + \| \int_{t_{0}}^{t} e^{i(t - \tau) \Delta} P_{\geq i}(|u|^{4} u) d\tau \|_{U_{\Delta}^{2}(J \times \mathbb{R})}.
\end{equation}

By $(\ref{6.2})$ and the fact that $\lambda(t) = \frac{1}{\eta_{1}}$ for all $t \in [0, \eta_{1}^{-2} T]$, if $i > 0$,
\begin{equation}\label{6.11}
\sup_{t_{0} \in [0, \eta_{1}^{-2} T]} \| P_{\geq i} u(t_{0}) \|_{L^{2}} \lesssim \eta_{0}.
\end{equation}
Next, choose $v \in V_{\Delta}^{2}(J \times \mathbb{R})$ such that $\| v \|_{V_{\Delta}^{2}(J \times \mathbb{R})} = 1$ and $\hat{v}(t, \xi)$ is supported on the Fourier support of $P_{i}$. It is a well-known fact that
\begin{equation}\label{6.12}
 \| \int_{t_{0}}^{t} e^{i(t - \tau) \Delta} P_{\geq i}(|u|^{4} u) d\tau \|_{U_{\Delta}^{2}(J \times \mathbb{R})} \lesssim \sup_{v} \| v P_{\geq i}(|u|^{4} u) \|_{L_{t,x}^{1}},
 \end{equation}
 where $\sup_{v}$ is the supremum over all such $v$ supported on $P_{i}$ satisfying $\| v \|_{V_{\Delta}^{2}(J \times \mathbb{R})} = 1$. See \cite{hadac2009well} for a proof.

By H{\"o}lder's inequality,
\begin{equation}\label{6.13}
\aligned
\| v(u_{\geq i - 3})^{2}(u_{\leq i - 3})^{3} \|_{L_{t,x}^{1}} + \|  v(u_{\geq i - 3})^{3}(u_{\leq i - 3})^{2} \|_{L_{t,x}^{1}} + \| v(u_{\geq i - 3})^{4}(u_{\leq i - 3}) \|_{L_{t,x}^{1}} + \| v (u_{\geq i - 3})^{5} \|_{L_{t,x}^{1}} \\
\lesssim \| v \|_{L_{t}^{\infty} L_{x}^{2}} \| u_{\geq i - 3} \|_{L_{t}^{5} L_{x}^{10}}^{5} + \| v \|_{L_{t}^{4} L_{x}^{\infty}} \| u_{\geq i - 3} \|_{L_{t}^{16/3} L_{x}^{8}}^{4} \| u_{\leq i - 3} \|_{L_{t}^{\infty} L_{x}^{2}} \\ + \| v \|_{L_{t,x}^{6}} \| (u_{\geq i - 3})(u_{\leq i - 6}) \|_{L_{t,x}^{2}} \| u_{\leq i - 6} \|_{L_{t,x}^{\infty}} \| u_{\geq i - 3} \|_{L_{t,x}^{6}}^{2} + \| v \|_{L_{t,x}^{6}} \| u_{\geq i - 3} \|_{L_{t,x}^{6}}^{3} \| u_{\geq i - 6} \|_{L_{t,x}^{6}}^{2} \\
+ \| v \|_{L_{t,x}^{6}} \| (u_{\geq i - 3})(u_{\leq i - 6}) \|_{L_{t,x}^{2}}^{3/2} \| u_{\leq i - 6} \|_{L_{t,x}^{\infty}}^{3/2} \| u_{\geq i - 3} \|_{L_{t,x}^{6}}^{1/2} + \| v \|_{L_{t,x}^{6}} \| u_{\geq i - 3} \|_{L_{t,x}^{6}}^{2} \| u_{\geq i - 6} \|_{L_{t,x}^{6}}^{3}.
\endaligned
\end{equation}
Since $V_{\Delta}^{2} \subset U_{\Delta}^{p}$ for any $p > 2$, again see \cite{hadac2009well},
\begin{equation}\label{6.14}
\| v \|_{L_{t}^{\infty} L_{x}^{2}} + \| v \|_{L_{t,x}^{6}} + \| v \|_{L_{t}^{4} L_{x}^{\infty}} \lesssim \| v \|_{V_{\Delta}^{2}} \lesssim 1.
\end{equation}

Next, when $i > 4$, since $U_{\Delta}^{2} \subset U_{\Delta}^{4}$, $(\ref{6.2})$ and $(\ref{6.11})$ imply
\begin{equation}\label{6.15}
\| u_{\geq i - 3} \|_{L_{t}^{5} L_{x}^{10}}^{5} \lesssim \| u_{\geq i - 3} \|_{L_{t}^{4} L_{x}^{\infty}}^{4} \| u_{\geq i - 3} \|_{L_{t}^{\infty} L_{x}^{2}} \lesssim \eta_{0} \| u_{\geq i - 3} \|_{L_{t}^{4} L_{x}^{\infty}}^{4} \lesssim \eta_{0} \| u \|_{X_{i - 3}([0, T] \times \mathbb{R})}^{4}.
\end{equation}
When $i \leq 4$, the fact that for any $a \in \mathbb{Z}$,
\begin{equation}\label{6.16}
\| u \|_{L_{t}^{4} L_{x}^{\infty}([a, a + 1] \times \mathbb{R})} \lesssim 1,
\end{equation}
the fact that the Fourier inversion formula and H{\"o}lder's inequality imply
\begin{equation}\label{6.16.1}
\|\int_{|\xi| \leq \eta_{1}^{1/2}} e^{ix \cdot \xi} \hat{u}(t, \xi) d\xi \|_{L^{\infty}} \lesssim \eta_{1}^{1/4} \| u(t) \|_{L^{2}},
\end{equation}
and the fact that $(\ref{6.2})$ implies, after rescaling $\lambda(t) = 1 \mapsto \lambda(t) = \frac{1}{\eta_{1}}$, 
\begin{equation}\label{6.16.1.1}
\int_{|\xi| \geq \eta_{1}^{1/2}} |\hat{u}(t,\xi)|^{2} d\xi)^{1/2} \lesssim \eta_{0},
\end{equation}
combine to imply that
\begin{equation}\label{6.16.2}
\| u_{\geq i - 3} \|_{L_{t}^{5} L_{x}^{10}}^{5} \lesssim \eta_{0}.
\end{equation}
Similar calculations can be made for the terms
\begin{equation}\label{6.17}
\| u_{\geq i - 3} \|_{L_{t}^{16/3} L_{x}^{8}}^{4} \| u_{\leq i - 3} \|_{L_{t}^{\infty} L_{x}^{2}} + \| u_{\geq i - 3} \|_{L_{t,x}^{6}}^{3} \| u_{\geq i - 6} \|_{L_{t,x}^{6}}^{2} +  \| u_{\geq i - 3} \|_{L_{t,x}^{6}}^{2} \| u_{\geq i - 6} \|_{L_{t,x}^{6}}^{3}.
\end{equation}
Therefore,
\begin{equation}\label{6.18}
\aligned
\| u_{\geq i - 3} \|_{L_{t}^{5} L_{x}^{10}}^{5} + \| u_{\geq i - 3} \|_{L_{t}^{16/3} L_{x}^{8}}^{4} \| u_{\leq i - 3} \|_{L_{t}^{\infty} L_{x}^{2}} + \| u_{\geq i - 3} \|_{L_{t,x}^{6}}^{3} \| u_{\geq i - 6} \|_{L_{t,x}^{6}}^{2} +  \| u_{\geq i - 3} \|_{L_{t,x}^{6}}^{2} \| u_{\geq i - 6} \|_{L_{t,x}^{6}}^{3} \\ \lesssim \eta_{0} \| u \|_{X_{i - 3}([0, T] \times \mathbb{R})}^{4} + \eta_{0} \| u \|_{X_{i - 3}([0, T] \times \mathbb{R})}^{3} + \eta_{0}.
\endaligned
\end{equation}

Next, by Definition $\ref{d6.1}$,
\begin{equation}\label{6.19}
 \| (u_{\geq i - 3})(u_{\leq i - 6}) \|_{L_{t,x}^{2}} \| u_{\leq i - 6} \|_{L_{t,x}^{\infty}} \| u_{\geq i - 3} \|_{L_{t,x}^{6}}^{2} \lesssim 2^{-i/2} \| u \|_{X_{i - 3}([0, T] \times \mathbb{R})}^{3} \| u_{\leq i - 6} \|_{L_{t,x}^{\infty}}.
 \end{equation}
 By $(\ref{6.16.1})$, $(\ref{6.16.1.1})$, and the Sobolev embedding theorem,
 \begin{equation}\label{6.20}
 2^{-i/2} \| u_{\leq i - 6} \|_{L_{t,x}^{\infty}} \lesssim \eta_{0},
 \end{equation}
 so
 \begin{equation}\label{6.21}
2^{-i/2} \| u \|_{X_{i - 3}([0, T] \times \mathbb{R})}^{3} \| u_{\leq i - 6} \|_{L_{t,x}^{\infty}} \lesssim \eta_{0} \| u \|_{X_{i - 3}([0, T] \times \mathbb{R})}^{3}.
\end{equation}

Making a similar calculation,
\begin{equation}\label{6.22}
\aligned
\| (u_{\geq i - 3})(u_{\leq i - 6}) \|_{L_{t,x}^{2}}^{3/2} \| u_{\leq i - 6} \|_{L_{t,x}^{\infty}}^{3/2} \| u_{\geq i - 3} \|_{L_{t,x}^{6}}^{1/2} \lesssim \eta_{0}^{3/2} \| u \|_{X_{i - 3}([0, T] \times \mathbb{R})}^{2}.
\endaligned
\end{equation}
Since
\begin{equation}\label{6.23}
P_{\geq i}(|u_{\leq i - 3}|^{4} u_{\leq i - 3}) = 0,
\end{equation}
it only remains to compute, using Definition $\ref{d6.1}$, $(\ref{6.18})$, and $(\ref{6.20})$,
\begin{equation}\label{6.24}
\aligned
\| v ((P_{\geq i - 3} u)(P_{\leq i - 3} u)^{4}) \|_{L_{t,x}^{1}} \lesssim \| (P_{\geq i - 3} u)(P_{\leq i - 6} u) \|_{L_{t,x}^{2}} \| P_{\leq i - 6} u \|_{L_{t,x}^{\infty}} \| v (P_{\leq i - 3} u)^{2} \|_{L_{t,x}^{2}} \\
+ \| P_{\geq i - 3} u \|_{L_{t,x}^{6}} \| P_{\geq i - 6} u \|_{L_{t,x}^{6}}^{2} \| v (P_{\leq i - 3} u)^{2} \|_{L_{t,x}^{2}} \\ \lesssim \eta_{0} \| u \|_{X_{i - 3}([0, T] \times \mathbb{R})}^{2} \| v (P_{\leq i - 3} u)^{2} \|_{L_{t,x}^{2}} + \eta_{0} \| u \|_{X_{i - 3}([0, T] \times \mathbb{R})}^{4} + \eta_{0}.
\endaligned
\end{equation}
By the Sobolev embedding theorem,
\begin{equation}\label{6.25}
\| P_{\leq i - 3} u \|_{L_{t,x}^{18}} \lesssim \sum_{0 \leq j \leq i - 3} \| P_{j} u \|_{L_{t,x}^{18}} \lesssim \sum_{0 \leq j \leq i - 3} 2^{\frac{(3i - 3j)}{18}} 2^{j/3} \| u \|_{X_{i - 3}([0, T] \times \mathbb{R})} \lesssim 2^{i/3} \| u \|_{X_{i - 3}([0, T] \times \mathbb{R})}.
\end{equation}
Also, by $V_{\Delta}^{2} \subset U_{\Delta}^{9/4}$ and the Sobolev embedding theorem,
\begin{equation}\label{6.26}
\| v (P_{\leq i - 3} u) \|_{L_{t,x}^{9/4}} \lesssim 2^{i/9} \| v \|_{V_{\Delta}^{12/5}(J \times \mathbb{R})} \cdot \sup_{v_{0}} \| (e^{it \Delta} v_{0})(u_{\leq i - 3}) \|_{L_{t,x}^{2}}^{8/9},
\end{equation}
where $\sup_{v_{0}}$ is over all $\| v_{0} \|_{L^{2}} = 1$ supported in Fourier space on the support of $P_{i}$. Therefore, we have finally proved,
\begin{equation}\label{6.27}
\aligned
\| P_{\geq i} u \|_{U_{\Delta}^{2}(J \times \mathbb{R})} \lesssim \eta_{0} + \eta_{0} \| u \|_{X_{i - 3}([0, T] \times \mathbb{R})}^{2} + \eta_{0} \| u \|_{X_{i - 3}([0, T] \times \mathbb{R})}^{4} \\ + \eta_{0} \| u \|_{X_{i - 3}([0, T] \times \mathbb{R})}^{3} \cdot 2^{4i/9} \sup_{v_{0}} \| (e^{it \Delta} v_{0})(u_{\leq i - 3}) \|_{L_{t,x}^{2}}^{8/9}.
\endaligned
\end{equation}

To complete the proof of Theorem $\ref{t6.2}$, it only remains to prove
\begin{equation}\label{6.28}
2^{i/2} \sup_{v_{0}} \| (e^{it \Delta} v_{0})(u_{\leq i - 3}) \|_{L_{t,x}^{2}} \lesssim 1 + \| u \|_{X_{i - 3}([0, T] \times \mathbb{R})}.
\end{equation}
Indeed, assuming that $(\ref{6.28})$ is true, $(\ref{6.27})$ becomes
\begin{equation}\label{6.29}
\| P_{\geq i} u \|_{U_{\Delta}^{2}(J \times \mathbb{R})} \lesssim \eta_{0} + \eta_{0} \| u \|_{X_{i - 3}([0, T] \times \mathbb{R})}^{2} + \eta_{0} \| u \|_{X_{i - 3}([0, T] \times \mathbb{R})}^{4}.
\end{equation}
Equation $(\ref{6.28})$ would also imply
\begin{equation}\label{6.30}
\aligned
2^{i/2} \| (P_{\geq i} u)(P_{\leq i - 3} u) \|_{L_{t,x}^{2}(J \times \mathbb{R})} \lesssim \| P_{\geq i} u \|_{U_{\Delta}^{2}(J \times \mathbb{R})}(1 +  \| u \|_{X_{i - 3}([0, T] \times \mathbb{R})}) \\ \lesssim \eta_{0} + \eta_{0} \| u \|_{X_{i - 3}([0, T] \times \mathbb{R})} + \eta_{0} \| u \|_{X_{i - 3}([0, T] \times \mathbb{R})}^{5}.
\endaligned
\end{equation}
Then taking a supremum over $0 \leq i \leq j$,
\begin{equation}\label{6.31}
\| u \|_{X_{j}([0, T] \times \mathbb{R})} \lesssim 1 + \eta_{0} \| u \|_{X_{j - 3}([0, T] \times \mathbb{R})} + \eta_{0} \| u \|_{X_{j - 3}([0, T] \times \mathbb{R})}^{5},
\end{equation}
which by induction on $j$, starting from the base case $(\ref{6.8})$, proves Theorem $\ref{t6.2}$.\medskip

The bilinear estimate $(\ref{6.28})$ is proved using the interaction Morawetz estimate (see \cite{planchon2009bilinear} and \cite{dodson2016global2}). To simplify notation, let
\begin{equation}\label{6.32}
v(t,x) = e^{it \Delta} v_{0},
\end{equation}
where $\| v_{0} \|_{L^{2}} = 1$ and $\hat{v}_{0}$ is supported on the Fourier support of $P_{j}$ for some $j \geq i$. Then take the Morawetz potential,
\begin{equation}\label{6.33}
M(t) = \int |v(t, y)|^{2} \frac{(x - y)}{|x - y|} Im[\bar{u}_{\leq i - 3} \partial_{x} u_{\leq i - 3}] dx dy + \int |u_{\leq i - 3}|^{2} \frac{(x - y)}{|x - y|} Im[\bar{v} v_{x}] dx dy.
\end{equation}
Let $F(u) = |u|^{4} u$. Then $u_{\leq i - 3}$ solves the equation
\begin{equation}\label{6.34}
i \partial_{t} u_{\leq i - 3} + \Delta u_{\leq i - 3} + F(u_{\leq i - 3}) = F(u_{\leq i - 3}) - P_{\leq i - 3} F(u) = -\mathcal N_{i - 3}.
\end{equation}
Following  \cite{planchon2009bilinear},
\begin{equation}\label{6.35}
\aligned
\frac{d}{dt} M(t) = 8 \int |\partial_{x}(\overline{v(t,x)} u_{\leq i - 3})(t,x)|^{2} dx - \frac{8}{3} \int |v(t,x)|^{2} |u_{\leq i - 3}(t,x)|^{6} dx \\
+ \int |v(t,y)|^{2} \frac{(x - y)}{|x - y|} Re[\bar{u}_{\leq i - 3} \partial_{x} \mathcal N_{i - 3}](t,x) dx -  \int |v(t,y)|^{2} \frac{(x - y)}{|x - y|} Re[\bar{\mathcal N}_{i - 3} \partial_{x} u_{\leq i - 3}](t,x) dx \\
+ 2 \int Im[\bar{u}_{\leq i - 3} \mathcal N_{i - 3}](t, y) \frac{(x - y)}{|x - y|} Im[\bar{v} \partial_{x} v](t,x) dx dy.
\endaligned
\end{equation}
Then by the fundamental theorem of calculus, Bernstein's inequality, the Fourier support of $\bar{v} u_{\leq i - 3}$, $\| v_{0} \|_{L^{2}} = 1$, and the fact that $\| u \|_{L^{2}} = \| Q \|_{L^{2}}$,
\begin{equation}\label{6.36}
\aligned
2^{2j} \| \bar{v} u_{\leq i - 3} \|_{L_{t,x}^{2}(J \times \mathbb{R})}^{2} \lesssim 2^{j} + \| |v|^{2} |u_{\leq i - 3}|^{6} \|_{L_{t,x}^{1}} -  \int |v(t,y)|^{2} \frac{(x - y)}{|x - y|} Re[\bar{\mathcal N}_{i - 3} \partial_{x} u_{\leq i - 3}](t,x) dx \\ + \int |v(t,y)|^{2} \frac{(x - y)}{|x - y|} Re[\bar{u}_{\leq i - 3} \partial_{x} \mathcal N_{i - 3}](t,x) dx
+ 2 \int Im[\bar{u}_{\leq i - 3} \mathcal N_{i - 3}](t, y) \frac{(x - y)}{|x - y|} Im[\bar{v} \partial_{x} v](t,x) dx dy.
\endaligned
\end{equation}
Also note that
\begin{equation}\label{6.37}
\| \bar{v} u_{\leq i - 3} \|_{L_{t,x}^{2}}^{2} = \| \bar{v} v \bar{u}_{\leq i - 3} u_{\leq i - 3} \|_{L_{t,x}^{1}} = \| v u_{\leq i - 3} \|_{L_{t,x}^{2}}^{2},
\end{equation}
so it is not too important to pay attention to complex conjugates in the proceeding calculations.

First, by $(\ref{6.20})$,
\begin{equation}\label{6.38}
\| |v|^{2} |u_{\leq i - 3}|^{6} \|_{L_{t,x}^{1}(J \times \mathbb{R})} \lesssim \| v u_{\leq i - 3} \|_{L_{t,x}^{2}}^{2} \| u_{\leq i - 3} \|_{L_{t,x}^{\infty}}^{4} \lesssim \eta_{0}^{4} 2^{2i} \| v u_{\leq i - 3} \|_{L_{t,x}^{2}}^{2}.
\end{equation}
Now consider the term,
\begin{equation}\label{6.39}
\mathcal N_{i - 3} = P_{\leq i - 3} F(u) - F(u_{\leq i - 3}).
\end{equation}
Since by Fourier support arguments
\begin{equation}\label{6.40}
P_{\leq i - 3} F(u_{\leq i - 6}) - F(u_{\leq i - 6}) = 0,
\end{equation}
\begin{equation}\label{6.41}
\aligned
\mathcal N_{i - 3} = P_{\leq i - 3}(3 |u_{\leq i - 6}|^{4} u_{\geq i - 6} + 2|u_{\leq i - 6}|^{2} (u_{\leq i - 6})^{2} \bar{u}_{\geq i - 6}) \\ - (3 |u_{\leq i - 6}|^{4} u_{i - 6 \leq \cdot \leq i - 3} + 2|u_{\leq i - 6}|^{2} (u_{\leq i - 6})^{2} \bar{u}_{i - 6 \leq \cdot \leq i - 3}) \\
+ P_{\leq i - 3} O((u_{\geq i - 6})^{2} u^{3}) + O((u_{i - 6 \leq \cdot \leq i - 3})^{2} u^{3}) = \mathcal N_{i - 3}^{(1)} + \mathcal N_{i - 3}^{(2)}.
\endaligned
\end{equation}
Following $(\ref{6.13})$--$(\ref{6.22})$,
\begin{equation}\label{6.42}
\aligned
\| |\mathcal N_{i - 3}^{(2)}| |u_{\leq i - 3}| \|_{L_{t,x}^{1}} \lesssim \| |u_{\geq i - 6}|^{2} |u_{\leq i - 9}|^{4} \|_{L_{t,x}^{1}} + \| |u_{\geq i - 6}|^{2} |u_{\geq i - 9}|^{4} \|_{L_{t,x}^{1}} \\ \lesssim \| (u_{\geq i - 6})(u_{\leq i - 9}) \|_{L_{t,x}^{2}}^{2} \| u_{\leq i - 9} \|_{L_{t,x}^{\infty}}^{2} + \| u_{\geq i - 6} \|_{L_{t,x}^{6}}^{2} \| u_{\geq i - 9} \|_{L_{t,x}^{6}}^{4} \lesssim \eta_{0} (1 + \| u \|_{X_{i - 3}([0, T] \times \mathbb{R})}^{6}).
\endaligned
\end{equation}
Therefore, since $\| v_{0} \|_{L^{2}} = 1$,
\begin{equation}\label{6.43}
\aligned
-  \int \int \int |v(t,y)|^{2} \frac{(x - y)}{|x - y|} Re[\bar{\mathcal N}_{i - 3}^{(2)} \partial_{x} u_{\leq i - 3}](t,x) dx dy dt \\ + \int \int \int |v(t,y)|^{2} \frac{(x - y)}{|x - y|} Re[\bar{u}_{\leq i - 3} \partial_{x} \mathcal N_{i - 3}^{(2)}](t,x) dx dy dt \\
+ 2 \int \int \int Im[\bar{u}_{\leq i - 3} \mathcal N_{i - 3}^{(2)}](t, y) \frac{(x - y)}{|x - y|} Im[\bar{v} \partial_{x} v](t,x) dx dy dt \\ \lesssim \eta_{0} 2^{j} (1 + \| u \|_{X_{i - 3}([0, T] \times \mathbb{R})}^{6}).
\endaligned
\end{equation}

Next, observe that
\begin{equation}\label{6.44}
3 P_{\leq i - 3}(|u_{\leq i - 6}|^{4} u_{\geq i - 6}) - 3(|u_{\leq i - 6}|^{4} u_{i - 6 \leq \cdot \leq i - 3}) = 3 P_{> i - 3}(|u_{\leq i - 6}|^{4} u_{i - 6 \leq \cdot \leq i - 3}) + 3 P_{\leq i - 3}(|u_{\leq i - 6}|^{4} u_{> i - 3}).
\end{equation}
Again following $(\ref{6.13})$--$(\ref{6.22})$,
\begin{equation}\label{6.45}
\aligned
\| P_{\leq i - 3}(|u_{\leq i - 6}|^{4} u_{> i - 3}) (u_{i - 6 \leq \cdot \leq i - 3}) \|_{L_{t,x}^{1}} \\ + \| P_{> i - 3}(|u_{\leq i - 6}|^{4} u_{ i - 6 \leq \cdot \leq i - 3}) (u_{i - 6 \leq \cdot \leq i - 3}) \|_{L_{t,x}^{1}} \lesssim \eta_{0}(1 + \| u \|_{X_{i - 3}([0, T] \times \mathbb{R})}^{6}).
\endaligned
\end{equation}
Finally, observe that the Fourier support of
\begin{equation}\label{6.46}
3 P_{> i - 3}(|u_{\leq i - 6}|^{4} u_{i - 6 \leq \cdot \leq i - 3})(u_{\leq i - 6}) + 3 P_{\leq i - 3}(|u_{\leq i - 6}|^{4} u_{> i - 3})(u_{\leq i - 6})
\end{equation}
is on frequencies $|\xi| \geq 2^{i - 6}$. Therefore, integrating by parts,
\begin{equation}\label{6.47}
\aligned
\int \int \int Im[\bar{u}_{\leq i - 6} P_{> i - 3}(|u_{\leq i - 6}|^{4} u_{i - 6 \leq \cdot \leq i - 3})](t, y) \frac{(x - y)}{|x - y|} Im[\bar{v} \partial_{x} v](t,x) dx dy dt \\
= \int \int \int Im[\bar{v} \partial_{x} v](t,x) \cdot \frac{\partial_{x}}{\partial_{x}^{2}} Im[\bar{u}_{\leq i - 6} P_{> i - 3}(|u_{\leq i - 6}|^{4} u_{i - 6 \leq \cdot \leq i - 3})](t, x)  dx dy dt \\
\lesssim 2^{-i} \| v_{x} \|_{L_{t}^{4} L_{x}^{\infty}} \| v \|_{L_{t}^{4} L_{x}^{\infty}} \| (u_{i - 6 \leq \cdot \leq i - 3})(u_{\leq i - 9}) \|_{L_{t,x}^{2}} \| u_{\leq i - 6} \|_{L_{t}^{\infty} L_{x}^{8}}^{4} \\
+ 2^{-i} \| v_{x} \|_{L_{t}^{4} L_{x}^{\infty}} \| v \|_{L_{t}^{4} L_{x}^{\infty}} \| u_{i - 6 \leq \cdot \leq i - 3} \|_{L_{t}^{4} L_{x}^{\infty}} \| u_{i - 9 \leq \cdot \leq i - 6} \|_{L_{t}^{4} L_{x}^{\infty}} \| u_{\leq i - 6} \|_{L_{t}^{\infty} L_{x}^{4}}^{4} \\
\lesssim \eta_{0} 2^{j} \| u \|_{X_{i - 3}([0, T] \times \mathbb{R})}^{2}.
\endaligned
\end{equation}
A similar calculation gives the estimate
\begin{equation}\label{6.47.1}
\int \int \int Im[\bar{u}_{\leq i - 6} P_{\leq i - 3}(|u_{\leq i - 6}|^{4} u_{> i - 3})](t, y) \frac{(x - y)}{|x - y|} Im[\bar{v} \partial_{x} v](t,x) dx dy dt \lesssim \eta_{0} 2^{j} \| u \|_{X_{i - 3}([0, T] \times \mathbb{R})}^{2}.
\end{equation}

The terms
\begin{equation}\label{6.48}
\int \int \int |v(t,y)|^{2} \frac{(x - y)}{|x - y|} Re[\bar{\mathcal N}_{i - 3}^{(1)} \partial_{x} u_{\leq i - 3}](t,x) dx dy dt,
 \end{equation}
 and
 \begin{equation}\label{6.49}
\int \int  \int |v(t,y)|^{2} \frac{(x - y)}{|x - y|} Re[\bar{u}_{\leq i - 3} \partial_{x} \mathcal N_{i - 3}^{(1)}](t,x) dx dy dt,
\end{equation}
may be analyzed in a similar manner.\medskip

Plugging $(\ref{6.37})$--$(\ref{6.49})$ into $(\ref{6.36})$ gives
\begin{equation}\label{6.50}
2^{2j} \| \bar{v} u_{\leq i - 3} \|_{L_{t,x}^{2}}^{2} + 2^{2j} \| v u_{\leq i - 3} \|_{L_{t,x}^{2}}^{2} \lesssim 2^{j} + \eta_{0} 2^{j} (1 + \| u \|_{X_{i - 3}}^{6}).
\end{equation}
Summing up over $j \geq i$ implies $(\ref{6.28})$, which completes the proof of Theorem $\ref{t6.2}$.
\end{proof}

Theorem $\ref{t6.2}$ may be upgraded to take advantage of the fact that $u$ is close to the soliton.
\begin{theorem}\label{t6.3}
When $\lambda(t) = \frac{1}{\eta_{1}}$, and $T = 2^{3k}$ for some positive integer $k$,
\begin{equation}\label{6.51}
\| P_{\geq k} u \|_{U_{\Delta}^{2}([0, T] \times \mathbb{R})} \lesssim (\frac{\eta_{1}^{2}}{T} \int_{0}^{\eta_{1}^{-2} T} \| \epsilon(t) \|_{L^{2}}^{2} dt)^{1/2} + \frac{1}{T^{10}}.
\end{equation}
\end{theorem}
\begin{proof}

Make another induction on frequency argument starting at level $\frac{k}{2}$. Observe that Theorem $\ref{t6.2}$ implies that for any $a \in \mathbb{Z}$, $0 \leq a < \eta_{1}^{-1} T^{1/2}$,
\begin{equation}\label{6.54}
\| P_{\geq \frac{k}{2}} u \|_{U_{\Delta}^{2}([a \eta_{1}^{-1} T^{1/2}, (a + 1) \eta_{1}^{-1} T^{1/2}] \times \mathbb{R})} \lesssim 1.
\end{equation}
Next, following Theorem $\ref{t6.2}$,
\begin{equation}\label{6.55}
\aligned
\| P_{\geq \frac{k}{2} + 3} u \|_{U_{\Delta}^{2}([512 a \eta_{1}^{-1} T^{1/2}, 512 (a + 1) \eta_{1}^{-1} T^{1/2}] \times \mathbb{R})} \lesssim \inf_{t \in [512 a \eta_{1}^{-1} T^{1/2}, 512 (a + 1) \eta_{1}^{-1} T^{1/2}]} \| P_{\geq \frac{k}{2} + 3} u(t) \|_{L^{2}} \\ + \eta_{0} \| P_{\geq \frac{k}{2}} u \|_{U_{\Delta}^{2}([512 a \eta_{1}^{-1} T^{1/2}, 512 (a + 1) \eta_{1}^{-1} T^{1/2}] \times \mathbb{R})}.
\endaligned
\end{equation}

Since $Q$ is a smooth function, if $\gamma(t)$ and $\lambda(t)$ are given by Theorem $\ref{t2.3}$ and $\lambda(t) = \frac{1}{\eta_{1}}$,
\begin{equation}\label{6.58}
\aligned
\| P_{\geq \frac{k}{2} + 3} u(t) \|_{L^{2}} \leq \| e^{i \gamma(t)} \lambda(t)^{1/2} u(t, \lambda(t) x) - Q(x) \|_{L^{2}} + \| P_{\geq \frac{k}{2} + 3} Q(x) \|_{L^{2}}
\lesssim \| \epsilon(t) \|_{L^{2}} + T^{-10}.
\endaligned
\end{equation}
Plugging $(\ref{6.58})$ back into $(\ref{6.55})$,
\begin{equation}\label{6.59}
\aligned
\| P_{\geq \frac{k}{2} + 3} u \|_{U_{\Delta}^{2}([512 a \eta_{1}^{-1} T^{1/2}, 512 (a + 1) \eta_{1}^{-1} T^{1/2}] \times \mathbb{R})} \lesssim (\frac{\eta_{1}}{512 T^{1/2}} \int_{512a \eta_{1}^{-1} T^{1/2}}^{512 (a + 1) \eta_{1}^{-1} T^{1/2}}  \| \epsilon (t, x) \|_{L^{2}}^{2} dt)^{1/2} \\ +  T^{-10} + \eta_{0} (\sum_{j = 1}^{512} \| P_{\geq \frac{k}{2}} u \|_{U_{\Delta}^{2}([(512 a + (j - 1)) \eta_{1}^{-1} T^{1/2} , (512 a + j) \eta_{1}^{-1} T^{1/2}] \times \mathbb{R})}^{2})^{1/2}. 
\endaligned
\end{equation}
Arguing by induction in $k$, taking $\lfloor \frac{k}{6} \rfloor$ steps in all, for $\eta_{0}$ sufficiently small,
\begin{equation}\label{6.60}
\aligned
\| P_{\geq k} u \|_{U_{\Delta}^{2}([0, \eta_{1}^{-2} T] \times \mathbb{R})} \lesssim T^{-10} + 2^{k/2} \eta_{0}^{-\frac{k}{6}} + (\frac{\eta_{1}^{2}}{T} \int_{0}^{\eta_{1}^{-2} T}  \| \epsilon(t,x) \|_{L^{2}}^{2} dt)^{1/2} \\
\lesssim T^{-10} + (\frac{\eta_{1}^{2}}{T} \int_{0}^{\eta_{1}^{-2} T}  \| \epsilon(t,x) \|_{L^{2}}^{2} dt)^{1/2}.
\endaligned
\end{equation}
This proves Theorem $\ref{t6.3}$.
\begin{remark}
If $C$ is the implicit constant in $(\ref{6.59})$, then for $\eta_{0} \ll 1$ sufficiently small,
\begin{equation}\label{6.60.1}
(C \eta_{0})^{\lfloor \frac{k}{6} \rfloor} \leq T^{-10}.
\end{equation}
\end{remark}
\end{proof}
The same argument can also be made when $\lambda(t) \geq \frac{1}{\eta_{1}}$ for all $t \in J$.
\begin{theorem}\label{t6.4}
When $\lambda(t) \geq \frac{1}{\eta_{1}}$ on $J = [a, b]$,
\begin{equation}
\int_{J} \lambda(t)^{-2} dt = T,
\end{equation}
and $\eta_{1}^{-2} T = 2^{3k}$, then
\begin{equation}
\| P_{\geq k} u \|_{U_{\Delta}^{2}([a, b] \times \mathbb{R})} \lesssim T^{-10} + (\frac{1}{T} \int_{a}^{b} \| \epsilon(t) \|_{L^{2}}^{2} \lambda(t)^{-2} dt)^{1/2}.
\end{equation}
\end{theorem}
The same argument could also be made for $\lambda(t)$ having a different lower bound, by rescaling $\lambda(t)$ to $\lambda(t) \geq \frac{1}{\eta_{1}}$, computing long time Strichartz estimates, and then rescaling back.

\section{Almost conservation of energy}
Since $(\ref{7.11})$ implies that $\| \epsilon(t) \|_{L^{2}}$ is continuous as a function of time, the mean value theorem implies that under the conditions of Theorem $\ref{t6.4}$, there exists $t_{0} \in [a, b]$ such that
\begin{equation}
\| \epsilon(t_{0}) \|_{L^{2}}^{2} = \frac{1}{T} \int_{a}^{b} \| \epsilon(t) \|_{L^{2}}^{2} \lambda(t)^{-2} dt.
\end{equation}

The next step in proving Theorem $\ref{t3.1}$ is to control
\begin{equation}
\sup_{t \in [a, b]} \| \epsilon(t) \|_{L^{2}},
\end{equation}
as a function of $\| \epsilon(t_{0}) \|_{L^{2}}$. Theorem $\ref{t2.4}$ would be a very useful tool for doing so, except that while $Q$ lies in $H^{s}(\mathbb{R})$ for any $s > 0$, $\epsilon$ need not belong to $H^{s}(\mathbb{R})$ for any $s > 0$. Therefore, Theorem $\ref{t2.4}$ will be used in conjunction with the Fourier truncation method of \cite{bourgain1998refinements}. See also the I-method, for example \cite{colliander2002almost}.

\begin{theorem}\label{t8.3}
Let $J = [a, b]$ be an interval such that
\begin{equation}\label{8.20}
\int_{J} \lambda(t)^{-2} dt = T,
\end{equation}
$\eta_{1}^{-2} T  = 2^{3k}$, and $\lambda(t) \geq \frac{1}{\eta_{1}}$ for all $t \in [a, b]$. Then,
\begin{equation}\label{8.21}
\sup_{t \in J} E(P_{\leq k + 9} u(t)) \lesssim \frac{2^{2k}}{T} \int_{J} \| \epsilon(t) \|_{L^{2}}^{2} \lambda(t)^{-2} dt + 2^{2k} T^{-10}.
\end{equation}
\end{theorem}
\begin{proof}
By the mean value theorem, there exists $t_{0} \in J$ such that
\begin{equation}\label{8.22}
\| \epsilon(t_{0}) \|_{L^{2}}^{2} \lesssim \frac{1}{T} \int_{J} \| \epsilon(t) \|_{L^{2}}^{2} \lambda(t)^{-2} dt.
\end{equation}

Next, decompose the energy. Let $\tilde{Q}$ refer to a rescaled version of $Q$, that is, $\tilde{Q} = \lambda(t_{0})^{-1/2} Q(\lambda(t_{0})^{-1} x)$, and let $\tilde{\epsilon}$ denote the rescaled $\epsilon$, $\tilde{\epsilon} = \lambda(t_{0})^{-1/2} \tilde{\epsilon}(t_{0}, \lambda(t_{0})^{-1} x)$. It is also convenient to split $\tilde{\epsilon}$ into real and imaginary parts,
\begin{equation}
\tilde{\epsilon} = \epsilon_{1} + i \epsilon_{2}.
\end{equation}
As in Theorem $\ref{t2.4}$, by $(\ref{2.14})$,
\begin{equation}\label{8.24}
\aligned
E(P_{\leq k + 9} u) = E(P_{\leq k + 9} \tilde{Q} + P_{\leq k + 9} \tilde{\epsilon}) = E(P_{\leq k + 9} \tilde{Q}) + Re \int P_{\leq k + 9}  \tilde{Q}_{x} \overline{P_{\leq k + 9} \tilde{\epsilon}_{x}} dx \\ - Re \int (P_{\leq k + 9} \tilde{Q})^{5} \overline{(P_{\leq k + 9} \tilde{\epsilon})} dx
+ \frac{1}{2} \| P_{\leq k + 9} \tilde{\epsilon} \|_{\dot{H}^{1}}^{2} - \frac{5}{2} \int (P_{\leq k + 9} \tilde{Q})^{4} (P_{\leq k + 9} \tilde{\epsilon}_{1})^{2} dx - \frac{1}{2} \int (P_{\leq k + 9} \tilde{Q})^{4} (P_{\leq k + 9} \epsilon_{2})^{2} \\
- \int O(|P_{\leq k + 9} \tilde{Q}|^{3} |P_{\leq k + 9} \tilde{\epsilon}|^{3}) + O(|P_{\leq k + 9} \tilde{\epsilon}|^{6}) dx.
\endaligned
\end{equation}
Since $\tilde{Q}$ is smooth and rapidly decreasing, $E(\tilde{Q}) = 0$, and $\lambda(t_{0}) \geq \frac{1}{\eta_{1}}$, Bernstein's inequality implies that
\begin{equation}\label{8.25}
E(P_{\leq k + 9} \tilde{Q}) = \frac{1}{2} \int (P_{\leq k + 9} \tilde{Q}_{x})^{2} - \frac{1}{6} \int (P_{\leq k + 9} \tilde{Q})^{6} \lesssim 2^{-30k}.
\end{equation}
Next, integrating by parts and using $(\ref{2.36})$, the smoothness of $Q$, and Bernstein's inequality,
\begin{equation}\label{8.26}
\aligned
Re \int P_{\leq k + 9} \tilde{Q}_{x} \overline{P_{\leq k + 9} \tilde{\epsilon}_{x}} dx - Re \int (P_{\leq k + 9} \tilde{Q})^{5} \overline{P_{\leq k + 9} \tilde{\epsilon}} dx  \\ = -Re \int (\overline{P_{\leq k + 9} \epsilon}) (P_{\leq k + 9} \tilde{Q}_{xx} + (P_{\leq k + 9} \tilde{Q})^{5}) dx = \frac{1}{2 \lambda(t_{0})^{2}} \| \epsilon \|_{L^{2}}^{2} + O(2^{-30 k}).
\endaligned
\end{equation}
Next, by H{\"o}lder's inequality, since $\lambda(t_{0}) \geq \frac{1}{\eta_{1}}$,
\begin{equation}\label{8.27}
\aligned
\frac{1}{2} \| P_{\leq k + 9} \tilde{\epsilon} \|_{\dot{H}^{1}}^{2} - \frac{5}{2} \int (P_{\leq k + 9} \tilde{Q})^{4} (P_{\leq k + 9} \tilde{\epsilon}_{1})^{2} dx - \frac{1}{2} \int (P_{\leq k + 9} \tilde{Q})^{4} (P_{\leq k + 9} \tilde{\epsilon}_{2})^{2} dx \\ \lesssim \| P_{\leq k + 9} \tilde{\epsilon} \|_{\dot{H}^{1}}^{2} + \frac{1}{\lambda(t_{0})^{2}} \| P_{\leq k + 9} \tilde{\epsilon} \|_{L^{2}}^{2} \lesssim 2^{2k} \| \epsilon(t_{0}) \|_{L^{2}}^{2}.
\endaligned
\end{equation}
By the Sobolev embedding theorem,
\begin{equation}\label{8.28}
\aligned
\int |P_{\leq k + 9} \tilde{\epsilon}|^{3} |P_{\leq k + 9} \tilde{Q}|^{3} + |P_{\leq k + 9} \tilde{\epsilon}|^{6} dx \lesssim \frac{1}{\lambda(t_{0})^{3/2}} \| P_{\leq k + 9} \tilde{\epsilon}(t_{0}) \|_{L^{2}}^{5/2} \| P_{\leq k + 9} \tilde{\epsilon}(t_{0}) \|_{\dot{H}^{1}}^{1/2} \\ + \| P_{\leq k + 9} \tilde{\epsilon}(t_{0}) \|_{L^{2}}^{4} \| P_{\leq k + 9} \tilde{\epsilon}(t_{0}) \|_{\dot{H}^{1}}^{2} \lesssim \frac{1}{\lambda(t_{0})^{2}} \| P_{\leq k + 9} \tilde{\epsilon}(t_{0}) \|_{L^{2}}^{2} + \| P_{\leq k + 9} \tilde{\epsilon}(t_{0}) \|_{L^{2}}^{4} \| P_{\leq k + 9} \tilde{\epsilon}(t_{0}) \|_{\dot{H}^{1}}^{2}.
\endaligned
\end{equation}
Therefore, since $\lambda(t_{0}) \geq \frac{1}{\eta_{1}}$,
\begin{equation}\label{8.29}
E(P_{\leq k + 9} \tilde{u}(t_{0})) \lesssim 2^{2k} \| \epsilon(t_{0}) \|_{L^{2}}^{2} + 2^{-30k}.
\end{equation}

Next compute the change of energy.
\begin{equation}\label{8.30}
\aligned
\frac{d}{dt} E(P_{\leq k + 9} u) = -( P_{\leq k + 9} u_{t}, \Delta P_{\leq k + 9} u )_{L^{2}} - ( P_{\leq k + 9} u_{t}, |P_{\leq k + 9} u|^{4} P_{\leq k + 9} u )_{L^{2}} \\ = -( P_{\leq k + 9} u_{t}, P_{\leq k + 9} F(u) - F(P_{\leq k + 9} u) )_{L^{2}} = ( i \Delta P_{\leq k + 9} u + i P_{\leq k + 9} (|u|^{4} u), P_{\leq k + 9} F(u) - F(P_{\leq k + 9} u) )_{L^{2}}.
\endaligned
\end{equation}

First compute
\begin{equation}\label{8.31}
\int_{t_{0}}^{t'} ( i \Delta P_{\leq k + 9} u, P_{\leq k + 9} F(u) - F(P_{\leq k + 9} u) )_{L^{2}} dt,
\end{equation}
for some $t' \in J$. Making a Littlewood--Paley decomposition,
\begin{equation}
\aligned
\int_{t_{0}}^{t'} ( i \Delta P_{\leq k + 9} u, P_{\leq k + 9} F(u) - F(P_{\leq k + 9} u) )_{L^{2}} dt \\ \sim \sum_{0 \leq k_{5} \leq k_{4} \leq k_{3} \leq k_{2} \leq k_{1}} \sum_{0 \leq k_{6} \leq k + 9} \int_{t_{0}}^{t'} ( i \Delta P_{k_{6}} u, P_{\leq k + 9} (P_{k_{1}} u \cdots P_{k_{5}} u) - (P_{\leq k + 9} P_{k_{1}} u) \cdots (P_{\leq k + 9} P_{k_{5}} u) )_{L^{2}} dt.
\endaligned
\end{equation}
\begin{remark}
For these computations, it is not so important to distinguish between $u$ and $\bar{u}$.
\end{remark}

\noindent \textbf{Case $1$, $k_{1} \leq k + 6$:} In this case $P_{\leq k + 9} P_{k_{1}} = P_{k_{1}}$ and $P_{\leq k + 9}(P_{k_{1}} u \cdots P_{k_{5}} u) = P_{k_{1}} u \cdots P_{k_{5}} u$, so the contribution of these terms is zero. That is, for $k_{1}, ..., k_{5} \leq k + 6$,
\begin{equation}
\aligned
 \int_{t_{0}}^{t'} ( i \Delta P_{k_{6}} u, P_{\leq k + 9} (P_{k_{1}} u \cdots P_{k_{5}} u) - (P_{\leq k + 9} P_{k_{1}} u) \cdots (P_{\leq k + 9} P_{k_{5}} u) )_{L^{2}} dt = 0.
\endaligned
\end{equation}

\noindent \textbf{Case $2$, $k_{1} \geq k + 6$ and $k_{2} \leq k$:} In this case, Fourier support properties imply that $k_{6} \geq k + 3$. Then by Theorem $\ref{t6.4}$, Theorem $\ref{t6.2}$, and $(\ref{6.28})$,
\begin{equation}\label{8.32}
 \int_{t_{0}}^{t'} ( i \Delta P_{k + 3 \leq \cdot \leq k + 9} u, P_{\leq k + 9} ((P_{\leq k} u)^{4} (P_{\geq k + 6} u)) - (P_{\leq k} u)^{4} (P_{k + 6 \leq \cdot \leq k + 9} u) )_{L^{2}} dt
\end{equation}
\begin{equation}\label{8.33}
\aligned
\lesssim 2^{2k} \| (P_{\geq k + 3} u)(P_{\leq k} u) \|_{L_{t,x}^{2}} \| (P_{\geq k + 6} u)(P_{\leq k} u) \|_{L_{t,x}^{2}} \| P_{\leq k} u \|_{L_{t,x}^{\infty}}^{2}
\lesssim \frac{2^{2k}}{T} \int_{J} \| \epsilon(t) \|_{L^{2}}^{2} \lambda(t)^{-2} dt + 2^{2k} T^{-10}.
\endaligned
\end{equation}

\noindent \textbf{Case 3, $k_{1} \geq k + 6$, $k_{2} \geq k$ $k_{3} \leq k$:} If $k_{6} \leq k$, then by Fourier support properties, $k_{2} \geq k + 3$. In that case,
\begin{equation}\label{8.34}
\aligned
\int_{t_{0}}^{t'} ( i \Delta P_{\leq k} u, P_{\leq k + 9}((P_{\geq k + 6} u)(P_{\geq k + 3} u) (P_{\leq k} u)^{3}) - (P_{k + 6 \leq \cdot \leq k + 9} u)(P_{k + 3 \leq \cdot \leq k + 9} u)(P_{\leq k} u)^{3})_{L^{2}} dt \\ \lesssim 2^{2k} \| (P_{\geq k + 6} u)(P_{\leq k} u) \|_{L_{t,x}^{2}} \| (P_{\geq k + 3} u)(P_{\leq k} u) \|_{L_{t,x}^{2}} \| P_{\leq k} u \|_{L_{t,x}^{\infty}}^{2} \lesssim \frac{2^{2k}}{T} \int_{J} \| \epsilon(t) \|_{L^{2}}^{2} \lambda(t)^{-2} dt + 2^{2k} T^{-10}.
\endaligned
\end{equation}
In the case when $k_{6} \geq k$,
\begin{equation}\label{8.34}
\aligned
\int_{t_{0}}^{t'} ( i \Delta P_{k \leq \cdot \leq k + 9} u, P_{\leq k + 9}((P_{\geq k + 6} u)(P_{\geq k} u) (P_{\leq k} u)^{3}) - (P_{k + 6 \leq \cdot \leq k + 9} u)(P_{k \leq \cdot \leq k + 9} u)(P_{\leq k} u)^{3})_{L^{2}} dt \\ \lesssim 2^{2k} \| (P_{\geq k + 6} u)(P_{\leq k} u) \|_{L_{t,x}^{2}} \| P_{\geq k} u \|_{L_{t}^{4} L_{x}^{\infty}}^{2} \| P_{\leq k} u \|_{L_{t,x}^{\infty}} \| u \|_{L_{t}^{\infty} L_{x}^{2}} \lesssim \frac{2^{2k}}{T} \int_{J} \| \epsilon(t) \|_{L^{2}}^{2} \lambda(t)^{-2} dt + 2^{2k} T^{-10}.
\endaligned
\end{equation}

\noindent \textbf{Case 4, $k_{1} \geq 2^{k + 6}$ and $k_{2}, k_{3} \geq k$:} In this case,
\begin{equation}\label{8.35}
\aligned
\int_{t_{0}}^{t'} ( i \Delta P_{\leq k + 9} u, P_{\leq k + 9}((P_{\geq k + 6} u)(P_{\geq k} u)^{2} u^{2}) - (P_{k + 6 \leq \cdot \leq k + 9} u)(P_{k \leq \cdot \leq k + 9} u)^{2}(P_{\leq k + 9} u)^{2} )_{L^{2}} dt \\ \lesssim 2^{2k} \| (P_{\geq k + 6} u)(P_{\leq k} u) \|_{L_{t,x}^{2}} \| P_{\geq k} u \|_{L_{t}^{4} L_{x}^{\infty}}^{2} \| P_{\leq k} u \|_{L_{t,x}^{\infty}} \| u \|_{L_{t}^{\infty} L_{x}^{2}} \\
+ 2^{2k} \| P_{\geq k} u \|_{L_{t}^{4} L_{x}^{\infty}}^{4} \| u \|_{L_{t}^{\infty} L_{x}^{2}}^{2} \lesssim \frac{2^{2k}}{T} \int_{J} \| \epsilon(t) \|_{L^{2}}^{2} \lambda(t)^{-2} dt + 2^{2k} T^{-10}.
\endaligned
\end{equation}

The contribution of the nonlinear terms is similar, using the fact that
\begin{equation}\label{8.37}
( i P_{\leq k + 9} F(u), P_{\leq k + 9} F(u) - F(P_{\leq k + 9} u) )_{L^{2}} = ( i P_{\leq k + 9} F(u), F(P_{\leq k + 9} u) )_{L^{2}}.
\end{equation}
Then make a Littlewood--Paley decomposition,
\begin{equation}\label{8.37.1}
\aligned
(i P_{\leq k + 9} F(u), F(P_{\leq k + 9} u))_{L^{2}} \\ = \sum_{0 \leq k_{5} \leq k_{4} \leq k_{3} \leq k_{2} \leq k_{1}} \sum_{0 \leq k_{5}' \leq k_{4}' \leq k_{3}' \leq k_{2}' \leq k_{1}'} (i P_{\leq k + 9} (u_{k_{1}} \cdots u_{k_{5}}), (P_{\leq k + 9} P_{k_{1}} u) \cdots (P_{\leq k + 9} P_{k_{5}} u))_{L^{2}}.
\endaligned
\end{equation}
\noindent \textbf{Case 1: $k_{1}, k_{1}' \leq k + 6$:} Once again, if $k_{1}, k_{1}' \leq k + 6$, then the right hand side of $(\ref{8.37.1})$ is zero.\medskip

\noindent \textbf{Case 2: $k_{1}$ or $k_{1}' \geq k + 6$, eight terms are $\leq k$:} In the case that $k_{1}$ or $k_{1}' \geq k + 6$, and eight of the terms in $(\ref{8.37.1})$ are at frequency $\leq k$, then by Fourier support properties the final term should be at frequency $\geq k + 3$. The contribution in this case is bounded by
\begin{equation}\label{8.37.2}
\| (P_{\geq k + 6} u)(P_{\leq k} u) \|_{L_{t,x}^{2}} \| (P_{\geq k + 3} u)(P_{\leq k} u) \|_{L_{t,x}^{2}} \| P_{\leq k} u \|_{L_{t,x}^{\infty}}^{6} \lesssim 2^{2k} \frac{1}{T} \int_{J} \| \epsilon(t) \|_{L^{2}}^{2} dt + 2^{2k} T^{-10}.
\end{equation}

\noindent \textbf{Case 3: $k_{1}$ or $k_{1}' \geq k + 6$, two terms are $\geq k$:} The contribution of the case that $k_{1}$ or $k_{1}' \geq k + 6$, two additional terms in $(\ref{8.37.1})$ are at frequency $\geq k$, and the other seven terms are at frequency $\leq k$ is bounded by
\begin{equation}\label{8.37.3}
\| (P_{\geq k + 6} u)(P_{\leq k} u) \|_{L_{t,x}^{2}} \| P_{\geq k} u \|_{L_{t}^{4} L_{x}^{\infty}}^{2} \| P_{\leq k} u \|_{L_{t,x}^{\infty}}^{5} \| u \|_{L_{t}^{\infty} L_{x}^{2}}  \lesssim 2^{2k} \frac{1}{T} \int_{J} \| \epsilon(t) \|_{L^{2}}^{2} dt + 2^{2k} T^{-10}.
\end{equation}

\noindent \textbf{Case 4: $k_{1}$ or $k_{1}' \geq k + 6$ and at least three additional terms in $(\ref{8.37.1})$ are at frequencies $\geq k$.} 

This case may be reduced to a case where at least four terms in $(\ref{8.37.1})$ are at frequency $\geq k$, and at least four terms are at frequency $\leq k + 9$. To see why, notice that all five terms in $F(P_{\leq k + 9} u)$ are at frequency $\leq k + 9$, so if four or five of the terms in $P_{\leq k + 9} F(u)$ are at frequency $\geq k$, then we are fine. 

If exactly, three terms in $P_{\leq k + 9} F(u)$ are at frequency $\geq k$, then take the two terms in $P_{\leq k + 9} F(u)$ that are at frequency $\leq k$ to be terms at frequency $\leq k + 9$. Meanwhile, since at least four terms are at frequency $\geq k$,
\begin{equation}\label{8.37.4}
F(P_{\leq k + 9} u) \sim (P_{k \leq \cdot \leq k + 9} u)(P_{\leq k + 9} u)^{4},
\end{equation}
so in $(\ref{8.37.4})$ there is one term at frequency $\geq k$ and two more terms at frequency $\leq k + 9$.

If exactly two terms in $P_{\leq k + 9} F(u)$ are at frequency $\geq k$, then there are three terms that are at frequency $\leq k$. In that case,
 \begin{equation}\label{8.37.5}
F(P_{\leq k + 9} u) \sim (P_{k \leq \cdot \leq k + 9} u)^{2}(P_{\leq k + 9} u)^{3},
\end{equation}
so in $(\ref{8.37.5})$ there are two terms at frequency $\geq k$ and one term at frequency $\leq k + 9$.

If one term in $P_{\leq k + 9} F(u)$ is at frequency $\geq k$, then there are four terms in $P_{\leq k + 9} F(u)$ at frequency $\leq k$. Then there must be at least three more in $F(P_{\leq k + 9} u)$, so
\begin{equation}\label{8.37.6}
F(P_{\leq k + 9} u) \sim (P_{k \leq \cdot \leq k + 9} u)^{3} u^{2}.
\end{equation}
If no terms in $P_{\leq k + 9} F(u)$ are at frequency $\geq k$, then there must be four in $F(P_{\leq k + 9} u)$, so
\begin{equation}\label{8.37.7}
F(P_{\leq k + 9} u) \sim (P_{k \leq \cdot \leq k + 9} u)^{4} u.
\end{equation}

The contribution of all the different subcases of case four, $(\ref{8.37.4})$--$(\ref{8.37.7})$, may be bounded by
\begin{equation}\label{8.37.8}
\| P_{\geq k} u \|_{L_{t}^{4} L_{x}^{\infty}}^{4} \| u \|_{L_{t}^{\infty} L_{x}^{2}}^{2} \| P_{\leq k + 9} u \|_{L_{t,x}^{\infty}}^{4} \lesssim 2^{2k} \frac{1}{T} \int_{J} \| \epsilon(t) \|_{L^{2}}^{2} dt + 2^{2k} T^{-10}.
\end{equation}
This proves Theorem $\ref{t8.3}$.
\end{proof}

\begin{corollary}\label{c8.4}
If
\begin{equation}
\frac{1}{\eta_{1}} \leq \lambda(t) \leq \frac{1}{\eta_{1}} T^{1/100},
\end{equation}
and
\begin{equation}
\int_{J} \lambda(t)^{-2} dt = T,
\end{equation}
then
\begin{equation}
\sup_{t \in J} \| P_{\leq k + 9} \frac{1}{\lambda(t)^{1/2}} \epsilon(t, \frac{x}{\lambda(t)}) \|_{\dot{H}^{1}}^{2} \lesssim \frac{2^{2k}}{T} \int_{J} \| \epsilon(t) \|_{L^{2}}^{2} \lambda(t)^{-2} dt + 2^{2k} T^{-10},
\end{equation}
and
\begin{equation}
\sup_{t \in J} \| \epsilon(t) \|_{L^{2}}^{2} \lesssim \frac{T^{1/50}}{\eta_{1}^{2}} \frac{2^{2k}}{T} \int_{J} \| \epsilon(t) \|_{L^{2}}^{2} \lambda(t)^{-2} dt + \frac{T^{1/50}}{\eta_{1}^{2}} 2^{2k} T^{-10},
\end{equation}
\end{corollary}
\begin{proof}
The proof uses Theorem $\ref{t8.3}$, Theorem $\ref{t2.4}$, rescaling, and the fact that $Q$ is smooth and all its derivatives are rapidly decreasing.
\end{proof}

\section{A frequency localized Morawetz estimate}
The next step will be to combine long time Strichartz estimates with almost conservation of energy to prove a frequency localized Morawetz estimate adapted to the case when $\lambda(t)$ does not vary too much.

\begin{theorem}\label{t10.1}
Let $J = [a, b]$ be an interval on which $(\ref{6.1.1})$ holds for all $t \in J$, $\frac{1}{\eta_{1}} \leq \lambda(t) \leq \frac{1}{\eta_{1}} T^{1/100}$ for all $t \in J$, and
\begin{equation}\label{10.1}
\int_{J} \lambda(t)^{-2} dt = T.
\end{equation}
Also suppose $2^{3k} = \eta_{1}^{-2} T$ and that $\epsilon = \epsilon_{1} + i \epsilon_{2}$, where $\epsilon$ is given by Theorem $\ref{t2.3}$. Then for $T$ sufficiently large,
\begin{equation}\label{10.1.1}
\aligned
\int_{a}^{b} \| \epsilon(t) \|_{L^{2}}^{2} \lambda(t)^{-2} dt \leq 3(\epsilon_{2}(a), (\frac{1}{2} Q + x Q_{x}))_{L^{2}} - 3(\epsilon_{2}(b), \frac{1}{2} Q + x Q_{x})_{L^{2}} + O(\frac{1}{T^{9}}).
\endaligned
\end{equation}
\end{theorem}
\begin{remark}
The signs on the right hand side of $(\ref{10.1.1})$ are very important.
\end{remark} 
\begin{proof}
The proof uses a frequency localized Morawetz estimate. The Morawetz potential is the same as the Morawetz potential used in \cite{dodson2015global}. See also \cite{dodson2016global}. 

Let $\psi(x) \in C^{\infty}(\mathbb{R})$ be a smooth, even function, satisfying $\psi(x) = 1$ on $|x| \leq 1$, and supported on $|x| \leq 2$. Then for some large $R$, $R = T^{1/25}$ will do, let
\begin{equation}\label{10.2}
\phi(x) = \int_{0}^{x} \chi(\frac{\eta_{1} y}{R}) dy = \int_{0}^{x} \psi^{2}(\frac{\eta_{1} y}{R}) dy,
\end{equation}
and let
\begin{equation}\label{10.3}
M(t) = \int \phi(x) Im[\overline{P_{\leq k + 9} u} \partial_{x} P_{\leq k + 9} u](t,x) dx.
\end{equation}

Doing some algebra using $(\ref{2.14})$, as in $(\ref{8.24})$,
\begin{equation}\label{10.14}
\aligned
u(t,x) = e^{-i \gamma(t)} \lambda(t)^{-1/2} Q(\frac{x}{\lambda(t)}) + e^{-i \gamma(t)} \lambda(t)^{-1/2} \epsilon(t,\frac{x}{\lambda(t)}) = e^{-i \gamma(t)} \tilde{Q}(x) + e^{-i \gamma(t)} \tilde{\epsilon}(t,x).
\endaligned
\end{equation}
Since $Im[\overline{P_{\leq k + 9} u} \partial_{x}(P_{\leq k + 9} u)]$ is invariant under the multiplication operator $u \mapsto e^{-i \gamma(t)}u$,
\begin{equation}\label{10.16}
\aligned
M(t) = \int \phi(x) Im[\overline{ P_{\leq k + 9} \tilde{Q}(x) + P_{\leq k + 9} \tilde{\epsilon}(t, x)} \partial_{x}( P_{\leq k + 9} \tilde{Q}(x) + P_{\leq k + 9} \tilde{\epsilon}(t, x))] dx.
\endaligned
\end{equation}
Since $Q$ is real valued,
\begin{equation}\label{10.17}
\int \phi(x) Im[\overline{ P_{\leq k + 9} \tilde{Q}(x)} \partial_{x}(P_{\leq k + 9} \tilde{Q}(x))] dx = 0.
\end{equation}
Next, by Corollary $\ref{c8.4}$,
\begin{equation}\label{10.20}
\aligned
\int \phi(x) Im[\overline{P_{\leq k + 9} \tilde{\epsilon}(t, x)} \partial_{x}(P_{\leq k + 9} \tilde{\epsilon}(t, x))] dx
\lesssim \frac{R}{\eta_{1}^{2}}  \frac{2^{2k}}{T^{99/100}} \int_{J} \| \epsilon(t) \|_{L^{2}}^{2} \lambda(t)^{-2} dt + \frac{R}{\eta_{1}^{2}} 2^{2k} T^{-9.99}.
\endaligned
\end{equation}
Next, since $\lambda(t) \geq \frac{1}{\eta_{1}}$, $Q$ and $\partial_{x} Q$ are rapidly decreasing, $\phi(x) = x$ for $|x| \leq \frac{R}{\eta_{1}}$, and $|\phi_{x}(x)| \leq 1$,
\begin{equation}\label{10.18}
\aligned
\int \phi(x) Im[\overline{P_{\leq k + 9} \tilde{\epsilon}(t, x)} \partial_{x}(P_{\leq k + 9} \tilde{Q}(x))] dx
= -(\epsilon_{2}, x Q_{x})_{L^{2}} + O(T^{-10}).
\endaligned
\end{equation}
Indeed, since $Q$ is real, by rescaling,
\begin{equation}\label{10.18.0}
\int x Im[\overline{\tilde{\epsilon}(t, x)} \partial_{x}(\tilde{Q}(x))] dx = -(\epsilon_{2}(t), x Q_{x})_{L^{2}}.
\end{equation}
Next, since $\lambda(t) \leq \frac{1}{\eta_{1}} T^{1/100}$ and $R = T^{1/25}$,
\begin{equation}\label{10.18.1}
\aligned
\int x Im[\overline{\tilde{\epsilon}(t, x)} \partial_{x}(\tilde{Q}(x))] dx
- \int \phi(x) Im[\overline{\tilde{\epsilon}(t, x)} \partial_{x}(\tilde{Q}(x)] dx \\
\leq \int_{|x| \geq \frac{R}{\eta_{1}}} \frac{x}{\lambda(t)^{3/2}} |Q_{x}(\frac{x}{\lambda(t)})| |\frac{1}{\lambda(t)^{1/2}} \epsilon(t,\frac{x}{\lambda(t)})| dx \lesssim T^{-10}.
\endaligned
\end{equation}
Also, since $Q$ and all its derivatives are rapidly decreasing, $\lambda(t) \geq \frac{1}{\eta_{1}}$, $R = T^{1/25}$, and $2^{3k} = \eta_{1}^{-2} T$,
\begin{equation}\label{10.18.2}
\aligned
\int \phi(x) Im[\overline{\tilde{\epsilon}(t, x)} \partial_{x}(\tilde{Q}(x))] dx
- \int \phi(x) Im[\overline{\tilde{\epsilon}(t, x)} \partial_{x}(P_{\leq k + 9} \tilde{Q}(x))] dx \lesssim R \| \epsilon \|_{L^{2}} \| P_{\geq k + 9} \tilde{Q}_{x}(x) \|_{L^{2}} \lesssim T^{-10}.
\endaligned
\end{equation}
Next, $(\ref{10.2})$ implies that that $|\phi^{(j)}(x)| \lesssim 1$ for any $j \geq 1$, and since $Q$ is smooth and all its derivatives are rapidly decreasing, so integrating by parts, for $j$ sufficiently large,
\begin{equation}\label{10.18.3}
\aligned
\int \phi(x) Im[\overline{P_{\geq k + 9} \tilde{\epsilon}(t, x)} \partial_{x}(P_{\leq k + 9} \tilde{Q}(x))] dx
= \int \phi(x) Im[\overline{\frac{\Delta^{j}}{\Delta^{j}} P_{\geq k + 9} \tilde{\epsilon}(t, x)} \partial_{x}(P_{\leq k + 9} \tilde{Q}(x))] dx \lesssim T^{-10},
\endaligned
\end{equation}
so $(\ref{10.18.0})$--$(\ref{10.18.3})$ imply $(\ref{10.18})$. Finally,
\begin{equation}\label{10.19}
\aligned
\int \phi(x) Im[\overline{P_{\leq k + 9} \tilde{Q}(x)} \partial_{x}(P_{\leq k + 9} \tilde{\epsilon}(t, x))] dx = (\ref{10.18}) - \int \chi(\frac{\eta_{1} x}{R}) Im[\overline{P_{\leq k + 9} \tilde{Q}(x)} \cdot P_{\leq k + 9} \tilde{\epsilon}(t, x)] dx.
\endaligned
\end{equation}
Making an argument similar to $(\ref{10.18.0})$--$(\ref{10.18.3})$,
\begin{equation}\label{10.19.1}
- \int \chi(\frac{\eta_{1} x}{R}) Im[\overline{P_{\leq k + 9} \tilde{Q}(x)} \cdot P_{\leq k + 9} \tilde{\epsilon}(t, x)] dx = -(\epsilon_{2}, Q)_{L^{2}} + O(T^{-10}).
\end{equation}
Therefore,
\begin{equation}\label{10.21}
\aligned
M(b) - M(a) = 2 (\epsilon_{2}(a), \frac{Q}{2} + x Q_{x})_{L^{2}} - 2 (\epsilon_{2}(b), \frac{Q}{2} + x Q_{x})_{L^{2}} + O(T^{-10}) \\ +O( \frac{R}{\eta_{1}^{2}}  \frac{2^{2k}}{T^{99/100}} \int_{J} \| \epsilon(t) \|_{L^{2}}^{2} \lambda(t)^{-2} dt) + O(\frac{R}{\eta_{1}^{2}} 2^{2k} T^{-9.99}).
\endaligned
\end{equation}

Following $(\ref{6.34})$,
\begin{equation}\label{10.4}
i \partial_{t} P_{\leq k + 9} u + \Delta P_{\leq k + 9} u + F(P_{\leq k + 9} u) = F(P_{\leq k + 9} u) - P_{\leq k + 9} F(u) = -\mathcal N.
\end{equation}

Plugging in $(\ref{10.4})$ and integrating by parts,
\begin{equation}\label{10.5}
\aligned
\frac{d}{dt} M(t) = \int \phi(x) Re[-\overline{P_{\leq k + 9} u}_{xx} P_{\leq k + 9} u_{x} + \overline{P_{\leq k + 9}u} P_{\leq k + 9} u_{xxx}] \\ + \int \phi(x) Re[-|P_{\leq k + 9} u|^{4} \overline{P_{\leq k + 9} u} (P_{\leq k + 9} u_{x}) + \overline{P_{\leq k + 9} u} \partial_{x}(|P_{\leq k + 9} u|^{4} P_{\leq k + 9} u)] \\ + \int \phi(x) Re[\overline{P_{\leq k + 9} u} \partial_{x} \mathcal N](t,x) - \int \phi(x) Re[\bar{\mathcal N} \partial_{x} P_{\leq k + 9} u](t,x) \\
= 2 \int \psi^{2}(\frac{\eta_{1} x}{R}) |\partial_{x} P_{\leq k +9} u|^{2} dx - \frac{\eta_{1}^{2}}{2R^{2}} \int \chi''(\frac{\eta_{1} x}{R}) |P_{\leq k + 9} u|^{2} dx - \frac{2}{3} \int \psi^{2}(\frac{\eta_{1} x}{R}) |P_{\leq k + 9} u|^{6} dx \\ + \int \phi(x) Re[\overline{P_{\leq k + 9} u} \partial_{x} \mathcal N](t,x) - \int \phi(x) Re[\bar{\mathcal N} \partial_{x} P_{\leq k + 9} u](t,x).
\endaligned
\end{equation}

Next, following $(\ref{6.41})$,
\begin{equation}\label{10.6}
\aligned
\mathcal N = P_{\leq k + 9} (3 |u_{\leq k}|^{4} u_{\geq k + 6} + 2|u_{\leq k}|^{2} (u_{\leq k})^{2} \bar{u}_{\geq k + 6}) \\ - (3 |u_{\leq k}|^{4}  u_{\geq k + 6} + 2|u_{\leq k}|^{2} (u_{\leq k})^{2} \overline{P_{k + 6 \leq \cdot \leq k + 9} u}) \\
+ P_{\leq k + 9} O((u_{\geq k})(u_{\geq k + 6}) u^{3}) + O((P_{k + 6 \leq \cdot \leq k + 9} u)(P_{k \leq \cdot \leq k + 9} u) u^{3}) = \mathcal N^{(1)} + \mathcal N^{(2)}.
\endaligned
\end{equation}
As in $(\ref{6.42})$, since $|\phi(x)| \lesssim \eta_{1}^{-1} R$, by Theorems $\ref{t6.2}$ and $\ref{t6.3}$,
\begin{equation}\label{10.7}
\aligned
\int_{a}^{b} \int \phi(x) Re[\overline{P_{\leq k + 9} u} \partial_{x} \mathcal N^{(2)}] dx dt - \int_{a}^{b} \int \phi(x) Re[\bar{\mathcal N^{(2)}} \partial_{x} P_{\leq k + 9} u] dx dt \\
\lesssim \frac{2^{k} R \eta_{1}^{-1}}{T} \int \| \epsilon(t) \|_{L^{2}}^{2} \lambda(t)^{-2} dt + \frac{2^{k} R \eta_{1}^{-1}}{T^{10}}.
\endaligned
\end{equation}
Making calculations identical to the estimate $(\ref{10.7})$,
\begin{equation}\label{10.9}
\aligned
 \int_{a}^{b} \int \phi(x) Re[\overline{P_{k + 3 \leq \cdot \leq k + 9} u} \partial_{x} \mathcal N^{(1)}] dx dt - \int_{a}^{b} \int \phi(x) Re[\overline{\mathcal N^{(1)}} \partial_{x}  P_{k + 3 \leq \cdot \leq k + 9} u] dx dt \\
\lesssim 2^{k} R \| (u_{\geq k + 6})(u_{\leq k}) \|_{L_{t,x}^{2}} \| (u_{\geq k + 3})(u_{\leq k}) \|_{L_{t,x}^{2}} \| u_{\leq k} \|_{L_{t,x}^{\infty}}^{2}
\lesssim \frac{2^{k} R \eta_{1}^{-1}}{T} \int_{a}^{b} \| \epsilon(t) \|_{L^{2}}^{2} \lambda(t)^{-2} dt +  \frac{2^{k} R \eta_{1}^{-1}}{T^{10}}.
\endaligned
\end{equation}
Finally, using Bernstein's inequality and the integration by parts argument in $(\ref{6.47})$,
\begin{equation}\label{10.10}
\aligned
 \int_{0}^{T} \int \phi(x) Re[\overline{P_{\leq k + 3} u} \partial_{x} \mathcal N^{(1)}] dx dt - \int_{0}^{T} \int \phi(x) Re[\overline{\mathcal N^{(1)}} \partial_{x} P_{\leq k + 3} u] dx dt \\
 \lesssim \| P_{\geq k + 6} \phi(x) \|_{L^{\infty}} \| (P_{\geq k + 6} u)(P_{\leq k } u) \|_{L_{t,x}^{2}} \| u_{\leq k } \|_{L_{t,x}^{8}}^{4} \lesssim \frac{1}{T^{10}}.
 \endaligned
\end{equation}
\begin{remark}
The last estimate follows from the fact that $\phi$ is smooth and $\int_{J} \lambda(t)^{-2} dt = T$, which by a local well-posedness argument implies
\begin{equation}
\| u \|_{L_{t,x}^{6}(J \times \mathbb{R})} \lesssim T^{1/6}.
\end{equation}
\end{remark}

Plugging this estimate of the error term back into $(\ref{10.5})$,
\begin{equation}\label{10.11}
\aligned
2 \int_{a}^{b} \int \psi^{2}(\frac{\eta_{1} x}{R}) |\partial_{x} P_{\leq k + 9} u|^{2} dx dt - \frac{\eta_{1}^{2}}{2R^{2}} \int_{a}^{b} \int \chi''(\frac{\eta_{1} x}{R}) |P_{\leq k + 9} u|^{2} dx dt \\ - \frac{2}{3} \int_{a}^{b} \int \psi^{2}(\frac{\eta_{1} x}{R}) |P_{\leq k + 9} u|^{6} dx dt  = 2 (\epsilon_{2}(a), \frac{Q}{2} + x Q_{x})_{L^{2}} - 2 (\epsilon_{2}(b), \frac{Q}{2} + x Q_{x})_{L^{2}} \\ + O(\frac{2^{2k} T^{1/20}}{T} \int_{a}^{b} \| \epsilon(t) \|_{L^{2}}^{2} \lambda(t)^{-2} dt) + O(\frac{1}{T^{9}}).
\endaligned
\end{equation}

Since $Q$ is a real-valued function,
\begin{equation}\label{10.15}
\aligned
|P_{\leq k + 9} u|^{2} = (P_{\leq k + 9} \tilde{Q}(x))^{2} + 2 P_{\leq k + 9} \tilde{Q}(x) \cdot P_{\leq k + 9} \tilde{\epsilon}_{1}(t, x)  + |P_{\leq k + 9} \tilde{\epsilon}(t, x)|^{2}.
\endaligned
\end{equation}
The support of $\psi''(x)$, the fact that $\lambda(t) \leq \frac{1}{\eta_{1}} T^{1/100}$, and $(\ref{1.12})$ imply that
\begin{equation}\label{10.15.1}
\frac{\eta_{1}^{2}}{R^{2}} \int \chi''(\frac{\eta_{1} x}{R}) \tilde{Q}(x)^{2} dx \lesssim \frac{\eta_{1}^{2}}{R^{2}} \frac{1}{T^{11}} \lesssim \frac{1}{\lambda(t)^{2}} \frac{1}{T^{11}}.
\end{equation}
Also, since $Q$ and all its derivatives are rapidly decreasing, and $\lambda(t) \geq \frac{1}{\eta_{1}}$,
\begin{equation}\label{10.15.2}
\| P_{\geq k + 9} \tilde{Q}(x) \|_{L^{2}}^{2} \lesssim 2^{-30k}.
\end{equation}
Therefore, since $R = T^{1/25}$ and $\lambda(t) \leq \frac{1}{\eta_{1}} T^{1/100}$, $(\ref{10.15.1})$, $(\ref{10.15.2})$, and the Cauchy--Schwarz inequality imply
\begin{equation}\label{10.16}
\frac{\eta_{1}^{2}}{R^{2}} \int \chi''(\frac{\eta_{1} x}{R}) |P_{\leq k + 9} u(t,x)|^{2} dx \lesssim \frac{1}{\lambda(t)^{2}} \frac{1}{T^{11}} + \frac{1}{\lambda(t)^{2}} \frac{1}{R} \| \epsilon \|_{L^{2}}^{2}.
\end{equation}

Next, letting $\epsilon_{1x} + i \epsilon_{2x} = \partial_{x} \epsilon$, decompose
\begin{equation}\label{10.17}
\aligned
2 \int \psi^{2}(\frac{\eta_{1} x}{R}) |P_{\leq k + 9} u_{x}|^{2} dx - \frac{2}{3} \int \psi^{2}(\frac{\eta_{1} x}{R}) |P_{\leq k + 9} u|^{6} dx \\ = \frac{4}{\lambda(t)^{3}} \int \psi^{2}(\frac{\eta_{1} x}{R}) (\frac{1}{2} P_{\leq k + 9} Q_{x}(\frac{x}{\lambda(t)})^{2} - \frac{1}{6} P_{\leq k + 9} Q(\frac{x}{\lambda(t)})^{6}) dx \\
+ \frac{4}{\lambda(t)^{3}} \int \psi^{2}(\frac{\eta_{1} x}{R}) (P_{\leq k + 9} Q_{x}(\frac{x}{\lambda(t)}) P_{\leq k + 9} \epsilon_{1x}(t, \frac{x}{\lambda(t)}) - P_{\leq k + 9} Q(\frac{x}{\lambda(t)})^{5} P_{\leq k + 9} \epsilon_{1}(t,\frac{x}{\lambda(t)})) dx \\
+ \frac{4}{\lambda(t)^{3}} \int \psi^{2}(\frac{\eta_{1} x}{R}) (\frac{1}{2} (P_{\leq k + 9} \epsilon_{1x}(t,\frac{x}{\lambda(t)}))^{2} - \frac{5}{2} P_{\leq k + 9} Q(\frac{x}{\lambda(t)})^{4} ((P_{\leq k + 9} \epsilon_{1}(t,\frac{x}{\lambda(t)})))^{2}) dx \\ + \frac{4}{\lambda(t)^{3}} \int \psi^{2}(\frac{\eta_{1} x}{R})(\frac{1}{2} ( (P_{\leq k + 9} \epsilon_{2x}(t,\frac{x}{\lambda(t)})))^{2} - \frac{1}{2} P_{\leq k +9} Q(\frac{x}{\lambda(t)})^{4} ((P_{\leq k + 9} \epsilon_{2}(t,\frac{x}{\lambda(t)})))^{2} dx \\
- \frac{4}{\lambda(t)^{3}} \int \psi^{2}(\frac{\eta_{1} x}{R}) (\frac{10}{3} (P_{\leq k + 9} Q(\frac{x}{\lambda(t)}))^{3} (P_{\leq k + 9} \epsilon(t,\frac{x}{\lambda(t)}))^{3} + \frac{5}{2} (P_{\leq k + 9} Q(\frac{x}{\lambda(t)}))^{2} (P_{\leq k + 9} \epsilon(t, \frac{x}{\lambda(t)}))^{4} \\ + (P_{\leq k + 9} Q(\frac{x}{\lambda(t)})) (P_{\leq k + 9}\epsilon(t, \frac{x}{\lambda(t)}))^{5} + \frac{1}{6} (P_{\leq k + 9} \epsilon(t, \frac{x}{\lambda(t)}))^{6}) dx.
\endaligned
\end{equation}
\begin{remark}
Due to the presence of derivatives in
\begin{equation}
2 \int \psi^{2}(\frac{\eta_{1} x}{R}) |P_{\leq k + 9} u_{x}|^{2} dx - \frac{2}{3} \int \psi^{2}(\frac{\eta_{1} x}{R}) |P_{\leq k + 9} u|^{6} dx,
\end{equation}
it is convenient to dispense with the $\tilde{Q}(x)$ and $\tilde{\epsilon}(t,x)$ notation and return to the $Q$ and $\epsilon$ notation. We understand that $P_{\leq k + 9} Q(\frac{x}{\lambda(t)})$ denotes the frequency projection after rescaling, not a rescaled projection. A rescaled projection appears in $(\ref{10.29})$.
\end{remark}

For terms of order $\epsilon^{3}$ and higher, it is not too important to pay attention to complex conjugates, since these terms will be estimated using H{\"o}lder's inequality.

First, using the fact that
\begin{equation}
\frac{1}{2} Q_{x}^{2} - \frac{1}{6} Q^{6} = \frac{1}{2} Q^{2} - \frac{1}{3} Q^{6},
\end{equation}
combined with the fact that $\frac{1}{\eta_{1}} \leq \lambda \leq \frac{1}{\eta_{1}} T^{1/100}$, $R = T^{1/25}$, and $Q$ is smooth and rapidly decreasing,
\begin{equation}\label{10.19}
\frac{4}{\lambda(t)^{3}} \int \psi^{2}(\frac{\eta_{1} x}{R}) (\frac{1}{2} Q_{x}(\frac{x}{\lambda(t)})^{2} - \frac{1}{6} Q(\frac{x}{\lambda(t)})^{6}) dx = \frac{4}{\lambda(t)^{3}} \int \psi^{2}(\frac{\eta_{1} x}{R}) (\frac{1}{2} Q(\frac{x}{\lambda(t)})^{2} - \frac{1}{3} Q(\frac{x}{\lambda(t)})^{6}) \lesssim \frac{1}{\lambda(t)^{2}} \frac{1}{T^{11}}.
\end{equation}
Also, since $\eta_{1}^{-2} T = 2^{3k}$, and $Q$ and its derivatives are smooth and rapidly decreasing, $\lambda(t) \geq \frac{1}{\eta_{1}}$ and Bernstein's inequality implies that
\begin{equation}\label{10.20}
\frac{2}{\lambda(t)^{3}} \int \psi^{2}(\frac{\eta_{1} x}{R}) Q_{x}(\frac{x}{\lambda(t)})^{2} dx - \frac{2}{\lambda(t)^{3}} \int \psi^{2}(\frac{\eta_{1} x}{R}) (P_{\leq k + 9} Q_{x})(\frac{x}{\lambda(t)})^{2} dx \lesssim \frac{1}{\lambda(t)^{2}} \frac{1}{T^{11}},
\end{equation}
and
\begin{equation}\label{10.21}
\frac{2}{3 \lambda(t)^{3}} \int \psi^{2}(\frac{\eta_{1} x}{R}) Q(\frac{x}{\lambda(t)})^{6} dx - \frac{2}{3 \lambda(t)^{3}} \int \psi^{2}(\frac{\eta_{1} x}{R}) (P_{\leq k + 9} Q)(\frac{x}{\lambda(t)})^{6} dx \lesssim \frac{1}{\lambda(t)^{2}} \frac{1}{T^{11}}.
\end{equation}
Therefore,
\begin{equation}\label{10.22}
\aligned
2 \int \psi^{2}(\frac{\eta_{1} x}{R}) |P_{\leq k + 9} u_{x}|^{2} dx - \frac{2}{3} \int \psi^{2}(\frac{\eta_{1} x}{R}) |P_{\leq k + 9} u|^{6} dx \\ = \frac{4}{\lambda(t)^{3}} \int \psi^{2}(\frac{\eta_{1} x}{R}) (P_{\leq k + 9} Q_{x}(\frac{x}{\lambda(t)}) P_{\leq k + 9} \epsilon_{1x}(t, \frac{x}{\lambda(t)}) - P_{\leq k + 9} Q(\frac{x}{\lambda(t)})^{5} P_{\leq k + 9} \epsilon_{1}(t,\frac{x}{\lambda(t)})) dx \\
+ \frac{4}{\lambda(t)^{3}} \int \psi^{2}(\frac{\eta_{1} x}{R}) (\frac{1}{2} (P_{\leq k + 9} \epsilon_{1x}(t,\frac{x}{\lambda(t)}))^{2} - \frac{5}{2} P_{\leq k + 9} Q(\frac{x}{\lambda(t)})^{4} ((P_{\leq k + 9} \epsilon_{1}(t,\frac{x}{\lambda(t)})))^{2}) dx \\ 
+ \frac{4}{\lambda(t)^{3}} \int \psi^{2}(\frac{\eta_{1} x}{R})(\frac{1}{2}  (P_{\leq k + 9} \epsilon_{2x}(t,\frac{x}{\lambda(t)}))^{2} - \frac{1}{2} P_{\leq k +9} Q(\frac{x}{\lambda(t)})^{4} ((P_{\leq k + 9} \epsilon_{2}(t,\frac{x}{\lambda(t)})))^{2} dx \\
- \frac{4}{\lambda(t)^{3}} \int \psi^{2}(\frac{\eta_{1} x}{R}) (\frac{10}{3} (P_{\leq k + 9} Q(\frac{x}{\lambda(t)}))^{3} (P_{\leq k + 9} \epsilon(t,\frac{x}{\lambda(t)}))^{3} + \frac{5}{2} (P_{\leq k + 9} Q(\frac{x}{\lambda(t)}))^{2} (P_{\leq k + 9} \epsilon(t, \frac{x}{\lambda(t)}))^{4} \\ + (P_{\leq k + 9} Q(\frac{x}{\lambda(t)})) (P_{\leq k + 9}\epsilon(t, \frac{x}{\lambda(t)}))^{5} + \frac{1}{6} (P_{\leq k + 9} \epsilon(t, \frac{x}{\lambda(t)}))^{6}) dx + O(\frac{1}{\lambda(t)^{2}} \frac{1}{T^{11}}).
\endaligned
\end{equation}

Integrating by parts,
\begin{equation}\label{10.23}
\aligned
\frac{4}{\lambda(t)^{3}} \int \psi^{2}(\frac{\eta_{1} x}{R}) (P_{\leq k + 5} Q_{x}(\frac{x}{\lambda(t)}) P_{\leq k + 5} \epsilon_{1x}(t, \frac{x}{\lambda(t)}) - P_{\leq k + 5} Q(\frac{x}{\lambda(t)})^{5} P_{\leq k + 5} \epsilon(t,\frac{x}{\lambda(t)})) dx \\ 
= -\frac{4}{\lambda(t)^{3}} \int \psi^{2}(\frac{\eta_{1} x}{R}) (P_{\leq k + 5} Q_{xx} + P_{\leq k + 5} Q^{5})(\frac{x}{\lambda(t)}) \cdot P_{\leq k + 5} \epsilon_{1}(t, \frac{x}{\lambda(t)}) dx \\ 
- \frac{8 \eta_{1}}{R \lambda(t)^{2}} \int \psi(\frac{\eta_{1} x}{R}) \psi'(\frac{\eta_{1} x}{R}) P_{\leq k + 5} Q_{x}(\frac{x}{\lambda(t)}) P_{\leq k + 5} \epsilon_{1}(t, \frac{x}{\lambda(t)}) dx \\ 
+ \frac{4}{\lambda(t)^{3}} \int \psi(\frac{\eta_{1} x}{R})^{2} P_{\leq k + 5} \epsilon_{1}(t, \frac{x}{\lambda(t)}) [(P_{\leq k + 5} Q)^{5} - P_{\leq k + 5} Q^{5}](\frac{x}{\lambda(t)}) dx.
\endaligned
\end{equation}
Again using $(\ref{1.12})$, $\frac{1}{\eta_{1}} \leq \lambda(t) \leq \frac{1}{\eta_{1}} T^{1/100}$, and the support of $\psi'(x)$,
\begin{equation}\label{10.24}
\frac{8 \eta_{1}}{R \lambda(t)^{2}} \int \psi(\frac{\eta_{1} x}{R}) \psi'(\frac{\eta_{1} x}{R}) Q_{x}(\frac{x}{\lambda(t)}) P_{\leq k + 5} \epsilon_{1}(t, \frac{x}{\lambda(t)}) dx \lesssim \frac{1}{T^{6}} \frac{1}{\lambda(t)^{2}} \| \epsilon \|_{L^{2}}
\end{equation}
Also, since $Q$ and all its derivatives are rapidly decreasing, by Bernstein's inequality,
\begin{equation}\label{10.25}
\aligned
\frac{8 \eta_{1}}{R \lambda(t)^{2}} \int \psi(\frac{\eta_{1} x}{R}) \psi'(\frac{\eta_{1} x}{R}) P_{\geq k + 5} Q_{x}(\frac{x}{\lambda(t)}) P_{\leq k + 5} \epsilon_{1}(t, \frac{x}{\lambda(t)}) dx \lesssim \frac{1}{T^{6} \lambda(t)^{2}} \| \epsilon \|_{L^{2}},
\endaligned
\end{equation}
and
\begin{equation}
\frac{4}{\lambda(t)^{3}} \int \psi(\frac{\eta_{1} x}{R})^{2} P_{\leq k + 5} \epsilon_{1}(t, \frac{x}{\lambda(t)}) [(P_{\leq k + 5} Q)^{5} - P_{\leq k + 5} Q^{5}](\frac{x}{\lambda(t)}) dx \lesssim \frac{1}{T^{6} \lambda(t)^{2}} \| \epsilon \|_{L^{2}}.
\end{equation}

Meanwhile, by conservation of mass, $(\ref{1.12})$, $(\ref{1.14})$, the upper and lower bounds of $\lambda(t)$, and the fact that $Q$ and all its derivatives are rapidly decreasing,
\begin{equation}\label{10.26}
\aligned
-\frac{4}{\lambda(t)^{3}} \int \psi^{2}(\frac{\eta_{1} x}{R}) P_{\leq k + 5} (Q_{xx} + Q^{5})(\frac{x}{\lambda(t)}) \cdot P_{\leq k + 5} \epsilon_{1}(t, \frac{x}{\lambda(t)}) dx \\ 
= -\frac{4}{\lambda(t)^{3}} \int \psi^{2}(\frac{\eta_{1} x}{R}) (Q_{xx} + Q^{5})(\frac{x}{\lambda(t)}) \cdot P_{\leq k + 5} \epsilon_{1}(t, \frac{x}{\lambda(t)}) dx + O(\frac{1}{\lambda(t)^{2} T^{6}} \| \epsilon \|_{L^{2}}) \\
= -\frac{4}{\lambda(t)^{3}} \int \psi^{2}(\frac{\eta_{1} x}{R}) Q(\frac{x}{\lambda(t)}) \cdot P_{\leq k + 5} \epsilon_{1}(t, \frac{x}{\lambda(t)}) dx + O(\frac{1}{\lambda(t)^{2} T^{6}} \| \epsilon \|_{L^{2}}) \\
= -\frac{4}{\lambda(t)^{3}} \int Q(\frac{x}{\lambda(t)}) \cdot P_{\leq k + 5} \epsilon_{1}(t, \frac{x}{\lambda(t)}) dx + O(\frac{1}{T^{6} \lambda(t)^{2}} \| \epsilon \|_{L^{2}}) \\
= -\frac{4}{\lambda(t)^{3}} \int Q(\frac{x}{\lambda(t)}) \cdot \epsilon_{1}(t, \frac{x}{\lambda(t)}) dx + O(\frac{1}{T^{6} \lambda(t)^{2}} \| \epsilon \|_{L^{2}}) = \frac{2}{\lambda(t)^{2}} \| \epsilon \|_{L^{2}}^{2} + O(\frac{1}{T^{6} \lambda(t)^{2}} \| \epsilon \|_{L^{2}}).
\endaligned
\end{equation}
Therefore,
\begin{equation}\label{10.27}
\aligned
2 \int \psi^{2}(\frac{\eta_{1} x}{R}) |P_{\leq k + 9} u_{x}|^{2} dx - \frac{2}{3} \int \psi^{2}(\frac{\eta_{1} x}{R}) |P_{\leq k + 9} u|^{6} dx = \frac{2}{\lambda(t)^{2}} \| \epsilon \|_{L^{2}}^{2} \\ 
+ \frac{4}{\lambda(t)^{3}} \int \psi^{2}(\frac{\eta_{1} x}{R}) (\frac{1}{2} (P_{\leq k + 9} \epsilon_{1x}(t,\frac{x}{\lambda(t)}))^{2} - \frac{5}{2} P_{\leq k + 9} Q(\frac{x}{\lambda(t)})^{4} ((P_{\leq k + 9} \epsilon_{1}(t,\frac{x}{\lambda(t)})))^{2}) dx \\ 
+ \frac{4}{\lambda(t)^{3}} \int \psi^{2}(\frac{\eta_{1} x}{R})(\frac{1}{2}  (P_{\leq k + 9} \epsilon_{2x}(t,\frac{x}{\lambda(t)}))^{2} - \frac{1}{2} P_{\leq k +9} Q(\frac{x}{\lambda(t)})^{4} ((P_{\leq k + 9} \epsilon_{2}(t,\frac{x}{\lambda(t)})))^{2} dx \\
- \frac{4}{\lambda(t)^{3}} \int \psi^{2}(\frac{\eta_{1} x}{R}) (\frac{10}{3} (P_{\leq k + 9} Q(\frac{x}{\lambda(t)}))^{3} (P_{\leq k + 9} \epsilon(t,\frac{x}{\lambda(t)}))^{3} + \frac{5}{2} (P_{\leq k + 9} Q(\frac{x}{\lambda(t)}))^{2} (P_{\leq k + 9} \epsilon(t, \frac{x}{\lambda(t)}))^{4} \\ + (P_{\leq k + 9} Q(\frac{x}{\lambda(t)})) (P_{\leq k + 9}\epsilon(t, \frac{x}{\lambda(t)}))^{5} + \frac{1}{6} (P_{\leq k + 9} \epsilon(t, \frac{x}{\lambda(t)}))^{6}) dx + O(\frac{1}{\lambda(t)^{2}} \frac{1}{T^{11}}) + O(\frac{1}{R \lambda(t)^{2}} \| \epsilon \|_{L^{2}}^{2}).
\endaligned
\end{equation}

Next, by Bernstein's inequality, since $\frac{1}{\eta_{1}} \leq \lambda(t) \leq \frac{1}{\eta_{1}} T^{1/100}$,
\begin{equation}\label{10.28}
\aligned
\frac{4}{\lambda(t)^{3}} \int \psi^{2}(\frac{\eta_{1} x}{R}) (\frac{1}{2} (P_{\leq k + 9} \epsilon_{1x}(t, \frac{x}{\lambda(t)}))^{2} - \frac{5}{2} P_{\leq k + 9} Q(\frac{x}{\lambda(t)})^{4} (P_{\leq k + 9} \epsilon_{1}(t, \frac{x}{\lambda(t)}))^{2}) dx \\
= \frac{4}{\lambda(t)^{3}} \int \psi^{2}(\frac{\eta_{1} x}{R}) (\frac{1}{2} (P_{\leq k + 9} \epsilon_{1x}(t, \frac{x}{\lambda(t)}))^{2} - \frac{5}{2}  Q(\frac{x}{\lambda(t)})^{4} (P_{\leq k + 9} \epsilon_{1}(t, \frac{x}{\lambda(t)}))^{2}) dx + O(\frac{1}{\lambda(t)^{2} R} \| \epsilon \|_{L^{2}}^{2}).
\endaligned
\end{equation}
Taking $k(t) \in \mathbb{R}$ that satisfies $2^{k(t)} = \lambda(t)$ and rescaling,
\begin{equation}
\aligned
\frac{4}{\lambda(t)^{3}} \int \psi^{2}(\frac{\eta_{1} x}{R}) (\frac{1}{2} (P_{\leq k + 9} \epsilon_{1x}(t, \frac{x}{\lambda(t)}))^{2} - \frac{5}{2}  Q(\frac{x}{\lambda(t)})^{4} (P_{\leq k + 9} \epsilon_{1}(t, \frac{x}{\lambda(t)}))^{2}) dx \\
= \frac{4}{\lambda(t)^{2}} \int \psi^{2}(\frac{\eta_{1} \lambda(t) x}{R}) (\frac{1}{2} (P_{\leq k + 9 + k(t)} \epsilon_{1x}(t,x))^{2} - \frac{5}{2}  Q(x)^{4} (P_{\leq k + 9 + k(t)} \epsilon_{1}(t, x))^{2}) dx.
\endaligned
\end{equation}
Integrating by parts,
\begin{equation}\label{10.29}
\aligned 
\frac{4}{\lambda(t)^{2}} \int \psi^{2}(\frac{\eta_{1} \lambda(t) x}{R}) (\frac{1}{2} (P_{\leq k + 9 + k(t)} \epsilon_{1}(t,x))_{x}^{2} - \frac{5}{2}  Q(x)^{4} (P_{\leq k + 9 + k(t)} \epsilon_{1}(t, x))^{2} dx \\
= \frac{2}{\lambda(t)^{2}} \| \psi(\frac{\eta_{1} \lambda(t) x}{R}) (P_{\leq k + 9 + k(t)} \epsilon_{1}(t, x)) \|_{\dot{H}^{1}}^{2} - \frac{10}{\lambda(t)^{2}} \int \psi(\frac{\eta_{1} \lambda(t) x}{R})^{2} Q(\frac{x}{\lambda(t)})^{4} (P_{\leq k + 9 + k(t)} \epsilon_{1}(t, x))^{2} dx \\ + O(\frac{\eta_{1}^{2}}{R^{2}} \| \epsilon \|_{L^{2}}^{2})
= \frac{2}{\lambda(t)^{2}} (\mathcal L \tilde{\epsilon}, \tilde{\epsilon}) - \frac{2}{\lambda(t)^{2}} \| \tilde{\epsilon} \|_{L^{2}}^{2} + O(\frac{1}{R \lambda(t)^{2}} \| \epsilon \|_{L^{2}}^{2}),
\endaligned
\end{equation}
where $\mathcal L$ is given in $(\ref{2.39})$, and
\begin{equation}\label{10.30}
\tilde{\epsilon} = \psi(\frac{\eta_{1} \lambda(t) x}{R}) (P_{\leq k + 9 + k(t)} \epsilon_{1}(t,x)).
\end{equation}
\begin{remark}
This $\tilde{\epsilon}$ is not the same as the $\tilde{\epsilon}$ in $(\ref{10.14})$.
\end{remark}
For a function $u \perp Q^{3}$ and $u \perp Q_{x}$, by the spectral properties of $\mathcal L$,
\begin{equation}
 (\mathcal L u, u)_{L^{2}} - (u, u)_{L^{2}} \geq 0.
\end{equation}
For a general $u \in L^{2}$, $u = a_{1} Q^{3} + a_{2} Q_{x} + u^{\perp}$, where $u^{\perp} \perp Q^{3}$ and $u^{\perp} \perp Q_{x}$,
\begin{equation}
 (\mathcal L u, u)_{L^{2}} - (u, u)_{L^{2}} \geq -O(a_{1}^{2}) - O(a_{2}^{2}).
\end{equation}
Since $\epsilon_{1} \perp Q^{3}$ and $\epsilon_{1} \perp Q_{x}$, by Bernstein's inequality, and the fact that $\frac{1}{\eta_{1}} \leq \lambda(t) \leq \frac{1}{\eta_{1}} T^{1/100}$,
\begin{equation}
(\tilde{\epsilon}, Q^{3})_{L^{2}} = (\epsilon_{1}, Q^{3})_{L^{2}} - ((1 - \psi(\frac{x \eta_{1} \lambda(t)}{R})) \epsilon_{1}, Q^{3})_{L^{2}} - (\psi(\frac{\eta_{1} \lambda(t) x}{R}) P_{\geq k + 9 + k(t)} \epsilon_{1}, Q^{3})_{L^{2}} \lesssim \frac{1}{R} \| \epsilon \|_{L^{2}},
\end{equation}
and
\begin{equation}
(\tilde{\epsilon}, Q_{x})_{L^{2}} = (\epsilon_{1}, Q_{x})_{L^{2}} - ((1 - \psi(\frac{x \eta_{1} \lambda(t)}{R})) \epsilon_{1}, Q_{x})_{L^{2}} - (\psi(\frac{\eta_{1} \lambda(t) x}{R}) P_{\geq k + 9 + k(t)} \epsilon_{1}, Q_{x})_{L^{2}} \lesssim \frac{1}{R} \| \epsilon \|_{L^{2}}.
\end{equation}

\noindent Therefore, for some $0 \ll \delta < 1$, $\delta = \frac{1}{100}$ will do, since $|Q(x)|^{3} \leq 3$,
\begin{equation}\label{10.33}
\aligned
\frac{2}{\lambda(t)^{2}} (\mathcal L \tilde{\epsilon}, \tilde{\epsilon}) - \frac{2}{\lambda(t)^{2}} \| \tilde{\epsilon} \|_{L^{2}}^{2} + O(\frac{1}{R \lambda(t)^{2}} \| \epsilon \|_{L^{2}}^{2}) \geq \frac{\delta}{\lambda(t)^{3}} \| \psi(\frac{\eta_{1} x}{R}) (P_{\leq k + 9} \epsilon_{1x}(t, \frac{x}{\lambda(t)})) \|_{L^{2}}^{2} \\ -O(\frac{1}{R} \frac{1}{\lambda(t)^{2}} \| \epsilon \|_{L^{2}}^{2}) - \frac{15 \delta}{\lambda(t)^{2}} \| \epsilon_{1} \|_{L^{2}}^{2}.
\endaligned
\end{equation}
Likewise, since $\epsilon \perp i Q^{3}$ and $\epsilon \perp i Q_{x}$,
\begin{equation}\label{10.34}
\aligned
\frac{4}{\lambda(t)^{3}} \int \psi^{2}(\frac{\eta_{1} x}{R})(\frac{1}{2}  (P_{\leq k + 9} \epsilon_{2x}(t,\frac{x}{\lambda(t)}))^{2} - \frac{1}{2} P_{\leq k +9} Q(\frac{x}{\lambda(t)})^{4} (P_{\leq k + 9} \epsilon_{2}(t,\frac{x}{\lambda(t)}))^{2}) dx \\
\geq \frac{\delta}{\lambda(t)^{3}} \| \psi(\frac{\eta_{1} x}{R}) (P_{\leq k + 9} \epsilon_{2x}(t, \frac{x}{\lambda(t)})) \|_{L^{2}}^{2} -O(\frac{1}{R} \frac{1}{\lambda(t)^{2}} \| \epsilon \|_{L^{2}}^{2}) - \frac{15 \delta}{\lambda(t)^{2}} \| \epsilon_{2} \|_{L^{2}}^{2}.
\endaligned
\end{equation}
Therefore,
\begin{equation}\label{10.35}
\aligned
2 \int \psi^{2}(\frac{\eta_{1} x}{R}) |P_{\leq k + 9} u_{x}|^{2} dx - \frac{2}{3} \int \psi^{2}(\frac{\eta_{1} x}{R}) |P_{\leq k + 9} u|^{6} dx 
\geq \frac{3}{2 \lambda(t)^{2}} \| \epsilon \|_{L^{2}}^{2} + \frac{\delta}{\lambda(t)^{3}} \| \psi(\frac{\eta_{1} x}{R}) P_{\leq k + 9} \epsilon_{x}(t, \frac{x}{\lambda(t)}) \|_{L^{2}}^{2} \\
- \frac{4}{\lambda(t)^{3}} \int \psi^{2}(\frac{\eta_{1} x}{R}) (\frac{10}{3} P_{\leq k + 9} Q(\frac{x}{\lambda(t)})^{3} P_{\leq k + 9} \epsilon(t,\frac{x}{\lambda(t)})^{3} + \frac{5}{2} P_{\leq k + 9} Q(\frac{x}{\lambda(t)})^{2} P_{\leq k + 9} \epsilon(t, \frac{x}{\lambda(t)})^{4} dx \\ + \frac{1}{\lambda(t)^{3}} \int P_{\leq k + 9} Q(\frac{x}{\lambda(t)}) P_{\leq k + 9} \epsilon(t, \frac{x}{\lambda(t)})^{5} + \frac{1}{6} P_{\leq k + 9} \epsilon(t, \frac{x}{\lambda(t)})^{6}) dx - O(\frac{1}{\lambda(t)^{2} T^{11}}) - O(\frac{1}{R \lambda(t)^{2}} \| \epsilon \|_{L^{2}}^{2}).
\endaligned
\end{equation}

Now, by the fundamental theorem of calculus and the product rule, for any $x \in \mathbb{R}$,
\begin{equation}\label{10.36}
\aligned
\frac{1}{\lambda(t)} \psi(\frac{\eta_{1} x}{R}) |P_{\leq k + 9} \epsilon(t, \frac{x}{\lambda(t)})|^{2} \lesssim \frac{1}{\lambda(t)} \int |\partial_{x} (\psi(\frac{\eta_{1} x}{R}) |P_{\leq k + 9} \epsilon(t, \frac{x}{\lambda(t)})|^{2})| dx \\
 \lesssim \frac{1}{\lambda(t)^{3/2}} \| \epsilon \|_{L^{2}(\mathbb{R})} \| \psi(\frac{\eta_{1} x}{R}) P_{\leq k + 9} \epsilon_{x}(t, \frac{x}{\lambda(t)}) \|_{L^{2}} + \frac{\eta_{1}}{R} \| \epsilon \|_{L^{2}}^{2}.
\endaligned
\end{equation}
Therefore, by H{\"o}lder's inequality, the fact that $\| \epsilon \|_{L^{2}} \leq \eta_{\ast}$, and the fact that $\frac{1}{\eta_{1}} \leq \lambda(t) \leq \frac{1}{\eta_{1}} T^{1/100}$, and $R = T^{1/25}$,
\begin{equation}\label{10.37}
\aligned
\frac{1}{\lambda(t)^{3}} \int \psi(\frac{\eta_{1} x}{R})^{2} |P_{\leq k + 9} \epsilon(t,\frac{x}{\lambda(t)})|^{6} dx \lesssim \frac{1}{\lambda(t)^{3}} \| \psi(\frac{\eta_{1} x}{R}) P_{\leq k + 9} \epsilon_{x}(t, \frac{x}{\lambda(t)}) \|_{L^{2}}^{2} \| \epsilon \|_{L^{2}}^{4} + \frac{\eta_{1}^{2}}{R^{2}} \| \epsilon \|_{L^{2}}^{6} \\
\lesssim \frac{\eta_{\ast}^{4}}{\lambda(t)^{3}} \| \psi(\frac{\eta_{1} x}{R}) P_{\leq k + 9} \epsilon_{x}(t, \frac{x}{\lambda(t)}) \|_{L^{2}}^{2} + \frac{\eta_{\ast}^{4}}{R \lambda(t)^{2}} \| \epsilon \|_{L^{2}}^{2}.
\endaligned
\end{equation}

Next, by H{\"o}lder's inequality and the Cauchy--Schwarz inequality, for $j = 3, 4, 5$,
\begin{equation}\label{10.38}
\aligned
\frac{1}{\lambda(t)^{3}} \int \psi(\frac{\eta_{1} x}{R})^{2} |P_{\leq k + 9} \epsilon(t,\frac{x}{\lambda(t)})|^{j} Q(\frac{x}{\lambda(t)})^{6 - j} dx \\ \lesssim \frac{1}{\lambda(t)^{\frac{6 - j}{2}}} (\frac{1}{\lambda(t)^{3}} \int \psi(\frac{\eta_{1} x}{R})^{2} |P_{\leq k + 9} \epsilon(t,\frac{x}{\lambda(t)})|^{6} dx)^{\frac{j - 2}{4}} \| \epsilon \|_{L^{2}}^{\frac{6 - j}{2}} \\
\lesssim \frac{\eta_{\ast}}{\lambda(t)^{2}} \| \epsilon \|_{L^{2}}^{2} + \frac{\eta_{\ast}}{\lambda(t)^{3}} \| \psi(\frac{\eta_{1} x}{R}) P_{\leq k + 9} \epsilon_{x}(t, \frac{x}{\lambda(t)}) \|_{L^{2}}^{2}.
\endaligned
\end{equation}
Therefore, for $\eta_{\ast} \ll \delta$ sufficiently small and $T$ sufficiently large,
\begin{equation}\label{10.39}
\aligned
2 \int \psi^{2}(\frac{\eta_{1} x}{R}) |P_{\leq k + 9} u_{x}|^{2} dx - \frac{2}{3} \int \psi^{2}(\frac{\eta_{1} x}{R}) |P_{\leq k + 9} u|^{6} dx 
\geq \frac{1}{ \lambda(t)^{2}} \| \epsilon \|_{L^{2}}^{2}  - O(\frac{1}{\lambda(t)^{2} T^{11}}).
\endaligned
\end{equation}
Plugging $(\ref{10.39})$ into $(\ref{10.11})$, integrating in time, and using the fact that $2^{3k} = \eta_{1}^{-2} T$, for $T(\eta_{1})$ sufficiently large, the term
\begin{equation}\label{10.40}
O(\frac{2^{2k} T^{1/20}}{T} \int_{a}^{b} \| \epsilon(t) \|_{L^{2}}^{2} \lambda(t)^{-2} dt)
\end{equation}
can be absorbed into the integral of the first term on the right hand side of $(\ref{10.39})$. Since
\begin{equation}\label{10.41}
\int_{J} \lambda(t)^{-2} dt = T,
\end{equation}
the proof of Theorem $\ref{t10.1}$ is complete.
\end{proof}

Since both the left and right hand sides of $(\ref{10.1.1})$ are scale invariant, the same argument also holds for an interval $J$ where
\begin{equation}\label{10.41.1}
A \leq \lambda(t) \leq A T^{1/100},
\end{equation}
for any $A > 0$.

\begin{corollary}\label{c10.2}
Let $J = [a, b]$ be an interval where $(\ref{10.1})$ holds for some $T$ sufficiently large, and $(\ref{10.41.1})$ also holds. Then $(\ref{10.1.1})$ holds.
\end{corollary}

\section{An $L_{s}^{p}$ bound on $\| \epsilon(s) \|_{L^{2}}$ when $p > 1$}
Transitioning to $s$ variables, under the change of variables $(\ref{7.5})$, Theorem $\ref{t10.1}$ and Corollary $\ref{c10.2}$ imply that if $[a, a + T] \subset [0, \infty)$ is an interval on which
\begin{equation}\label{11.1}
\frac{\sup_{s \in [a, a + T]} \lambda(s)}{\inf_{s \in [a, a + T]} \lambda(s)} \leq T^{1/100},
\end{equation}
then
\begin{equation}\label{11.2}
\int_{a}^{a + T} \| \epsilon(s) \|_{L^{2}}^{2} ds \leq 3(\epsilon(a), \frac{1}{2} Q + x Q_{x})_{L^{2}} - 3 (\epsilon(a + T), \frac{1}{2} Q + x Q_{x})_{L^{2}} + O(\frac{1}{T^{9}}).
\end{equation}

Theorem $\ref{t10.1}$ implies good $L_{s}^{p}$ integrability bounds on $\| \epsilon(s) \|_{L^{2}}$ under $(\ref{3.13})$, which is equivalent to 
\begin{equation}\label{11.7}
\sup_{s \in [0, \infty)} \| \epsilon(s) \|_{L^{2}} \leq \eta_{\ast}.
\end{equation}
\begin{theorem}\label{t7.1}
Let $u$ be a symmetric solution to $(\ref{1.1})$ that satisfies $\| u \|_{L^{2}} = \| Q \|_{L^{2}}$, and suppose
\begin{equation}\label{7.1}
\sup_{s \in [0, \infty)} \| \epsilon(s) \|_{L^{2}} \leq \eta_{\ast},
\end{equation}
and $\| \epsilon(0) \|_{L^{2}} = \eta_{\ast}$. Then
\begin{equation}\label{7.2}
\int_{0}^{\infty} \| \epsilon(s) \|_{L^{2}}^{2} ds \lesssim \eta_{\ast},
\end{equation}
with implicit constant independent of $\eta_{\ast}$ when $\eta_{\ast} \ll 1$ is sufficiently small.

Furthermore, for any $j \in \mathbb{Z}_{\geq 0}$, let
\begin{equation}\label{7.3}
s_{j} = \inf \{ s \in [0, \infty) : \| \epsilon(s) \|_{L^{2}} = 2^{-j} \eta_{\ast} \}.
\end{equation}
By definition, $s_{0} = 0$, and the continuity of $\| \epsilon(s) \|_{L^{2}}$ combined with Theorem $\ref{t2.1}$ implies that such an $s_{j}$ exists for any $j > 0$. Then,
\begin{equation}\label{7.4}
\int_{s_{j}}^{\infty} \| \epsilon(s) \|_{L^{2}}^{2} ds \lesssim 2^{-j} \eta_{\ast},
\end{equation}
for each $j$, with implicit constant independent of $\eta_{\ast}$ and $j \geq 0$.
\end{theorem}

\begin{proof}
Set $T_{\ast} = \frac{1}{\eta_{\ast}}$ and suppose that $T_{\ast}$ is sufficiently large such that Theorem $\ref{t10.1}$ holds. Then by $(\ref{7.14})$ and $(\ref{7.1})$, for any $s' \geq 0$,
\begin{equation}\label{7.16}
|\sup_{s \in [s', s' + T_{\ast}]} \ln(\lambda(s)) - \inf_{s \in [s', s' + T_{\ast}]} \ln(\lambda(s))| \lesssim 1,
\end{equation}
with implicit constant independent of $s' \geq 0$. Let $J$ be the largest dyadic integer that satisfies
\begin{equation}
J = 2^{j_{\ast}} \leq -\ln(\eta_{\ast})^{1/2}.
\end{equation}
By $(\ref{7.16})$ and the triangle inequality,
\begin{equation}\label{7.17}
\aligned
|\sup_{s \in [s', s' + J T_{\ast}]} \ln(\lambda(s)) - \inf_{s \in [s', s' + J T_{\ast}]} \ln(\lambda(s))| \lesssim J,
\endaligned
\end{equation}
and therefore,
\begin{equation}\label{7.18}
\frac{\sup_{s \in [s', s' + 3 J T^{\ast}]} \lambda(s)}{\inf_{s \in [s', s' + 3JT^{\ast}]} \lambda(s)} \lesssim T_{\ast}^{1/100}.
\end{equation}
Therefore, Theorem $\ref{t10.1}$ may be utilized on $[s', s' + J T_{\ast}]$. In particular, for any $s' \geq 0$,
\begin{equation}\label{7.19}
\int_{s'}^{s' + J T_{\ast}} \| \epsilon(s) \|_{L^{2}}^{2} ds \lesssim \| \epsilon(s') \|_{L^{2}} + \| \epsilon(s' + J T_{\ast}) \|_{L^{2}} + O(\frac{1}{J^{9} T_{\ast}^{9}}).
\end{equation}
In fact, if $s' > J T_{\ast}$, then by $(\ref{7.18})$,
\begin{equation}\label{7.20}
\int_{s'}^{s' + J T_{\ast}} \| \epsilon(s) \|_{L^{2}}^{2} ds \lesssim \inf_{s \in [s' - J T_{\ast}, s']} \| \epsilon(s) \|_{L^{2}} + \inf_{s \in [s' + J T_{\ast}, s' + 2J T_{\ast}]} \| \epsilon(s) \|_{L^{2}} + O(\frac{1}{J^{9} T_{\ast}^{9}}).
\end{equation}
In particular, for a fixed $s' \geq 0$,
\begin{equation}\label{7.21}
\sup_{a > 0} \int_{s' + a J T_{\ast}}^{s' + (a + 1) J T_{\ast}} \| \epsilon(s) \|_{L^{2}}^{2} \lesssim \frac{1}{J^{1/2} T_{\ast}^{1/2}} (\sup_{a \geq 0} \int_{s' + a J T_{\ast}}^{s' + (a + 1) J T_{\ast}} \| \epsilon(s) \|_{L^{2}}^{2} ds)^{1/2} + O(\frac{1}{J^{9} T_{\ast}^{9}}).
\end{equation}
Meanwhile, when $a = 0$,
\begin{equation}\label{7.22}
 \int_{s'}^{s' + J T_{\ast}} \| \epsilon(s) \|_{L^{2}}^{2} \lesssim \| \epsilon(s') \|_{L^{2}} + \frac{1}{J^{1/2} T_{\ast}^{1/2}} (\sup_{a \geq 0} \int_{s' + a J T_{\ast}}^{s' + (a + 1) J T_{\ast}} \| \epsilon(s) \|_{L^{2}}^{2} ds)^{1/2} + O(\frac{1}{J^{9} T_{\ast}^{9}}).
\end{equation}
Therefore, taking $s' = s_{j_{\ast}}$,
\begin{equation}\label{7.23}
\sup_{a \geq 0} \int_{s_{j_{\ast}} + a J T_{\ast}}^{s_{j_{\ast}} + (a + 1) J T_{\ast}} \| \epsilon(s) \|_{L^{2}}^{2} ds \lesssim 2^{-j_{\ast}} \eta_{\ast} + O(2^{-9 j_{\ast}} \eta_{\ast}^{9}).
\end{equation}
Then by the triangle inequality,
\begin{equation}\label{7.24}
\sup_{s' \geq s_{j_{\ast}}} \int_{s'}^{s' + J T_{\ast}} \| \epsilon(s) \|_{L^{2}}^{2} ds \lesssim 2^{-j_{\ast}} \eta_{\ast},
\end{equation}
and by H{\"o}lder's inequality,
\begin{equation}\label{7.24.1}
\sup_{s' \geq s_{j_{\ast}}} \int_{s'}^{s' + J T_{\ast}} \| \epsilon(s) \|_{L^{2}} ds \lesssim 1.
\end{equation}

In fact, arguing by induction, there exists a constant $C < \infty$ such that,
\begin{equation}\label{7.25}
\sup_{s' \geq s_{nj_{\ast}}} \int_{s'}^{s' + J^{n} T_{\ast}} \| \epsilon(s) \|_{L^{2}} ds \leq C,
\end{equation}
for some $n > 0$ implies that
\begin{equation}\label{7.26}
\sup_{s' \geq s_{(n + 1) j_{\ast}}} \int_{s'}^{s' + J^{n + 1} T_{\ast}} \| \epsilon(s) \|_{L^{2}}^{2} ds \leq C J^{-(n + 1)} T_{\ast}^{-1},
\end{equation}
and by H{\"o}lder's inequality,
\begin{equation}\label{7.26.1}
\sup_{s' \geq s_{(n + 1) j_{\ast}}} \int_{s'}^{s' + J^{n + 1} T_{\ast}} \| \epsilon(s) \|_{L^{2}} ds \leq C^{1/2}.
\end{equation}
Therefore, $(\ref{7.25})$ holds for any integer $n > 0$.

Now take any $j \in \mathbb{Z}$ and suppose $n j_{\ast} < j \leq (n + 1) j_{\ast}$. Then by $(\ref{7.25})$,
\begin{equation}\label{7.27}
\sup_{a \geq 0} \int_{s_{j} + a J^{n + 1} T_{\ast}}^{s_{j} + (a + 1) J^{n + 1} T_{\ast}} \| \epsilon(s) \|_{L^{2}} ds \lesssim J.
\end{equation}
Therefore, as in $(\ref{7.23})$,
\begin{equation}\label{7.28}
\sup_{a \geq 0} \int_{s_{j} + a J^{n + 1} T_{\ast}}^{s_{j} + (a + 1) J^{n + 1} T_{\ast}} \| \epsilon(s) \|_{L^{2}}^{2} ds \lesssim 2^{-j} \eta_{\ast},
\end{equation}
and therefore by H{\"o}lder's inequality, for any $s' \geq s_{j}$,
\begin{equation}\label{7.29}
\sup_{s' \geq s_{j}} \int_{s'}^{s' + 2^{j} T_{\ast}} \| \epsilon(s) \|_{L^{2}} ds \lesssim 1,
\end{equation}
with bound independent of $j$. Then by the triangle inequality, $(\ref{7.18})$ holds for the interval $[s', s' + 3 \cdot 2^{j} J T_{\ast}]$, and by $(\ref{7.19})$--$(\ref{7.22})$,
\begin{equation}\label{7.24.1}
\int_{s_{j}}^{s_{j} + 2^{j} J T_{\ast}} \| \epsilon(s) \|_{L^{2}}^{2} \lesssim 2^{-j} \eta_{\ast},
\end{equation}
and therefore, by the mean value theorem,
\begin{equation}\label{7.30}
\inf_{s \in [s_{j}, s_{j} + 2^{j} J T_{\ast}]} \| \epsilon(s) \|_{L^{2}} \lesssim 2^{-j} \eta_{\ast} J^{-1/2},
\end{equation}
which implies
\begin{equation}\label{7.31}
s_{j + 1} \in [s_{j}, s_{j} + 2^{j} J T_{\ast}].
\end{equation}
Therefore, by $(\ref{7.24.1})$ and H{\"o}lder's inequality,
\begin{equation}\label{7.32}
\int_{s_{j}}^{s_{j + 1}} \| \epsilon(s) \|_{L^{2}}^{2} ds \lesssim 2^{-j} \eta_{\ast}, \qquad \text{and} \qquad \int_{s_{j}}^{s_{j + 1}} \| \epsilon(s) \|_{L^{2}} ds \lesssim 1,
\end{equation}
with constant independent of $j$. Summing in $j$ gives $(\ref{7.2})$ and $(\ref{7.4})$.
\end{proof}

Now, by $(\ref{3.31})$,
\begin{equation}\label{7.6}
\| \epsilon(s') \|_{L^{2}} \sim \| \epsilon(s) \|_{L^{2}},
\end{equation}
for any $s' \in [s, s + 1]$, so $(\ref{7.2})$ implies
\begin{equation}\label{7.7}
\lim_{s \rightarrow \infty} \| \epsilon(s) \|_{L^{2}} = 0.
\end{equation}
Next, by definition of $s_{j}$, $(\ref{7.32})$ implies
\begin{equation}\label{7.8}
\int_{s_{j}}^{s_{j + 1}} \| \epsilon(s) \|_{L^{2}} ds \lesssim 1,
\end{equation}
and for any $1 < p < \infty$,
\begin{equation}\label{7.8.1}
(\int_{s_{j}}^{s_{j + 1}} \| \epsilon(s) \|_{L^{2}}^{p} ds) \lesssim \eta_{\ast}^{p - 1} 2^{-j(p - 1)}, 
\end{equation}
which implies that $\| \epsilon(s) \|_{L^{2}}$ belongs to $L_{s}^{p}$ for any $p > 1$, but not $L_{s}^{1}$.

Comparing $(\ref{7.8.1})$ to the pseudoconformal transformation of the soliton, $(\ref{1.15})$, for $0 < t < 1$,
\begin{equation}\label{7.8.2}
\lambda(t) \sim t, \qquad \text{and} \qquad \| \epsilon(t) \|_{L^{2}} \sim t,
\end{equation}
so
\begin{equation}\label{7.8.3}
\int_{0}^{1} \| \epsilon(t) \|_{L^{2}} \lambda(t)^{-2} dt = \infty,
\end{equation}
but for any $p > 1$,
\begin{equation}\label{7.8.4}
\int_{0}^{1} \| \epsilon(t) \|_{L^{2}}^{p} \lambda(t)^{-2} dt < \infty.
\end{equation}
For the soliton, $\epsilon(s) \equiv 0$ for any $s \in \mathbb{R}$, so obviously, $\epsilon \in L_{s}^{p}$ for $1 \leq p \leq \infty$.

\section{Monotonicity of $\lambda$}
Next, using a virial identity from \cite{merle2005blow}, it is possible to show that $\lambda(s)$ is an approximately monotone decreasing function.
\begin{theorem}\label{t8.2}
For any $s \geq 0$, let
\begin{equation}\label{14.0}
\tilde{\lambda}(s) = \inf_{\tau \in [0, s]} \lambda(\tau).
\end{equation}
Then for any $s \geq 0$,
\begin{equation}\label{8.0}
1 \leq \frac{\lambda(s)}{\tilde{\lambda}(s)} \leq 3.
\end{equation}

\end{theorem}

\begin{proof}
Suppose there exist $0 \leq s_{-} \leq s_{+} < \infty$ satisfying
\begin{equation}\label{8.19}
\frac{\lambda(s_{+})}{\lambda(s_{-})} = e.
\end{equation}
Then we can show that $u$ is a soliton solution to $(\ref{1.1})$, which is a contradiction, since $\lambda(s)$ is constant in that case.

The proof that $(\ref{8.19})$ implies that $u$ is a soliton uses a virial identity from \cite{merle2005blow}. Using $(\ref{7.11})$, compute
\begin{equation}\label{8.41}
\frac{d}{ds} (\epsilon, y^{2} Q) + \frac{\lambda_{s}}{\lambda} \| y Q \|_{L^{2}}^{2} + 4 (\frac{Q}{2} + y Q_{y}, \epsilon_{2})_{L^{2}} = O(|\gamma_{s} + 1| \| \epsilon \|_{L^{2}}) + O(|\frac{\lambda_{s}}{\lambda}| \| \epsilon \|_{L^{2}}) + O(\| \epsilon \|_{L^{2}}^{2}) + O(\| \epsilon \|_{L^{2}} \| \epsilon \|_{L^{8}}^{4}).
\end{equation}
Indeed, by direct computation,
\begin{equation}
\partial_{xx}(x^{2} Q) + Q^{4} (x^{2} Q) - x^{2} Q = 4(\frac{Q}{2} + x Q_{x}).
\end{equation}
Then by $(\ref{7.14})$, $(\ref{7.15})$, $(\ref{7.2})$, and the fundamental theorem of calculus,
\begin{equation}\label{8.42}
\| y Q \|_{L^{2}}^{2} + 4 \int_{s_{-}}^{s_{+}} (\epsilon_{2}, \frac{Q}{2} + x Q_{x})_{L^{2}} = O(\eta_{\ast}).
\end{equation}
Therefore, there exists $s' \in [s_{-}, s_{+}]$ such that
\begin{equation}\label{8.43}
(\epsilon_{2}, \frac{Q}{2} + x Q_{x})_{L^{2}} < 0.
\end{equation}
Since $s' \geq 0$, there exists some $j \geq 0$ such that $s_{j} \leq s' + T_{\ast} < s_{j + 1}$. Using the proof of Theorem $\ref{t7.1}$, in particular $(\ref{7.32})$,
\begin{equation}\label{8.44}
\int_{s'}^{s_{j + 1 + J}} |\frac{\lambda_{s}}{\lambda}| ds \lesssim J.
\end{equation}
Then by Theorem $\ref{t10.1}$, $(\ref{8.43})$ implies
\begin{equation}\label{8.45}
\int_{s'}^{s_{j + 1 + J}} \| \epsilon(s) \|_{L^{2}}^{2} ds \lesssim 2^{-(j + 1 + J)} \eta_{\ast},
\end{equation}
and therefore by definition of $s_{j + 1 + J}$,
\begin{equation}\label{8.46}
\int_{s'}^{s_{j + 1 + J}} \| \epsilon(s) \|_{L^{2}} ds \lesssim 1.
\end{equation}
Then, $(\ref{8.46})$ implies that $(\ref{8.44})$ holds on the interval $[s', s_{j + 1 + 2J}]$, and arguing by induction, for any $k \geq 1$,
\begin{equation}\label{8.47}
\int_{s'}^{s_{j + k}} \| \epsilon(s) \|_{L^{2}}^{2} ds \lesssim 2^{-j - k} \eta_{\ast},
\end{equation}
and
\begin{equation}\label{8.48}
\int_{s'}^{s_{j + k}} \| \epsilon(s) \|_{L^{2}} ds \lesssim 1,
\end{equation}
with implicit constant independent of $k$. Taking $k \rightarrow \infty$,
\begin{equation}\label{8.49}
\int_{s'}^{\infty} \| \epsilon(s) \|_{L^{2}}^{2} ds = 0,
\end{equation}
which implies that $\epsilon(s) = 0$ for all $s \geq s'$. Therefore,
\begin{equation}\label{8.50}
u_{0} = \lambda^{1/2} Q(\lambda x) e^{i \gamma}
\end{equation}
for some $\gamma \in \mathbb{R}$ and $\lambda > 0$, which proves that $u$ is a soliton solution.
\end{proof}

\section{Almost monotone $\lambda(t)$}
The almost monotonicity of $\lambda$ implies that when $\sup(I) = \infty$, $u$ is equal to a soliton solution, and when $\sup(I) < \infty$, $u$ is the pseudoconformal transformation of the soliton solution.

\begin{theorem}\label{t14.0}
If $u$ satisfies the conditions of Theorem $\ref{t3.1}$, blows up forward in time, and
\begin{equation}\label{14.1}
\sup(I) = \infty,
\end{equation}
then $u$ is equal to a soliton solution.
\end{theorem}
\begin{proof}
For any integer $k \geq 0$, let
\begin{equation}\label{14.4}
I(k) = \{ s \geq 0 : 2^{-k + 2} \leq \tilde{\lambda}(s) \leq 2^{-k + 3} \}.
\end{equation}
Then by $(\ref{8.0})$,
\begin{equation}\label{14.5}
2^{-k} \leq \lambda(s) \leq 2^{-k + 3},
\end{equation}
for all $s \in I(k)$. By $(\ref{7.5})$, the fact that $\sup(I) = \infty$ implies that
\begin{equation}\label{14.7}
\sum 2^{-2k} |I(k)| = \infty.
\end{equation}
Therefore, there exists a sequence $k_{n} \nearrow \infty$ such that
\begin{equation}\label{14.8}
|I(k_{n})| 2^{-2k_{n}} \geq \frac{1}{k_{n}^{2}},
\end{equation}
and that $|I(k_{n})| \geq |I(k)|$ for all $k \leq k_{n}$.

\begin{lemma}\label{l14.1}
For $n$ sufficiently large, there exists $s_{n} \in I(k_{n})$ such that
\begin{equation}\label{14.10}
\| \epsilon(s_{n}) \|_{L^{2}} \lesssim k_{n}^{2} 2^{-2k_{n}}.
\end{equation}
\end{lemma}
\begin{proof}
Let $I(k_{n}) = [a_{n}, b_{n}]$. By Theorem $\ref{t10.1}$, for $n$ sufficiently large,
\begin{equation}\label{14.2}
\int_{I(k_{n})} \| \epsilon(s) \|_{L^{2}}^{2} ds \lesssim \eta_{\ast} + 2^{-18 k_{n}} k_{n}^{18} \lesssim \eta_{\ast}.
\end{equation}
Then, using the virial identity in $(\ref{8.41})$,
\begin{equation}\label{14.3}
\int_{a_{n}}^{\frac{3 a_{n} + b_{n}}{4}} (\epsilon_{2}, \frac{Q}{2} + x Q_{x})_{L^{2}} ds = O(\eta_{\ast}) + O(1).
\end{equation}
Therefore, by the mean value theorem, there exists $s_{n}^{-} \in [a_{n}, \frac{3 a_{n} + b_{n}}{4}]$ such that
\begin{equation}\label{14.8}
|(\epsilon_{2}(s_{n}^{-}), \frac{Q}{2} + x Q_{x})_{L^{2}}| \lesssim \frac{1}{|I(k_{n})|}.
\end{equation}
By a similar calculation, there exists $s_{n}^{+} \in [\frac{a_{n} + 3 b_{n}}{4}, b_{n}]$ such that
\begin{equation}\label{14.8.1}
|(\epsilon_{2}(s_{n}^{+}), \frac{Q}{2} + x Q_{x})_{L^{2}}| \lesssim \frac{1}{|I(k_{n})|}.
\end{equation}
Plugging $(\ref{14.8})$ and $(\ref{14.8.1})$ into Theorem $\ref{t10.1}$,
\begin{equation}\label{14.8.2}
\int_{s_{n}^{-}}^{s_{n}^{+}} \| \epsilon(s) \|_{L^{2}}^{2} ds \lesssim \frac{1}{|I(k_{n})|}.
\end{equation}
Then by the mean value theorem there exists $s_{n} \in [s_{n}^{-}, s_{n}^{+}]$ such that
\begin{equation}\label{14.8.3}
\| \epsilon(s_{n}) \|_{L^{2}}^{2} \lesssim \frac{1}{|I(k_{n})|^{2}}.
\end{equation}
Since $|I(k_{n})| \geq 2^{2k_{n}} k_{n}^{-2}$, the proof of Lemma $\ref{l14.1}$ is complete.
\end{proof}

Returning to the proof of Theorem $\ref{t14.0}$, let $m$ be the smallest integer such that
\begin{equation}\label{14.11}
\frac{2^{2k_{n}}}{k_{n}^{2}} 2^{m} \geq |I(k_{n})|.
\end{equation}
Since $|I(k)| \leq |I(k_{n})|$ for all $0 \leq k \leq k_{n}$, $(\ref{14.11})$ implies that
\begin{equation}\label{14.12}
|s_{n}| \leq 2^{2k_{n} + m + 1}.
\end{equation}
Let $r_{n}$ be the smallest integer that satisfies
\begin{equation}\label{14.13}
2^{\frac{2k_{n} + m + 1}{3}} 2^{k_{n}} \frac{1}{\eta_{1}} \leq 2^{r_{n}}.
\end{equation}
Since $\lambda(s) \geq 2^{-k_{n}}$ for all $s \in [0, s_{n}]$, setting $t_{n} = s^{-1}(s_{n})$, rescaling so that $\lambda(t) \geq \frac{1}{\eta_{1}}$ on $[0, 2^{2k_{n}} \eta_{1}^{-2} t_{n}]$, applying Theorem $\ref{t10.1}$, then rescaling back,
\begin{equation}\label{14.14}
\| P_{\geq r_{n}} u \|_{U_{\Delta}^{2}([0, t_{n}] \times \mathbb{R})} \lesssim \eta_{\ast}.
\end{equation}
Arguing by induction on frequency, and using $(\ref{6.60.1})$ and the preceding computations,
\begin{equation}\label{14.15}
\| P_{\geq r_{n} + \frac{k_{n}}{4} + \frac{m}{4}} u \|_{U_{\Delta}^{2}([0, t_{n}] \times \mathbb{R})} \lesssim k_{n}^{2} 2^{-2k_{n}} 2^{-m}.
\end{equation}
Then using the computations in $(\ref{8.24})$--$(\ref{8.29})$,
\begin{equation}\label{14.16}
E(P_{\leq r_{n} + \frac{k_{n}}{4} + \frac{m}{4}} u(t_{n})) \lesssim (k_{n}^{2} 2^{-2k_{n}} 2^{-m} 2^{r_{n} + \frac{k_{n}}{4} + \frac{m}{4}})^{2} \sim (k_{n}^{2} 2^{-\frac{k_{n}}{12} - \frac{5m}{12}} \eta_{1}^{-1})^{2}.
\end{equation}
Next, following the computations in the proof of Theorem $\ref{t8.3}$, and using $(\ref{14.15})$,
\begin{equation}\label{14.17}
\sup_{t \in [0, t_{n}]} E(P_{\leq r_{n} + \frac{k_{n}}{4} + \frac{m}{4}} u(t)) \lesssim (k_{n}^{2} 2^{-\frac{k_{n}}{12} - \frac{5m}{12}} \eta_{1}^{-1})^{2}.
\end{equation}
Since $m \geq 0$ for any $n$, taking $n \rightarrow \infty$ implies that $E(u_{0}) = 0$. Then by the Gagliardo--Nirenberg inequality, $u_{0}$ is a soliton.
\end{proof}

It only remains to show that in the case that $\sup(I) < \infty$, $u$ is a pseudoconformal transformation of the soliton. If one could show that the energy of $u_{0}$ is finite, then this fact would follow directly from the result of \cite{merle1993determination}. Similarly, if one could generalize the result of \cite{merle1993determination} to data that need not have finite energy, then the proof would also be complete.

We do not quite prove this fact. Instead, suppose without loss of generality that $\sup(I) = 0$, and
\begin{equation}\label{14.18}
\sup_{-1 < t < 0} \| \epsilon(t) \|_{L^{2}} \leq \eta_{\ast}.
\end{equation}
Then decompose
\begin{equation}\label{14.19}
u(t,x) = \frac{e^{-i \gamma(t)}}{\lambda(t)^{1/2}} Q(\frac{x}{\lambda(t)}) + \frac{e^{-i \gamma(t)}}{\lambda(t)^{1/2}} \epsilon(t,\frac{x}{\lambda(t)}).
\end{equation}
Then apply the pseudoconformal transformation to $u(t,x)$. For $-\infty < t < -1$, let
\begin{equation}\label{14.20}
v(t,x) = \frac{1}{t^{1/2}} \overline{u(\frac{1}{t}, \frac{x}{t})} e^{i x^{2}/4t} = \frac{1}{t^{1/2}} \frac{e^{i \gamma(1/t)}}{\lambda(1/t)^{1/2}} Q(\frac{x}{t \lambda(1/t)}) e^{i x^{2}/4t} + \frac{1}{t^{1/2}} \frac{e^{i \gamma(1/t)}}{\lambda(1/t)^{1/2}} \overline{\epsilon(\frac{1}{t}, \frac{x}{t \lambda(1/t)})} e^{i x^{2}/4t}.
\end{equation}
Since the $L^{2}$ norm is preserved by the pseudoconformal transformation,
\begin{equation}\label{14.21}
\aligned
\lim_{t \searrow -\infty} \| \frac{1}{t^{1/2}} \frac{e^{i \gamma(1/t)}}{\lambda(1/t)^{1/2}} \overline{\epsilon(\frac{1}{t}, \frac{x}{t \lambda(1/t)})} e^{i x^{2}/4t} \|_{L^{2}} = 0, \qquad \text{and} \\ 
\qquad \sup_{-\infty < t < -1}  \| \frac{1}{t^{1/2}} \frac{e^{i \gamma(1/t)}}{\lambda(1/t)^{1/2}} \overline{\epsilon(\frac{1}{t}, \frac{x}{t \lambda(1/t)})} e^{i x^{2}/4t} \|_{L^{2}} \leq \eta_{\ast}.
\endaligned
\end{equation}

Since
\begin{equation}\label{14.22}
 \frac{1}{t^{1/2}} \frac{e^{i \gamma(1/t)}}{\lambda(1/t)^{1/2}} Q(\frac{x}{t \lambda(1/t)})
\end{equation}
is in the form of $\frac{e^{i \tilde{\gamma}(t)}}{\tilde{\lambda}(t)^{1/2}} Q(\frac{x}{\tilde{\lambda}(t)})$, it only remains to estimate
\begin{equation}\label{14.23}
 \| \frac{1}{t^{1/2}} \frac{e^{i \gamma(1/t)}}{\lambda(1/t)^{1/2}} Q(\frac{x}{t \lambda(1/t)}) (e^{i x^{2}/4t} - 1) \|_{L^{2}}.
\end{equation}
Once again take $(\ref{14.4})$. As in $(\ref{14.5})$, for any $k \geq 0$, $\lambda(s) \sim 2^{-k}$ for all $s \in I(k)$. Furthermore, by  $(\ref{7.14})$, $\| \epsilon(t) \|_{L^{2}} \rightarrow 0$ as $t \nearrow 0$ implies that there exists a sequence $c_{k} \nearrow \infty$ such that
\begin{equation}\label{14.24}
|I(k)| \geq c_{k}, \qquad \text{for all} \qquad k \geq 0.
\end{equation}
Then by $(\ref{7.5})$, there exists $r(t) \searrow 0$ as $t \nearrow 0$ such that
\begin{equation}\label{14.25}
\lambda(t) \leq t^{1/2} r(t), \qquad \text{so} \qquad \lambda(1/t) \leq t^{-1/2} r(1/t).
\end{equation}
Therefore, since $Q$ is rapidly decreasing,
\begin{equation}\label{14.26}
\lim_{t \searrow -\infty} \| \frac{1}{t^{1/2} \lambda(1/t)^{1/2}} Q(\frac{x}{t \lambda(1/t)}) \frac{x^{2}}{4t} \|_{L^{2}} = 0,
\end{equation}
as well as
\begin{equation}\label{14.27}
\lim_{t \searrow -\infty} \| \frac{1}{t^{1/2} \lambda(1/t)^{1/2}} Q(\frac{x}{t \lambda(1/t)}) (e^{i x^{2}/4t} - 1)\|_{L^{2}} = 0,
\end{equation}
Therefore, by time reversal symmetry, $v$ satisfies the conditions of Theorem $\ref{t3.1}$, and $v$ is a solution that blows up backward in time at $\inf(I) = -\infty$, so therefore, by Theorem $\ref{t14.0}$, $v$ must be a soliton. In particular,
\begin{equation}\label{14.28}
v(t,x) = e^{i \lambda^{2} t} e^{i \theta} \lambda^{1/2} Q(\lambda x) = \frac{1}{t^{1/2}} \overline{u(\frac{1}{t}, \frac{x}{t})} e^{i x^{2}/4t}.
\end{equation}
Doing some algebra,
\begin{equation}\label{14.29}
\overline{u(\frac{1}{t}, \frac{x}{t})} = e^{i \lambda^{2} t} e^{i \theta} e^{-i x^{2}/4t} t^{1/2} \lambda^{1/2} Q(\lambda x),
\end{equation}
so
\begin{equation}\label{14.30}
u(t,x) = e^{-i \lambda^{2}/t} e^{-i \theta} e^{i x^{2}/4t} \frac{1}{t^{1/2}} \lambda^{1/2} Q(\frac{\lambda x}{t}).
\end{equation}
This is clearly the pseudoconformal transformation of a soliton. This finally completes the proof of Theorem $\ref{t3.1}$.

\section{A non-symmetric solution}
When there is no symmetry assumption on $u$, there is no preferred origin, either in space or in frequency. As a result, two additional group actions on a solution $u$ must be accounted for, translation in space,
\begin{equation}\label{15.1}
u(t,x) \mapsto u(t, x - x_{0}), \qquad x_{0} \in \mathbb{R},
\end{equation}
and the Galilean symmetry,
\begin{equation}\label{15.2}
e^{-it \xi_{0}^{2}} e^{ix \xi_{0}} u(t, x - 2 t \xi_{0}), \qquad \xi_{0} \in \mathbb{R}.
\end{equation}
This gives a four parameter family of soliton solutions to $(\ref{1.1})$, given by $(\ref{1.16.1})$. Making the pseudoconformal transformation of $(\ref{1.16.1})$ gives a solution in the form of $(\ref{1.17.1})$.

In this section we prove Theorem $\ref{t1.2}$, that the only non-symmetric blowup solutions to $(\ref{1.1})$ with mass $\| u_{0} \|_{L^{2}}^{2} = \| Q \|_{L^{2}}^{2}$ belong to the family of solitons and pseudoconformal transformation of a soliton. To prove this, we will go through the proof of Theorem $\ref{t1.1}$ in sections two through nine, section by section, generalizing each step to the non-symmetric case. There are several steps for which the argument in the symmetric case has an easy generalization to the non-symmetric case, after accounting for the additional group actions $(\ref{15.1})$ and $(\ref{15.2})$. There are other steps for which the non-symmetric case will require substantially more work.

\subsection{Reductions of a non-symmetric blowup solution}
Using the same arguments that show that Theorem $\ref{t1.1}$ may be reduced to Theorem $\ref{t3.1}$, Theorem $\ref{t1.2}$ may be reduced to
\begin{theorem}\label{t15.1}
Let $0 < \eta_{\ast} \ll 1$ be a small fixed constant to be defined later. If $u$ is a solution to $(\ref{1.1})$ on the maximal interval of existence $I \subset \mathbb{R}$, $\| u_{0} \|_{L^{2}} = \| Q \|_{L^{2}}$, $u$ blows up forward in time, and
\begin{equation}\label{15.3}
\sup_{t \in [0, \sup(I))} \inf_{\lambda, \gamma, \xi_{0}, x_{0}} \| e^{i \gamma} e^{ix \xi_{0}} \lambda^{1/2} u(t, \lambda x + x_{0}) - Q(x) \|_{L^{2}} \leq \eta_{\ast},
\end{equation}
then $u$ is a soliton solution of the form $(\ref{1.16.1})$ or the pseudoconformal transformation of a soliton of the form $(\ref{1.17.1})$.
\end{theorem}

Reducing Theorem $\ref{t1.2}$ to Theorem $\ref{t15.1}$ requires the following generalization of Theorem $\ref{t2.1}$, which was proved in Theorem $2$ of \cite{dodson20202}.

\begin{theorem}\label{t15.2}
Assume that $u$ is a solution to $(\ref{1.1})$ with $\| u_{0} \|_{L^{2}} = \| Q \|_{L^{2}}$ that does not scatter forward in time. Let $(T^{-}(u), T^{+}(u))$ be its lifespan, $T^{-}(u)$ could be $-\infty$ and $T^{+}(u)$ could be $+\infty$. Then there exists a sequence $t_{n} \nearrow T^{+}(u)$ and a family of parameters $\lambda_{n} > 0$, $\xi_{n} \in \mathbb{R}$, $x_{n} \in \mathbb{R}$, and $\gamma_{n} \in \mathbb{R}$ such that
\begin{equation}\label{15.4}
\lambda_{n}^{1/2} e^{ix \xi_{n}} e^{i \gamma_{n}} u(t_{n}, \lambda_{n} x + x_{n}) \rightarrow Q, \qquad \text{in} \qquad L^{2}.
\end{equation}
\end{theorem}

Lemma $\ref{l3.3}$ can easily be generalized to the non-symmetric case, proving that $\| e^{i \gamma} e^{i x \xi_{0}} \lambda^{1/2} u_{0}(\lambda x + x_{0}) - Q \|_{L^{2}}$ attains its infimum on $\gamma \in \mathbb{R}$, $\xi_{0} \in \mathbb{R}$, $x_{0} \in \mathbb{R}$, $\lambda > 0$. Theorem $\ref{t3.2}$ is also easily generalized to the non-symmetric case, showing that the left hand side of $(\ref{15.3})$ is upper semicontinuous in time and continuous in time when small. Therefore, Theorem $\ref{t1.2}$ is easily reduced to Theorem $\ref{t15.1}$ using the same argument that reduced Theorem $\ref{t1.1}$ to Theorem $\ref{t3.1}$.

\subsection{Decomposition of a non-symmetric solution near $Q$}
When a non-symmetric $u$ is close to a soliton, it is possible to make a decomposition of $u$, generalizing Theorem $\ref{t2.3}$ to account 
for the additional group actions in $(\ref{15.1})$ and $(\ref{15.2})$.
\begin{theorem}\label{t15.3}
Take $u \in L^{2}$. There exists $\alpha > 0$ sufficiently small such that if there exist $\lambda_{0} > 0$, $\gamma_{0} \in \mathbb{R}$, $x_{0} \in \mathbb{R}$, $\xi_{0} \in \mathbb{R}$ that satisfy
\begin{equation}\label{15.5}
\| e^{i \gamma_{0}} e^{ix \xi_{0}} \lambda_{0}^{1/2} u(\lambda_{0} x + x_{0}) - Q(x) \|_{L^{2}} \leq \alpha,
\end{equation}
then there exist unique $\lambda > 0$, $\gamma \in \mathbb{R}$, $\tilde{x} \in \mathbb{R}$, $\xi \in \mathbb{R}$ that satisfy
\begin{equation}\label{15.6}
(\epsilon, Q^{3})_{L^{2}} = (\epsilon, i Q^{3})_{L^{2}} = (\epsilon, Q_{x})_{L^{2}} = (\epsilon, i Q_{x})_{L^{2}} = 0,
\end{equation}
where
\begin{equation}\label{15.7}
\epsilon(x) = e^{i \gamma} e^{ix \xi} \lambda^{1/2} u(\lambda x + \tilde{x}) - Q.
\end{equation}
Furthermore,
\begin{equation}\label{15.8}
\| \epsilon \|_{L^{2}} + |\frac{\lambda}{\lambda_{0}} - 1| + |\gamma - \gamma_{0} - \xi_{0}(\tilde{x} - x_{0})| + |\xi - \frac{\lambda}{\lambda_{0}} \xi_{0}| + |\frac{\tilde{x} - x_{0}}{\lambda_{0}}| \lesssim \| e^{i \gamma_{0}} e^{ix \xi_{0}} \lambda_{0}^{1/2} u(\lambda_{0} x + x_{0}) - Q \|_{L^{2}}.
\end{equation}
\end{theorem}
\begin{remark}
Once again, since $e^{i \gamma}$ is $2\pi$-periodic, the $\gamma$ in $(\ref{15.7})$ is unique up to translations by $2 \pi k$ for some integer $k$.
\end{remark}

\begin{proof}
By H{\"o}lder's inequality, if $\epsilon = e^{i \gamma_{0}} e^{ix \xi_{0}} \lambda_{0}^{1/2} u(\lambda_{0} x + x_{0}) - Q(x)$, then
\begin{equation}\label{15.9}
|(\epsilon, Q^{3})_{L^{2}}| + |(\epsilon, Q_{x})_{L^{2}}| + |(\epsilon, i Q^{3})_{L^{2}}| + |(\epsilon, i Q_{x})_{L^{2}}| \lesssim \| e^{i \gamma_{0}} e^{ix \xi_{0}} \lambda_{0}^{1/2} u(\lambda_{0} x + x_{0}) - Q(x) \|_{L^{2}}.
\end{equation}
As in the proof of Theorem $\ref{t2.3}$,
\begin{equation}\label{15.10}
(e^{i \gamma} e^{ix \xi} \lambda^{1/2} u(\lambda x + \tilde{x}) - Q(x), f)_{L^{2}}
\end{equation}
is $C^{1}$ as a function of $\gamma$, $\lambda$, $\tilde{x}$, and $\xi$, when
\begin{equation}\label{15.11}
f \in \{ Q^{3}, i Q^{3}, Q_{x}, i Q_{x} \}.
\end{equation}
Indeed, by H{\"o}lder's inequality and the $L^{2}$-invariance of the scaling symmetry,
\begin{equation}\label{15.12}
\frac{\partial}{\partial \gamma} ( e^{i \gamma} e^{ix \xi} \lambda^{1/2} u(\lambda x + \tilde{x}) - Q(x), f )_{L^{2}} = ( i e^{i \gamma} e^{ix \xi} \lambda^{1/2} u(\lambda x + \tilde{x}), f )_{L^{2}} \lesssim \| u \|_{L^{2}} \| f \|_{L^{2}}.
\end{equation}
Next,
\begin{equation}\label{15.13}
\frac{\partial}{\partial \xi} ( e^{i \gamma} e^{ix \xi} \lambda^{1/2} u(\lambda x + \tilde{x}) - Q(x), f )_{L^{2}} = ( ix e^{i \gamma} e^{ix \xi} \lambda^{1/2} u(\lambda x + \tilde{x}), f )_{L^{2}} \lesssim \| u \|_{L^{2}} \| xf \|_{L^{2}}.
\end{equation}
Since $Q$ and all its derivatives are rapidly decreasing, $xf \in L^{2}$, and $(\ref{15.13})$ is well-defined.

Next, integrating by parts,
\begin{equation}\label{15.14}
\aligned
\frac{\partial}{\partial \lambda} ( e^{i \gamma} e^{ix \xi} \lambda^{1/2} u(\lambda x + \tilde{x}) - Q(x), f )_{L^{2}} = ( \frac{1}{2 \lambda} e^{i \gamma} e^{ix \xi} \lambda^{1/2} u(\lambda x + \tilde{x}) + x e^{i \gamma} e^{ix \xi} \lambda^{1/2} u_{x}(\lambda x + \tilde{x}), f )_{L^{2}} \\
= \frac{1}{2 \lambda} ( e^{i \gamma} e^{ix \xi} \lambda^{1/2} u(\lambda x + \tilde{x}), f)_{L^{2}} - \frac{1}{\lambda} (e^{i \gamma} e^{ix \xi} \lambda^{1/2} u(\lambda x + \tilde{x}), f )_{L^{2}} - \frac{\xi}{\lambda} (i e^{i \gamma} e^{ix \xi} \lambda^{1/2} u(\lambda x + \tilde{x}), x f )_{L^{2}} \\ 
- \frac{1}{\lambda} (e^{i \gamma} e^{ix \xi} \lambda^{1/2} u(\lambda x + \tilde{x}), x f_{x} )_{L^{2}}  \lesssim \frac{1}{\lambda} \| u \|_{L^{2}} \| f \|_{L^{2}} + \frac{1}{\lambda} \| u \|_{L^{2}} \| x f_{x} \|_{L^{2}} + \frac{|\xi|}{\lambda} \| u \|_{L^{2}} \| x f \|_{L^{2}}.
\endaligned
\end{equation}
Similarly,
\begin{equation}\label{15.15}
\aligned
\frac{\partial}{\partial \tilde{x}} ( e^{i \gamma} e^{ix \xi} \lambda^{1/2} u(\lambda x + \tilde{x}) - Q(x), f )_{L^{2}} = ( e^{i \gamma} e^{ix \xi} \lambda^{1/2} u_{x}(\lambda x + \tilde{x}), f )_{L^{2}}
=  -\frac{1}{\lambda} (i \xi e^{i \gamma} e^{ix \xi} \lambda^{1/2} u(\lambda x + \tilde{x}), f)_{L^{2}} \\ - \frac{1}{\lambda} (e^{i \gamma} e^{ix \xi} \lambda^{1/2} u(\lambda x + \tilde{x}), f_{x} )_{L^{2}}  \lesssim \frac{1}{\lambda} \| u \|_{L^{2}} \| f \|_{L^{2}} + \frac{|\xi|}{\lambda} \| u \|_{L^{2}} \| f_{x} \|_{L^{2}}.
\endaligned
\end{equation}
Similar calculations also prove uniform bounds on the Hessians of $(\ref{15.10})$.

Suppose $\lambda_{0} = 1$, $\gamma_{0} = 0$, $x_{0} = 0$, and $\xi_{0} = 0$. Compute
\begin{equation}\label{15.17}
\aligned
\frac{\partial}{\partial \lambda} ( e^{i \gamma} e^{ix \xi} \lambda^{1/2} u(\lambda x + \tilde{x}) - Q(x), Q^{3} )_{L^{2}}|_{\lambda = 1, \gamma = 0, \tilde{x} = 0, \xi = 0, u = Q} = ( \frac{Q}{2} + x Q_{x}, Q^{3} )_{L^{2}} = \frac{1}{4} \| Q \|_{L^{4}}^{4}, \\
\frac{\partial}{\partial \lambda} ( e^{i \gamma} e^{ix \xi} \lambda^{1/2} u(\lambda x + \tilde{x}) - Q(x), i Q^{3} )_{L^{2}}|_{\lambda = 1, \gamma = 0, \tilde{x} = 0, \xi = 0, u = Q} = ( \frac{Q}{2} + x Q_{x}, i Q^{3} )_{L^{2}} = 0, \\
\frac{\partial}{\partial \lambda} ( e^{i \gamma} e^{ix \xi} \lambda^{1/2} u(\lambda x + \tilde{x}) - Q(x), Q^{3} )_{L^{2}}|_{\lambda = 1, \gamma = 0, \tilde{x} = 0, \xi = 0, u = Q} = ( \frac{Q}{2} + x Q_{x}, Q_{x} )_{L^{2}} = 0, \\
\frac{\partial}{\partial \lambda} ( e^{i \gamma} e^{ix \xi} \lambda^{1/2} u(\lambda x + \tilde{x}) - Q(x), Q^{3} )_{L^{2}}|_{\lambda = 1, \gamma = 0, \tilde{x} = 0, \xi = 0, u = Q} = ( \frac{Q}{2} + x Q_{x}, i Q_{x} )_{L^{2}} = 0.
\endaligned
\end{equation}

\begin{equation}\label{15.16}
\aligned
\frac{\partial}{\partial \gamma} (  e^{i \gamma} e^{ix \xi} \lambda^{1/2} u(\lambda x + \tilde{x}) - Q(x), Q^{3} )_{L^{2}}|_{\lambda = 1, \gamma = 0, \tilde{x} = 0, \xi = 0, u = Q} = ( i Q, Q^{3} )_{L^{2}} = 0, \\
\frac{\partial}{\partial \gamma} (  e^{i \gamma} e^{ix \xi} \lambda^{1/2} u(\lambda x + \tilde{x}) - Q(x), i Q^{3} )_{L^{2}}|_{\lambda = 1, \gamma = 0, \tilde{x} = 0, \xi = 0, u = Q} = ( i Q, i Q^{3} )_{L^{2}} = \| Q \|_{L^{4}}^{4}, \\
\frac{\partial}{\partial \gamma} (  e^{i \gamma} e^{ix \xi} \lambda^{1/2} u(\lambda x + \tilde{x}) - Q(x), Q^{3} )_{L^{2}}|_{\lambda = 1, \gamma = 0, \tilde{x} = 0, \xi = 0, u = Q} = ( i Q, Q_{x} )_{L^{2}} = 0, \\
\frac{\partial}{\partial \gamma} (  e^{i \gamma} e^{ix \xi} \lambda^{1/2} u(\lambda x + \tilde{x}) - Q(x), Q^{3} )_{L^{2}}|_{\lambda = 1, \gamma = 0, \tilde{x} = 0, \xi = 0, u = Q} = ( i Q, i Q_{x} )_{L^{2}} = 0.
\endaligned
\end{equation}

\begin{equation}\label{15.18}
\aligned
\frac{\partial}{\partial \tilde{x}} ( e^{i \gamma} e^{ix \xi} \lambda^{1/2} u(\lambda x + \tilde{x}) - Q(x), Q^{3} )_{L^{2}}|_{\lambda = 1, \gamma = 0, \tilde{x} = 0, \xi = 0, u = Q} = ( Q_{x}, Q^{3} )_{L^{2}} = 0, \\
\frac{\partial}{\partial \tilde{x}} ( e^{i \gamma} e^{ix \xi} \lambda^{1/2} u(\lambda x + \tilde{x}) - Q(x), i Q^{3} )_{L^{2}}|_{\lambda = 1, \gamma = 0, \tilde{x} = 0, \xi = 0, u = Q} = ( Q_{x}, i Q^{3} )_{L^{2}} = 0, \\
\frac{\partial}{\partial \tilde{x}} ( e^{i \gamma} e^{ix \xi} \lambda^{1/2} u(\lambda x + \tilde{x}) - Q(x), Q^{3} )_{L^{2}}|_{\lambda = 1, \gamma = 0, \tilde{x} = 0, \xi = 0, u = Q} = ( Q_{x}, Q_{x} )_{L^{2}} = \| Q_{x} \|_{L^{2}}^{2}, \\
\frac{\partial}{\partial \tilde{x}} ( e^{i \gamma} e^{ix \xi} \lambda^{1/2} u(\lambda x + \tilde{x}) - Q(x), Q^{3} )_{L^{2}}|_{\lambda = 1, \gamma = 0, \tilde{x} = 0, \xi = 0, u = Q} = ( Q_{x}, i Q_{x} )_{L^{2}} = 0.
\endaligned
\end{equation}

\begin{equation}\label{15.19}
\aligned
\frac{\partial}{\partial \xi} (  e^{i \gamma} e^{ix \xi} \lambda^{1/2} u(\lambda x + \tilde{x}) - Q(x), Q^{3} )_{L^{2}}|_{\lambda = 1, \gamma = 0, \tilde{x} = 0, \xi = 0, u = Q} = ( ix Q, Q^{3} )_{L^{2}} = 0, \\
\frac{\partial}{\partial \xi} (  e^{i \gamma} e^{ix \xi} \lambda^{1/2} u(\lambda x + \tilde{x}) - Q(x), i Q^{3} )_{L^{2}}|_{\lambda = 1, \gamma = 0, \tilde{x} = 0, \xi = 0, u = Q} = ( i x Q, i Q^{3} )_{L^{2}} = 0, \\
\frac{\partial}{\partial \xi} (  e^{i \gamma} e^{ix \xi} \lambda^{1/2} u(\lambda x + \tilde{x}) - Q(x), Q_{x} )_{L^{2}}|_{\lambda = 1, \gamma = 0, \tilde{x} = 0, \xi = 0, u = Q} = ( i x Q, Q_{x} )_{L^{2}} = 0, \\
\frac{\partial}{\partial \xi} (  e^{i \gamma} e^{ix \xi} \lambda^{1/2} u(\lambda x + \tilde{x}) - Q(x), Q^{3} )_{L^{2}}|_{\lambda = 1, \gamma = 0, \tilde{x} = 0, \xi = 0, u = Q} = ( i x Q, i Q_{x} )_{L^{2}} = -\frac{1}{2} \| Q \|_{L^{2}}^{2}.
\endaligned
\end{equation}

Therefore, by the inverse function theorem, if $\lambda_{0} = 1$, $\gamma_{0} = 0$, $\xi_{0} = 0$, $x_{0} = 0$, there exists $\lambda > 0$, $\gamma \in \mathbb{R}$, $\xi \in \mathbb{R}$, $\tilde{x} \in \mathbb{R}$, satisfying
\begin{equation}\label{15.20}
\| \epsilon \|_{L^{2}} + |\lambda - 1| + |\gamma| + |\xi| + |\tilde{x}| \lesssim \| e^{i \gamma_{0}} e^{ix \xi_{0}} \lambda_{0}^{1/2} u(\lambda_{0} x + x_{0}) - Q \|_{L^{2}}.
\end{equation}
As in $(\ref{2.29.1})$ and $(\ref{2.29.2})$, $\lambda > 0$, $\gamma \in \mathbb{R}$, and $\tilde{x} \in \mathbb{R}$ are unique, and $\gamma \in \mathbb{R}$ is unique in $\mathbb{R} / 2 \pi n$.

For general $\lambda_{0} > 0$, $x_{0} \in \mathbb{R}$, $\xi_{0} \in \mathbb{R}$, and $\gamma_{0} \in \mathbb{R}$, combining $(\ref{15.20})$ with symmetries of $(\ref{1.1})$,
\begin{equation}\label{15.21}
\| \epsilon \|_{L^{2}} + |\frac{\lambda}{\lambda_{0}} - 1| + |\gamma - \gamma_{0} - \xi_{0} (\tilde{x} - x_{0})| + |\xi - \frac{\lambda}{\lambda_{0}} \xi_{0}| + |\frac{\tilde{x} - x_{0}}{\lambda_{0}}| \lesssim \| e^{i \gamma_{0}} e^{ix \xi_{0}} \lambda_{0}^{1/2} u(\lambda_{0} x + x_{0}) - Q \|_{L^{2}}.
\end{equation}
\end{proof}

As in Theorem $\ref{t7.2}$, it is possible to show that $\lambda(t)$, $\gamma(t)$, $x(t)$, and $\xi(t)$ are continuous functions on $[0, \sup(I))$, and are differentiable almost everywhere on $[0, \sup(I))$. Let $s(t)$ be as in $(\ref{7.5})$. Since $s : [0, \sup(I)) \rightarrow [0, \infty)$ is monotone, the function is invertible, $t(s) : [0, \infty) \rightarrow [0, \sup(I))$. Letting
\begin{equation}\label{15.22}
\gamma(s) = \gamma(t(s)), \qquad \lambda(s) = \lambda(t(s)), \qquad x(s) = x(t(s)), \qquad \xi(s) = \xi(t(s)),
\end{equation}
and letting
\begin{equation}\label{15.23}
\epsilon(s, x) = e^{i \gamma(s)} e^{ix \xi(s)} \lambda(s)^{1/2} u(t(s), \lambda(x) x + x(s)) - Q(x),
\end{equation}
we can compute
\begin{equation}\label{15.24}
\aligned
\epsilon_{s} = i \gamma_{s} (Q + \epsilon) + i \xi_{s} x (Q + \epsilon) + \frac{\lambda_{s}}{\lambda} (\frac{1}{2} (Q + \epsilon) + x (Q + \epsilon)_{x}) - i \frac{\lambda_{s}}{\lambda} \xi(s) x (Q + \epsilon) \\
+ \frac{x_{s}}{\lambda} (Q + \epsilon)_{x} - i \frac{x_{s}}{\lambda} \xi(s) (Q + \epsilon) + i (Q + \epsilon)_{xx} + 2 \xi(s) (Q + \epsilon)_{x} - i \xi(s)^{2} (Q + \epsilon) + i |Q + \epsilon|^{4} (Q + \epsilon).
\endaligned
\end{equation}
Taking $f \in \{ Q^{3}, i Q^{3}, Q_{x}, i Q_{x} \}$,
\begin{equation}\label{15.25}
\frac{d}{ds} (\epsilon, f)_{L^{2}} = (\epsilon_{s}, f)_{L^{2}} = 0.
\end{equation}
Using the fact that $f$ belongs to the span of $\{ Q^{3}, i Q^{3}, Q_{x}, i Q_{x} \}$ if and only if $if$ belongs to the span of $\{ Q^{3}, i Q^{3}, Q_{x}, i Q_{x} \}$ as a real vector space, compute
\begin{equation}\label{15.26}
(i \gamma_{s} (Q + \epsilon), f)_{L^{2}} = (i \gamma_{s} Q, f)_{L^{2}} = 0, \qquad \text{if} \qquad f = Q^{3}, Q_{x}, i Q_{x}, \qquad \text{and} \qquad (i \gamma_{s} (Q + \epsilon), i Q^{3}) = \gamma_{s} \| Q \|_{L^{4}}^{4}.
\end{equation}
\begin{equation}\label{15.27}
\aligned
(\xi_{s} - \frac{\lambda_{s}}{\lambda} \xi(s)) (i x (Q + \epsilon), f)_{L^{2}} = -\frac{1}{2} (\xi_{s}(s) - \frac{\lambda_{s}}{\lambda} \xi(s)) \| Q \|_{L^{2}}^{2} + O(|\xi_{s}(s) - \frac{\lambda_{s}}{\lambda} \xi(s)| \| \epsilon \|_{L^{2}}), \qquad \text{if} \qquad f = i Q_{x}, \\ (\xi_{s} - \frac{\lambda_{s}}{\lambda} \xi(s)) (i x (Q + \epsilon), f)_{L^{2}} = O(|\xi_{s}(s) - \frac{\lambda_{s}}{\lambda} \xi(s)| \| \epsilon \|_{L^{2}}), \qquad \text{if} \qquad f \in \{ Q^{3}, Q_{x}, i Q^{3} \}.
\endaligned
\end{equation}

\begin{equation}\label{15.29}
\aligned
(- i \frac{x_{s}}{\lambda} \xi(s) (Q + \epsilon), f)_{L^{2}} = (-i \frac{x_{s}}{\lambda} \xi(s) Q, f)_{L^{2}} = 0, \qquad \text{if} \qquad f = Q^{3}, Q_{x}, i Q_{x}, \\ (-i \frac{x_{s}}{\lambda} \xi(s), Q, i Q^{3})_{L^{2}} = -\frac{x_{s}}{\lambda} \xi(s) \| Q \|_{L^{4}}^{4}.
\endaligned
\end{equation}

\begin{equation}\label{15.28}
(i \xi(s)^{2} (Q + \epsilon), f)_{L^{2}} = (i \xi(s)^{2} Q, f)_{L^{2}} = 0, \qquad \text{if} \qquad f = Q^{3}, Q_{x}, i Q_{x}, \qquad (i \xi(s)^{2} Q, i Q^{3}) = \xi(s)^{2} \| Q \|_{L^{4}}^{4}.
\end{equation}

\begin{equation}\label{15.30}
\aligned
\frac{\lambda_{s}}{\lambda} (\frac{1}{2} (Q + \epsilon) + x (Q + \epsilon)_{x}, Q^{3})_{L^{2}} = \frac{\lambda_{s}}{4 \lambda} \| Q \|_{L^{4}}^{4} + O(|\frac{\lambda_{s}}{\lambda}| \| \epsilon \|_{L^{2}}) \\  \frac{\lambda_{s}}{\lambda} (\frac{1}{2} (Q + \epsilon) + x (Q + \epsilon)_{x}, f)_{L^{2}} = O(|\frac{\lambda_{s}}{\lambda}| \| \epsilon \|_{L^{2}}), \qquad \text{if} \qquad f = Q_{x}, i Q^{3}, i Q_{x}.
\endaligned
\end{equation}

\begin{equation}\label{15.31}
\aligned
(\frac{x_{s}}{\lambda} + 2 \xi(s)) ((Q + \epsilon)_{x}, Q_{x})_{L^{2}} = (\frac{x_{s}}{\lambda} + 2 \xi(s)) \| Q_{x} \|_{L^{2}}^{2} + O(|\frac{x_{s}}{\lambda} + 2 \xi(s)| \| \epsilon \|_{L^{2}}), \\
(\frac{x_{s}}{\lambda} + 2 \xi(s)) ((Q + \epsilon)_{x}, f)_{L^{2}} = O(|\frac{x_{s}}{\lambda} + 2 \xi(s)| \| \epsilon \|_{L^{2}}), \qquad \text{if} \qquad f = Q^{3}, i Q^{3}, i Q_{x}.
\endaligned
\end{equation}
Finally, taking $\epsilon = \epsilon_{1} + i \epsilon_{2}$,
\begin{equation}\label{15.32}
\aligned
( i (Q + \epsilon)_{xx}   + i |Q + \epsilon|^{4} (Q + \epsilon), f)_{L^{2}} = (i Q, f)_{L^{2}} + (i \mathcal L \epsilon_{1} - \mathcal L_{-} \epsilon_{2}, f)_{L^{2}} + O((|\epsilon|^{2}(|\epsilon|^{3} + |Q|^{3}), f)_{L^{2}},
\endaligned
\end{equation}
where $\mathcal L$ and $\mathcal L_{-}$ are given by $(\ref{2.39})$. Since $\mathcal L$ and $\mathcal L_{-}$ are self-adjoint operators, $(\epsilon_{1}, Q^{3})_{L^{2}} = (\epsilon_{2}, Q^{3})_{L^{2}} = 0$, and $\mathcal L Q_{x} = 0$,
\begin{equation}\label{15.33}
\aligned
( i (Q + \epsilon)_{xx}   + i |Q + \epsilon|^{4} (Q + \epsilon), f)_{L^{2}} = \| Q \|_{L^{4}}^{4} + O((|\epsilon|^{2}(|\epsilon|^{3} + |Q|^{3}), f)_{L^{2}}, \qquad \text{if} \qquad f = i Q^{3}, \\ 
( i (Q + \epsilon)_{xx}   + i |Q + \epsilon|^{4} (Q + \epsilon), f)_{L^{2}} = O((|\epsilon|^{2}(|\epsilon|^{3} + |Q|^{3}), f)_{L^{2}}, \qquad \text{if} \qquad f = i Q_{x}, \\ 
( i (Q + \epsilon)_{xx}   + i |Q + \epsilon|^{4} (Q + \epsilon), f)_{L^{2}} = -(\epsilon_{2}, \mathcal L_{-} Q^{3})_{L^{2}} + O((|\epsilon|^{2}(|\epsilon|^{3} + |Q|^{3}), f)_{L^{2}}, \qquad \text{if} \qquad f = Q^{3}, \\ 
( i (Q + \epsilon)_{xx}   + i |Q + \epsilon|^{4} (Q + \epsilon), f)_{L^{2}} = -(\epsilon_{2}, \mathcal L_{-} Q_{x})_{L^{2}} + O((|\epsilon|^{2}(|\epsilon|^{3} + |Q|^{3}), f)_{L^{2}}, \qquad \text{if} \qquad f = Q_{x}.
\endaligned
\end{equation}

Combining $(\ref{15.26})$--$(\ref{15.33})$, we have proved
\begin{equation}\label{15.34}
\aligned
(\gamma_{s} + 1 - \frac{x_{s}}{\lambda} \xi(s) - \xi(s)^{2}) \| Q \|_{L^{4}}^{4} + O(|\xi_{s} - \frac{\lambda_{s}}{\lambda} \xi(s)| \| \epsilon \|_{L^{2}}) + O(|\frac{\lambda_{s}}{\lambda}| \| \epsilon \|_{L^{2}}) \\ + O(|\frac{x_{s}}{\lambda} + 2 \xi(s)| \| \epsilon \|_{L^{2}}) + O(\| \epsilon \|_{L^{2}}^{2} (\| Q \|_{L^{\infty}}^{3} + \| \epsilon \|_{L^{\infty}}^{3})) = 0,
\endaligned
\end{equation}
\begin{equation}\label{15.35}
\aligned
-\frac{1}{2}(\xi_{s} - \frac{\lambda_{s}}{\lambda} \xi(s)) \| Q \|_{L^{2}}^{2} + O(|\xi_{s} - \frac{\lambda_{s}}{\lambda} \xi(s)| \| \epsilon \|_{L^{2}}) + O(|\frac{\lambda_{s}}{\lambda}| \| \epsilon \|_{L^{2}}) \\ + O(|\frac{x_{s}}{\lambda} + 2 \xi(s)| \| \epsilon \|_{L^{2}}) + O(\| \epsilon \|_{L^{2}}^{2} (\| Q \|_{L^{\infty}}^{3} + \| \epsilon \|_{L^{\infty}}^{3})) = 0,
\endaligned
\end{equation}
\begin{equation}\label{15.36}
\aligned
\frac{\lambda_{s}}{4 \lambda} \| Q \|_{L^{4}}^{4} - (\epsilon_{2}, \mathcal L_{-} Q^{3})_{L^{2}} + O(|\xi_{s} - \frac{\lambda_{s}}{\lambda} \xi(s)| \| \epsilon \|_{L^{2}}) + O(|\frac{\lambda_{s}}{\lambda}| \| \epsilon \|_{L^{2}}) \\ + O(|\frac{x_{s}}{\lambda} + 2 \xi(s) | \| \epsilon \|_{L^{2}}) + O(\| \epsilon \|_{L^{2}}^{2} (\| Q \|_{L^{\infty}}^{3} + \| \epsilon \|_{L^{\infty}}^{3})) = 0,
\endaligned
\end{equation}
and
\begin{equation}\label{15.37}
\aligned
(\frac{x_{s}}{\lambda} + 2 \xi) \| Q_{x} \|_{L^{2}}^{2} - (\epsilon_{2}, \mathcal L_{-} Q_{x})_{L^{2}} + O(|\xi_{s} - \frac{\lambda_{s}}{\lambda} \xi(s)| \| \epsilon \|_{L^{2}}) + O(|\frac{\lambda_{s}}{\lambda}| \| \epsilon \|_{L^{2}}) \\ + O(|\frac{x_{s}}{\lambda} + 2 \xi(s) | \| \epsilon \|_{L^{2}}) + O(\| \epsilon \|_{L^{2}}^{2} (\| Q \|_{L^{\infty}}^{3} + \| \epsilon \|_{L^{\infty}}^{3})) = 0.
\endaligned
\end{equation}
Using the same analysis as in $(\ref{7.14})$--$(\ref{7.15.2})$, for any $a \in \mathbb{Z}_{\geq 0}$,
\begin{equation}\label{15.38}
\int_{a}^{a + 1} |\gamma_{s} + 1 - \frac{x_{s}}{\lambda} \xi(s) - \xi(s)^{2} | ds \lesssim \int_{a}^{a + 1} \| \epsilon(s) \|_{L^{2}}^{2} ds,
\end{equation}
\begin{equation}\label{15.39}
\int_{a}^{a + 1} |\xi_{s} - \frac{\lambda_{s}}{\lambda} \xi(s)| ds \lesssim \int_{a}^{a + 1} \| \epsilon(s) \|_{L^{2}}^{2} ds,
\end{equation}
\begin{equation}\label{15.40}
\int_{a}^{a + 1} |\frac{\lambda_{s}}{\lambda}| ds \lesssim \int_{a}^{a + 1} \| \epsilon(s) \|_{L^{2}} ds,
\end{equation}
and
\begin{equation}\label{15.41}
\int_{a}^{a + 1} |\frac{x_{s}}{\lambda} + 2 \xi| ds \lesssim \int_{a}^{a + 1} \| \epsilon(s) \|_{L^{2}} ds.
\end{equation}

\subsection{A long time Strichartz estimate in the nonsymmetric case}
The symmetry $(\ref{15.1})$ does not impact the long time Strichartz estimates in Theorems $\ref{t6.2}$--$\ref{t6.4}$ at all. However, the Galilean symmetry $(\ref{15.2})$ does, since it involves a translation in frequency, and therefore will impact estimates of $u$ under frequency cutoffs. Nevertheless, it is possible to prove a modification of Theorem $\ref{t6.4}$ using virtually the same arguments.

\begin{theorem}\label{t15.4}
Suppose $\lambda(t)$, $x(t)$, $\xi(t)$, and $\gamma(t)$ are as in $(\ref{15.7})$. Also suppose that on the interval $J = [a, b]$,
\begin{equation}\label{15.42}
\lambda(t) \geq \frac{1}{\eta_{1}}, \qquad \int_{J} \lambda(t)^{-2} dt = T, \qquad \text{and} \qquad \eta_{1}^{-2} T = 2^{3k}.
\end{equation}
Furthermore, suppose that
\begin{equation}\label{15.43}
\frac{|\xi(t)|}{\lambda(t)} \leq \eta_{0}, \qquad \text{for all} \qquad t \in [a, b].
\end{equation}
Then
\begin{equation}\label{15.44}
\| P_{\geq k} u \|_{U_{\Delta}^{2}([a, b] \times \mathbb{R})} \lesssim T^{-10} + (\frac{1}{T} \int_{a}^{b} \| \epsilon(t) \|_{L^{2}}^{2} \lambda(t)^{-2} dt)^{1/2}.
\end{equation}
\end{theorem}
\begin{proof}
Observe that by $(\ref{15.7})$,
\begin{equation}\label{15.45}
u(t,x) = e^{-i \gamma(t)} e^{-ix \frac{\xi(t)}{\lambda(t)}} \lambda(t)^{-1/2} Q(\frac{x - x(t)}{\lambda(t)}) + e^{-i \gamma(t)} e^{-i x \frac{\xi(t)}{\lambda(t)}} \lambda(t)^{-1/2} \epsilon(t, \frac{x - x(t)}{\lambda(t)}).
\end{equation}
Then by $(\ref{6.2})$, $(\ref{6.3})$, and $(\ref{15.45})$,
\begin{equation}\label{15.46}
\| P_{> 0} u \|_{L_{t}^{\infty} L_{x}^{2}([a, b] \times \mathbb{R})}^{2} \leq 4 \eta_{0}^{2}.
\end{equation}
Applying the induction on frequency arguments in Theorems $\ref{t6.2}$--$\ref{t6.4}$ gives the same results.
\end{proof}

\subsection{Almost conservation of energy for a non-symmetric solution}
It is possible to use the long time Strichartz estimates in Theorem $\ref{t15.4}$ to prove an almost conservation of energy for a non-symmetric solution.
\begin{theorem}\label{t15.5}
Let $J = [a, b]$ be an interval such that
\begin{equation}\label{15.48}
\lambda(t) \geq \frac{1}{\eta_{1}}, \qquad \frac{|\xi(t)|}{\lambda(t)} \leq \eta_{0}, \qquad \text{for all} \qquad t \in J, \qquad \int_{J} \lambda(t)^{-2} dt = T, \qquad \eta_{1}^{-2} T = 2^{3k}.
\end{equation}
Then,
\begin{equation}\label{15.49}
\sup_{t \in J} E(P_{\leq k + 9} u(t)) \lesssim \frac{2^{2k}}{T} \int_{J} \| \epsilon(t) \|_{L^{2}}^{2} \lambda(t)^{-2} dt + (\sup_{t \in J} \frac{\xi(t)}{\lambda(t)})^{2} + 2^{2k} T^{-10}.
\end{equation}
\end{theorem}
\begin{proof}
Decompose the energy as in Theorem $\ref{t2.4}$. Since $E(Q) = 0$ and $(\epsilon_{2}, Q_{x}) = 0$,
\begin{equation}\label{15.50}
\aligned
E(u) = E(e^{-i \gamma(t)} e^{-i x \frac{\xi(t)}{\lambda(t)}} \lambda(t)^{-1/2} Q(\frac{x - x(t)}{\lambda(t)}) + e^{-i \gamma(t)} e^{-i x \frac{\xi(t)}{\lambda(t)}} \lambda(t)^{-1/2} \epsilon(t, \frac{x - x(t)}{\lambda(t)})) \\
= \frac{1}{2 \lambda(t)^{2}} \| Q_{x} \|_{L^{2}}^{2} + \frac{\xi(t)^{2}}{2 \lambda(t)^{2}} \| Q \|_{L^{2}}^{2} - \frac{1}{6 \lambda(t)^{2}} \| Q \|_{L^{6}}^{6} + \frac{1}{2 \lambda(t)^{2}} \| \epsilon \|_{L^{2}}^{2} - \frac{2 \xi(t)}{\lambda(t)^{2}} (Q_{x}, \epsilon_{2})_{L^{2}} - \frac{\xi(t)^{2}}{2 \lambda(t)^{2}} \| \epsilon \|_{L^{2}}^{2} \\ 
+ \frac{1}{2 \lambda(t)^{2}} \| \nabla \epsilon \|_{L^{2}}^{2} - \frac{\xi(t)}{\lambda(t)^{2}} (\nabla \epsilon_{1}, \epsilon_{2})_{L^{2}} + \frac{\xi(t)}{\lambda(t)^{2}} (\nabla \epsilon_{2}, \epsilon_{1})_{L^{2}} + \frac{\xi(t)^{2}}{2 \lambda(t)^{2}} \| \epsilon \|_{L^{2}}^{2} \\
- \frac{5}{2 \lambda(t)^{2}} \int Q(x)^{4} \epsilon_{1}(t,x)^{2} dx - \frac{1}{2 \lambda(t)^{2}} \int Q(x)^{4} \epsilon_{2}(t,x)^{2} dx + O(\frac{1}{\lambda(t)^{2}} \| \epsilon \|_{L^{3}}^{3} + \frac{1}{\lambda(t)^{2}} \| \epsilon \|_{L^{6}}^{6}) \\
= \frac{\xi(t)^{2}}{2 \lambda(t)^{2}} \| Q \|_{L^{2}}^{2}  + \frac{1}{2 \lambda(t)^{2}} \| \epsilon \|_{L^{2}}^{2}  - \frac{\xi(t)}{\lambda(t)^{2}} (\nabla \epsilon_{1}, \epsilon_{2})_{L^{2}} + \frac{\xi(t)}{\lambda(t)^{2}} (\nabla \epsilon_{2}, \epsilon_{1})_{L^{2}}
\\ + \frac{1}{2 \lambda(t)^{2}} \| \nabla \epsilon \|_{L^{2}}^{2} - \frac{5}{2 \lambda(t)^{2}} \int Q(x)^{4} \epsilon_{1}(t,x)^{2} dx - \frac{1}{2 \lambda(t)^{2}} \int Q(x)^{4} \epsilon_{2}(t,x)^{2} dx + O(\frac{1}{\lambda(t)^{2}} \| \epsilon \|_{L^{3}}^{3} + \frac{1}{\lambda(t)^{2}} \| \epsilon \|_{L^{6}}^{6}).
\endaligned
\end{equation}
Using the bounds on $\frac{|\xi(t)|}{\lambda(t)}$, the fact that $Q$ and all its derivatives are rapidly decreasing, Fourier truncation, and the mean value theorem implies that $(\ref{15.49})$ holds for some $t_{0} \in J$. Then, using the long time Strichartz estimates in Theorem $\ref{t15.4}$ and following the proof of Theorem $\ref{t8.3}$ gives Theorem $\ref{t15.5}$.
\end{proof}

It is also possible to generalize Corollary $\ref{c8.4}$ to the non-symmetric case.
\begin{corollary}\label{c15.6}
If
\begin{equation}
\frac{1}{\eta_{1}} \leq \lambda(t) \leq \frac{1}{\eta_{1}} T^{1/100}, \qquad \text{and} \qquad  \frac{|\xi(t)|}{\lambda(t)} \leq \eta_{0}, \qquad \text{for all} \qquad t \in J,
\end{equation}
and
\begin{equation}
\int_{J} \lambda(t)^{-2} dt = T, \qquad \text{and} \qquad \eta_{1}^{-2} T = 2^{3k},
\end{equation}
then
\begin{equation}\label{15.51}
\sup_{t \in J} \| P_{\leq k + 9} (\frac{e^{-i \gamma(t)} e^{-ix \frac{\xi(t)}{\lambda(t)}}}{\lambda(t)^{1/2}} \epsilon(t, \frac{x - x(t)}{\lambda(t)})) \|_{\dot{H}^{1}}^{2} \lesssim \frac{2^{2k}}{T} \int_{J} \| \epsilon(t) \|_{L^{2}}^{2} \lambda(t)^{-2} dt + (\sup_{t \in J} \frac{\xi(t)^{2}}{\lambda(t)^{2}}) + 2^{2k} T^{-10},
\end{equation}
and
\begin{equation}\label{15.52}
\sup_{t \in J} \| \epsilon(t) \|_{L^{2}}^{2} \lesssim \frac{2^{2k} T^{1/50}}{\eta_{1}^{2} T} \int_{J} \| \epsilon(t) \|_{L^{2}}^{2} \lambda(t)^{-2} dt + \frac{T^{1/50}}{\eta_{1}^{2}} (\sup_{t \in J} \frac{\xi(t)^{2}}{\lambda(t)^{2}}) + 2^{2k} \frac{T^{1/50}}{\eta_{1}^{2}} T^{-10}.
\end{equation}
\end{corollary}
\begin{proof}
As in the proof of Theorem $\ref{t2.4}$, since $\epsilon \perp \{ Q^{3}, Q_{x}, i Q^{3}, i Q_{x} \}$, there exists some $c > 0$ such that
\begin{equation}\label{15.53}
 \frac{1}{2 \lambda(t)^{2}} \| \epsilon \|_{L^{2}}^{2} + \frac{1}{2 \lambda(t)^{2}} \| \nabla \epsilon \|_{L^{2}}^{2} - \frac{5}{2 \lambda(t)^{2}} \int Q(x)^{4} \epsilon_{1}(t,x)^{2} dx - \frac{1}{2 \lambda(t)^{2}} \int Q(x)^{4} \epsilon_{2}(t,x)^{2} dx \geq \frac{1}{2 \lambda(t)^{2}} \| \epsilon \|_{L^{2}}^{2} + \frac{c}{\lambda(t)^{2}} \| \nabla \epsilon \|_{L^{2}}^{2}.
\end{equation}
Next, for $\| \epsilon \|_{L^{2}} \leq \eta_{0}$ sufficiently small, by the Cauchy--Schwarz inequality, taking $\delta = \frac{\| \epsilon \|_{L^{2}}}{\| Q \|_{L^{2}}}$ in the last step,
\begin{equation}\label{15.54}
\aligned
\frac{\xi(t)^{2}}{\lambda(t)^{2}} \| Q \|_{L^{2}}^{2} - \frac{\xi(t)}{\lambda(t)^{2}} (\nabla \epsilon_{1}, \epsilon_{2})_{L^{2}} + \frac{\xi(t)}{\lambda(t)^{2}} (\nabla \epsilon_{2}, \epsilon_{1})_{L^{2}} \\ \geq \frac{\xi(t)^{2}}{\lambda(t)^{2}} \| Q \|_{L^{2}}^{2} - \frac{1}{\delta}\frac{\xi(t)^{2}}{\lambda(t)^{2}} \| \epsilon \|_{L^{2}}^{2} -  \frac{\delta}{\lambda(t)^{2}} \| \nabla \epsilon \|_{L^{2}}^{2} \geq -O(\frac{\eta_{0}}{\lambda(t)^{2}}) \| \nabla \epsilon \|_{L^{2}}^{2}.
\endaligned
\end{equation}
Finally, by H{\"o}lder's inequality and the Sobolev embedding theorem, since $\| \epsilon \|_{L^{2}} \ll 1$,
\begin{equation}\label{15.55}
\aligned
O(\frac{1}{\lambda(t)^{2}} \| \epsilon \|_{L^{3}}^{3} + \frac{1}{\lambda(t)^{2}} \| \epsilon \|_{L^{6}}^{6}) \ll \frac{1}{\lambda(t)^{2}} \| \epsilon \|_{L^{2}}^{2} + \frac{1}{\lambda(t)^{2}} \| \nabla \epsilon \|_{L^{2}}^{2}.
\endaligned
\end{equation}
Plugging $(\ref{15.53})$--$(\ref{15.55})$ into $(\ref{15.50})$ proves the corollary.
\end{proof}

\subsection{A frequency localized Morawetz estimate for nonsymmetric $u$}
As in section six, the long time Strichartz estimates of Theorem $\ref{t15.4}$ and the energy estimates of Theorem $\ref{t15.5}$ and Corollary $\ref{c15.6}$ give a theorem analogous to Theorem $\ref{t10.1}$ in the nonsymmetric case.

\begin{theorem}\label{t15.7}
Let $J = [a, b]$ be an interval on which
\begin{equation}\label{15.56}
\frac{|\xi(t)|}{\lambda(t)} \leq \eta_{0}, \qquad \frac{1}{\eta_{1}} \leq \lambda(t) \leq \frac{1}{\eta_{1}} T^{1/100}, \qquad \text{for all} \qquad t \in J, \qquad \int_{J} \lambda(t)^{-2} dt = T, \qquad \eta_{1}^{-2} T = 2^{3k}.
\end{equation}
Also suppose $\epsilon = \epsilon_{1} + i \epsilon_{2}$, where $\epsilon$ is given by Theorem $\ref{t2.3}$. Finally, suppose there exists a uniform bound on $x(t)$,
\begin{equation}\label{15.57}
\sup_{t \in J} |x(t)| \leq R = T^{1/25}.
\end{equation}
Finally, suppose that $\xi(a) = 0$ and $x(b) = 0$. Then for $T$ sufficiently large,
\begin{equation}\label{15.58}
\aligned
\int_{a}^{b} \| \epsilon(t) \|_{L^{2}}^{2} \lambda(t)^{-2} dt \leq 3(\epsilon_{2}(a), (\frac{1}{2} Q + x Q_{x}))_{L^{2}} - 3(\epsilon_{2}(b), \frac{1}{2} Q + x Q_{x})_{L^{2}} + \frac{T^{1/50}}{\eta_{1}^{2}} \sup_{t \in J} \frac{\xi(t)^{2}}{\lambda(t)^{2}} + O(\frac{1}{T^{9}}).
\endaligned
\end{equation}
\end{theorem}
\begin{proof}
This time let
\begin{equation}\label{15.59}
\phi(x) = \int_{0}^{x} \chi(\frac{\eta_{1} y}{2R}) dy = \int_{0}^{x} \psi^{2}(\frac{\eta_{1} y}{2R}) dy,
\end{equation}
and let
\begin{equation}\label{15.60}
M(t) = \int \phi(x) Im[\overline{P_{\leq k + 9} u} \partial_{x} P_{\leq k + 9} u](t,x) dx.
\end{equation}

Since $|\phi(x)| \lesssim \eta_{1}^{-1} R$ and $\frac{|\xi(t)|}{\lambda(t)} \leq \eta_{0}$, Theorem $\ref{t15.4}$ implies that the error terms arising from frequency truncation may be handled in exactly the same manner as in Theorem $\ref{t10.1}$.

Next, observe that by $(\ref{15.54})$ and $(\ref{15.55})$, the additional terms in the left hand side of $(\ref{10.17})$ that arise from the fact that $\xi(t)$ need not be zero may be handled in exactly the same manner as the terms involving $\epsilon^{3}$ and higher powers of $\epsilon$.

Now decompose $M(b) - M(a)$. Since $Q$ is real valued, symmetric, and rapidly decreasing, $(\ref{15.59})$, the bounds on $\lambda(t)$, and $(\ref{15.57})$ imply
\begin{equation}\label{15.61}
\aligned
\int \phi(x) Im[\overline{e^{-i \gamma(t)} e^{-ix \frac{\xi(t)}{\lambda(t)}} \lambda(t)^{-1/2} P_{\leq k + 9} Q(\frac{x - x(t)}{\lambda(t)})} \partial_{x}(e^{-i \gamma(t)} e^{-ix \frac{\xi(t)}{\lambda(t)}} \lambda(t)^{-1/2} P_{\leq k + 9} Q(\frac{x - x(t)}{\lambda(t)}))] dx \\
= \frac{\xi(t)}{\lambda(t)^{2}} \int \phi(x) Q(\frac{x - x(t)}{\lambda(t)})^{2} dx + O(T^{-10}) = \frac{\xi(t)}{\lambda(t)} x(t) \| Q \|_{L^{2}}^{2} + O(T^{-10}).
\endaligned
\end{equation}
Since $\xi(a) = 0$ and $x(b) = 0$, $\frac{\xi(t)}{\lambda(t)} x(t) \| Q \|_{L^{2}}^{2}|_{a}^{b} = 0$.

Next, by Corollary $\ref{c15.6}$,
\begin{equation}\label{15.62}
\aligned
\int \phi(x) Im[\overline{P_{\leq k + 9} e^{-i \gamma(t)} e^{-i x \frac{\xi(t)}{\lambda(t)}} \lambda(t)^{-1/2} \epsilon(t, \frac{x - x(t)}{\lambda(t)})} \partial_{x}(P_{\leq k + 9} e^{-i \gamma(t)} e^{-ix \frac{\xi(t)}{\lambda(t)}} \lambda(t)^{-1/2} \epsilon(t, \frac{x - x(t)}{\lambda(t)}))] dx \\
\lesssim \frac{R}{\eta_{1}^{2}}  \frac{2^{2k}}{T^{99/100}} \int_{J} \| \epsilon(t) \|_{L^{2}}^{2} \lambda(t)^{-2} dt + \frac{R}{\eta_{1}^{2}} 2^{2k} T^{-9.99} + \frac{T^{1/50}}{\eta_{1}^{2}} \sup_{t \in J} \frac{\xi(t)^{2}}{\lambda(t)^{2}}.
\endaligned
\end{equation}
Next, using the computations proving $(\ref{10.18})$ combined with the fact that $(\epsilon_{2}, Q_{x}) = 0$,
\begin{equation}\label{15.63}
\aligned
\int \phi(x) Im[\overline{P_{\leq k + 9} e^{-i \gamma(t)} e^{-ix \frac{\xi(t)}{\lambda(t)}} \lambda(t)^{-1/2} \epsilon(t, \frac{x - x(t)}{\lambda(t)})} \partial_{x}(P_{\leq k + 9} e^{-i \gamma(t)} e^{-ix \frac{\xi(t)}{\lambda(t)}} \lambda(t)^{-1/2} Q(\frac{x - x(t)}{\lambda(t)}))] dx \\
= -(\epsilon_{2}, x Q_{x})_{L^{2}} + \frac{x(t)}{\lambda(t)} (\epsilon_{2}, Q_{x})_{L^{2}} - \frac{\xi(t)}{\lambda(t)} (\epsilon_{1}, Q)_{L^{2}} + O(T^{-10}) = -(\epsilon_{2}, x Q_{x})_{L^{2}} - \frac{\xi(t)}{2 \lambda(t)} \| \epsilon \|_{L^{2}}^{2} + O(T^{-10}).
\endaligned
\end{equation}
Finally, integrating by parts,
\begin{equation}\label{15.64}
\aligned
\int \phi(x) Im[\overline{P_{\leq k + 9} e^{-i \gamma(t)} e^{-ix \frac{\xi(t)}{\lambda(t)}} \lambda(t)^{-1/2} Q(\frac{x - x(t)}{\lambda(t)})} \partial_{x}(P_{\leq k + 9} \lambda(t)^{-1/2} e^{-i \gamma(t)} e^{-i x \frac{\xi(t)}{\lambda(t)}} \epsilon(t, \frac{x - x(t)}{\lambda(t)}))] dx \\ = (\ref{15.63}) - \int \chi(\frac{\eta_{1} x}{2 R}) Im[\overline{P_{\leq k + 9} e^{-i \gamma(t)} e^{-ix \frac{\xi(t)}{\lambda(t)}} \lambda(t)^{-1/2} Q(\frac{x - x(t)}{\lambda(t)})} \\ \cdot P_{\leq k + 9} e^{-i \gamma(t)} e^{-ix \frac{\xi(t)}{\lambda(t)}} \lambda(t)^{-1/2} \epsilon(t, \frac{x - x(t)}{\lambda(t)})] dx.
\endaligned
\end{equation}
As in $(\ref{10.19.1})$,
\begin{equation}\label{15.65}
\aligned
- \int \chi(\frac{\eta_{1} x}{2 R}) Im[\overline{P_{\leq k + 9} e^{-i \gamma(t)} e^{-ix \frac{\xi(t)}{\lambda(t)}} \lambda(t)^{-1/2} Q(\frac{x - x(t)}{\lambda(t)})}  \cdot P_{\leq k + 9} e^{-i \gamma(t)} e^{-ix \frac{\xi(t)}{\lambda(t)}} \lambda(t)^{-1/2} \epsilon(t, \frac{x - x(t)}{\lambda(t)})] dx \\ = -(\epsilon_{2}, Q)_{L^{2}} + O(T^{-10}).
\endaligned
\end{equation}
Summing up $(\ref{15.61})$--$(\ref{15.65})$ and using the fundamental theorem of calculus and the Morawetz estimate completes the proof of Theorem $\ref{t15.7}$.
\end{proof}

\subsection{An $L_{s}^{p}$ bound on $\| \epsilon(s) \|_{L^{2}}$ when $p > 1$ for nonsymmetric $u$}
Combining Theorem $\ref{t15.7}$ with $(\ref{15.34})$--$(\ref{15.37})$, it is possible to prove Theorem $\ref{t7.1}$ for nonsymmetric $u$.
\begin{theorem}\label{t15.8}
Let $u$ be a nonsymmetric solution to $(\ref{1.1})$ that satisfies $\| u \|_{L^{2}} = \| Q \|_{L^{2}}$, and suppose
\begin{equation}\label{15.66}
\sup_{s \in [0, \infty)} \| \epsilon(s) \|_{L^{2}} \leq \eta_{\ast},
\end{equation}
and $\| \epsilon(0) \|_{L^{2}} = \eta_{\ast}$. Then
\begin{equation}\label{15.67}
\int_{0}^{\infty} \| \epsilon(s) \|_{L^{2}}^{2} ds \lesssim \eta_{\ast},
\end{equation}
with implicit constant independent of $\eta_{\ast}$ when $\eta_{\ast} \ll 1$ is sufficiently small.

Furthermore, for any $j \in \mathbb{Z}_{\geq 0}$, let
\begin{equation}\label{15.68}
s_{j} = \inf \{ s \in [0, \infty) : \| \epsilon(s) \|_{L^{2}} = 2^{-j} \eta_{\ast} \}.
\end{equation}
By definition, $s_{0} = 0$, and as in Theorem $\ref{t7.1}$, such an $s_{j}$ exists for any $j > 0$. Then,
\begin{equation}\label{15.69}
\int_{s_{j}}^{\infty} \| \epsilon(s) \|_{L^{2}}^{2} ds \lesssim 2^{-j} \eta_{\ast},
\end{equation}
for each $j$, with implicit constant independent of $\eta_{\ast}$ and $j \geq 0$.
\end{theorem}

\begin{proof}
Set $T_{\ast} = \frac{1}{\eta_{\ast}}$ and suppose that $T_{\ast}$ is sufficiently large such that Theorem $\ref{t15.7}$ holds. Then by $(\ref{15.66})$ and $(\ref{15.38})$,
\begin{equation}\label{15.70}
|\sup_{s \in [s', s' + T_{\ast}]} \ln(\lambda(s)) - \inf_{s \in [s', s' + T_{\ast}]} \ln(\lambda(s))| \lesssim 1.
\end{equation}
Let $J$ be the largest dyadic integer that satisfies
\begin{equation}\label{15.71}
J = 2^{j_{\ast}} \leq -\ln(\eta_{\ast})^{1/4}.
\end{equation}
By $(\ref{15.40})$ and the triangle inequality,
\begin{equation}\label{15.72}
\aligned
|\sup_{s \in [s', s' + J T_{\ast}]} \ln(\lambda(s)) - \inf_{s \in [s', s' + J T_{\ast}]} \ln(\lambda(s))| \lesssim J,
\endaligned
\end{equation}
and therefore,
\begin{equation}\label{15.73}
\frac{\sup_{s \in [s', s' + 3 J T_{\ast}]} \lambda(s)}{\inf_{s \in [s', s' + 3JT^{\ast}]} \lambda(s)} \lesssim T_{\ast}^{1/100}.
\end{equation}
Rescale so that $\inf_{s \in [s', s' + 3J T_{\ast}]} \lambda(s) = \frac{1}{\eta_{1}}$. Then make a Galilean transformation so that $\xi(s') = 0$ and a translation in space so that $x(s'') = 0$ when $s'' \in [s', s' + 3J T_{\ast}]$ is the other endpoint of the interval of integration. Then by $(\ref{15.39})$ and $(\ref{15.41})$,
\begin{equation}\label{15.74}
\sup_{s \in [s', s' + 3 J T_{\ast}]} \frac{|\xi(s)|}{\lambda(s)} \lesssim \eta_{\ast} J \eta_{1} \ll \eta_{0}, \qquad \text{and} \qquad \sup_{s \in [s', s' + 3 J T_{\ast}]} |x(s)| \lesssim J^{2} T_{\ast}^{1/100} + \frac{1}{\eta_{1}} T_{\ast}^{1/100} J \ll T_{\ast}^{1/25}.
\end{equation}

Therefore, by Theorem $\ref{t15.7}$,
\begin{equation}\label{15.75}
\sup_{a > 0} \int_{s' + a J T_{\ast}}^{s' + (a + 1) J T_{\ast}} \| \epsilon(s) \|_{L^{2}}^{2} \lesssim \frac{1}{J^{1/2} T_{\ast}^{1/2}} (\sup_{a \geq 0} \int_{s' + a J T_{\ast}}^{s' + (a + 1) J T_{\ast}} \| \epsilon(s) \|_{L^{2}}^{2} ds)^{1/2} + T_{\ast}^{1/50} \eta_{\ast}^{2} + O(\frac{1}{J^{9} T_{\ast}^{9}}),
\end{equation}
and when $a = 0$,
\begin{equation}\label{15.76}
 \int_{s'}^{s' + J T_{\ast}} \| \epsilon(s) \|_{L^{2}}^{2} \lesssim \| \epsilon(s') \|_{L^{2}} + \frac{1}{J^{1/2} T_{\ast}^{1/2}} (\sup_{a \geq 0} \int_{s' + a J T_{\ast}}^{s' + (a + 1) J T_{\ast}} \| \epsilon(s) \|_{L^{2}}^{2} ds)^{1/2} + T_{\ast}^{1/50} \eta_{\ast}^{2} + O(\frac{1}{J^{9} T_{\ast}^{9}}).
\end{equation}
Therefore, taking $s' = s_{j_{\ast}}$,
\begin{equation}\label{15.77}
\sup_{a \geq 0} \int_{s_{j_{\ast}} + a J T_{\ast}}^{s_{j_{\ast}} + (a + 1) J T_{\ast}} \| \epsilon(s) \|_{L^{2}}^{2} ds \lesssim 2^{-j_{\ast}} \eta_{\ast} + O(2^{-9 j_{\ast}} \eta_{\ast}^{9}).
\end{equation}

By the triangle inequality,
\begin{equation}\label{15.78}
\sup_{s' \geq s_{j_{\ast}}} \int_{s'}^{s' + J T_{\ast}} \| \epsilon(s) \|_{L^{2}}^{2} ds \lesssim 2^{-j_{\ast}} \eta_{\ast},
\end{equation}
and by H{\"o}lder's inequality,
\begin{equation}\label{15.79}
\sup_{s' \geq s_{j_{\ast}}} \int_{s'}^{s' + J T_{\ast}} \| \epsilon(s) \|_{L^{2}} ds \lesssim 1.
\end{equation}

It is therefore possible to prove Theorem $\ref{t15.8}$ by induction. Indeed, suppose that for some $n > 0$,
\begin{equation}\label{15.83}
\sup_{s' \geq s_{nj_{\ast}}} \int_{s'}^{s' + J^{n} T_{\ast}} \| \epsilon(s) \|_{L^{2}} ds \leq C, \qquad \text{and} \qquad \sup_{s' \geq s_{n j_{\ast}}} \int_{s'}^{s' + J^{n} T_{\ast}} \| \epsilon(s) \|_{L^{2}}^{2} ds \leq C^{2} J^{-n} \eta_{\ast}.
\end{equation}
Then by $(\ref{15.40})$,
\begin{equation}\label{15.80}
\sup_{s' \geq s_{j_{\ast}}} |\sup_{s \in [s', s' + J^{n + 1} T_{\ast}]} \ln(\lambda(s)) - \inf_{s \in [s', s' + J T_{\ast}]} \ln \lambda(s)| \lesssim CJ.
\end{equation}
Next, rescaling so that $\inf_{s \in [s', s' + J^{n + 1} T_{\ast}]} \lambda(s) = \frac{1}{\eta_{1}}$, setting $\xi(s') = 0$, $(\ref{15.39})$ implies
\begin{equation}\label{15.81}
\sup_{s \in [s', s' + J^{n + 1} T_{\ast}]} \frac{|\xi(s)|}{\lambda(s)} \lesssim 2^{-j_{\ast} n} \eta_{\ast} \eta_{1} C^{2} J \ll \eta_{0},
\end{equation}
and by $(\ref{15.41})$, if $x(s'') = 0$, where $s''$ is the other endpoint of the interval of integration,
\begin{equation}\label{15.82}
\sup_{s \in [s' + J^{n + 1} T_{\ast}]} |x(s)| \lesssim C^{2} J^{2} T_{\ast}^{1/100} + C \frac{1}{\eta_{1}} T_{\ast}^{1/100} J \ll T_{\ast}^{1/25}.
\end{equation}
Then by Theorem $\ref{t15.7}$, as in $(\ref{7.26})$,
\begin{equation}\label{15.84}
\sup_{s' \geq s_{(n + 1) j_{\ast}}} \int_{s'}^{s' + J^{n + 1} T_{\ast}} \| \epsilon(s) \|_{L^{2}}^{2} ds \lesssim J^{-(n + 1)} T_{\ast}^{-1} + T_{\ast}^{1/25} 2^{-2j_{\ast} n} \eta_{\ast}^{2} C^{4} J^{2} \lesssim J^{-(n + 1)} T_{\ast}^{-1}.
\end{equation}
and by H{\"o}lder's inequality,
\begin{equation}\label{15.85}
\sup_{s' \geq s_{(n + 1) j_{\ast}}} \int_{s'}^{s' + J^{n + 1} T_{\ast}} \| \epsilon(s) \|_{L^{2}} ds \lesssim 1.
\end{equation}
It is important to observe that the implicit constants in $(\ref{15.84})$ and $(\ref{15.85})$ are independent of $C$ so long as the final inequalities in $(\ref{15.81})$ and $(\ref{15.82})$ hold and $C \ll T_{\ast}^{1/2}$.

Now take any $j \in \mathbb{Z}$ and suppose $n j_{\ast} < j \leq (n + 1) j_{\ast}$. Then by $(\ref{15.84})$ and $(\ref{15.85})$,
\begin{equation}\label{15.86}
\sup_{a \geq 0} \int_{s_{j} + a J^{n + 1} T_{\ast}}^{s_{j} + (a + 1) J^{n + 1} T_{\ast}} \| \epsilon(s) \|_{L^{2}} ds \lesssim J,
\end{equation}
and
\begin{equation}\label{15.87}
\sup_{a \geq 0} \int_{s_{j} + a J^{n + 1} T_{\ast}}^{s_{j} + (a + 1) J^{n + 1} T_{\ast}} \| \epsilon(s) \|_{L^{2}}^{2} ds \lesssim 2^{-j} \eta_{\ast},
\end{equation}
and therefore, after appropriate rescaling and Galilean and spatial translation, $(\ref{15.80})$--$(\ref{15.82})$ hold. Therefore, by Theorem $\ref{t15.7}$,
\begin{equation}\label{15.88}
\sup_{s' \geq s_{j}} \int_{s'}^{s' + 2^{j} T_{\ast}} \| \epsilon(s) \|_{L^{2}} ds \lesssim 1, \qquad \text{and} \qquad \int_{s'}^{s' + 2^{j} T_{\ast}} \| \epsilon(s) \|_{L^{2}}^{2} ds \lesssim 2^{-j} \eta_{\ast},
\end{equation}
with implicit constant independent of $j$. Furthermore, as in $(\ref{7.24.1})$,
\begin{equation}
\int_{s'}^{s' + 2^{j} J T_{\ast}} \| \epsilon(s) \|_{L^{2}}^{2} ds \lesssim 2^{-j} \eta_{\ast},
\end{equation}
so by the mean value theorem,
\begin{equation}\label{15.90}
\inf_{s \in [s_{j}, s_{j} + 2^{j} J T_{\ast}]} \| \epsilon(s) \|_{L^{2}} \lesssim 2^{-j} \eta_{\ast} J^{-1/2},
\end{equation}
which implies
\begin{equation}\label{15.91}
s_{j + 1} \in [s_{j}, s_{j} + 2^{j} J T_{\ast}].
\end{equation}
Therefore,
\begin{equation}\label{15.92}
\int_{s_{j}}^{s_{j + 1}} \| \epsilon(s) \|_{L^{2}}^{2} ds \lesssim 2^{-j} \eta_{\ast}, \qquad \text{and} \qquad \int_{s_{j}}^{s_{j + 1}} \| \epsilon(s) \|_{L^{2}} ds \lesssim 1,
\end{equation}
with constant independent of $j$. Summing in $j$ gives $(\ref{15.67})$ and $(\ref{15.69})$.
\end{proof}

Now then, as in section seven,
\begin{equation}\label{15.93}
\lim_{s \rightarrow \infty} \| \epsilon(s) \|_{L^{2}} = 0,
\end{equation}
\begin{equation}\label{15.94}
\int_{s_{j}}^{s_{j + 1}} \| \epsilon(s) \|_{L^{2}} ds \lesssim 1,
\end{equation}
and for any $1 < p < \infty$,
\begin{equation}\label{15.95}
(\int_{s_{j}}^{s_{j + 1}} \| \epsilon(s) \|_{L^{2}}^{p} ds) \lesssim \eta_{\ast}^{p - 1} 2^{-j(p - 1)}, 
\end{equation}
which implies that $\| \epsilon(s) \|_{L^{2}}$ belongs to $L_{s}^{p}$ for any $p > 1$, but not $L_{s}^{1}$.

\subsection{Monotonicity of $\lambda$ in the non-symmetric case}
It is possible to use the virial identity from \cite{merle2005blow} to show monotonicity in the non-symmetric case as well.
\begin{theorem}\label{t15.9}
For any $s \geq 0$, let
\begin{equation}\label{15.96}
\tilde{\lambda}(s) = \inf_{\tau \in [0, s]} \lambda(\tau).
\end{equation}
Then for any $s \geq 0$,
\begin{equation}\label{15.97}
1 \leq \frac{\lambda(s)}{\tilde{\lambda}(s)} \leq 3.
\end{equation}

\end{theorem}

\begin{proof}
Suppose there exist $0 \leq s_{-} \leq s_{+} < \infty$ satisfying
\begin{equation}\label{15.98}
\frac{\lambda(s_{+})}{\lambda(s_{-})} = e.
\end{equation}
Then using $(\ref{15.24})$ and the computations in Theorem $\ref{t8.2}$,

\begin{equation}\label{15.99}
\aligned
\frac{d}{ds} (\epsilon, y^{2} Q)_{L^{2}} + \frac{\lambda_{s}}{\lambda} \| y Q \|_{L^{2}}^{2} + 4 (\epsilon_{2}, \frac{Q}{2} + y Q_{y})_{L^{2}} = O(|\gamma_{s} + 1 - \frac{x_{s}}{\lambda} \xi(s) - \xi(s)^{2}| \| \epsilon \|_{L^{2}}) + O(|\frac{\lambda_{s}}{\lambda}| \| \epsilon \|_{L^{2}}) \\ + O(|\frac{x_{s}}{\lambda} + 2 \xi(s)| \| \epsilon \|_{L^{2}}) + O(|\xi_{s} - \frac{\lambda_{s}}{\lambda} \xi(s)| \| \epsilon \|_{L^{2}}) + O(\| \epsilon \|_{L^{2}}^{2}) + O(\| \epsilon \|_{L^{2}}^{2} \| \epsilon \|_{L^{\infty}}^{3}).
\endaligned
\end{equation}
Then by Theorem $\ref{t15.7}$ and the fundamental theorem of calculus,
\begin{equation}\label{15.100}
\| y Q \|_{L^{2}}^{2} + 4 \int_{s_{-}}^{s_{+}} (\epsilon_{2}, \frac{Q}{2} + x Q_{x})_{L^{2}} = O(\eta_{\ast}).
\end{equation}
Therefore, there exists $s' \in [s_{-}, s_{+}]$ such that
\begin{equation}\label{15.101}
(\epsilon_{2}, \frac{Q}{2} + x Q_{x})_{L^{2}} < 0.
\end{equation}
Make a Galilean transformation setting $\xi(s') = 0$ and a translation in space such that $x(s'') = 0$, where $s''$ is the other endpoint of the interval of integration. Also rescale so that $\lambda(s') = \frac{1}{\eta_{1}} T_{\ast}^{1/200}$. Since $s' \geq 0$, there exists some $j \geq 0$ such that $s_{j} \leq s' + T_{\ast} < s_{j + 1}$. By Theorem $\ref{t15.8}$ and $(\ref{15.39})$,
\begin{equation}\label{15.102}
\int_{s'}^{s_{j + 1 + J}} |\frac{\lambda_{s}}{\lambda}| ds \lesssim J \Rightarrow \frac{1}{\eta_{1}} \leq \lambda(t) \leq \frac{1}{\eta_{1}} T_{\ast}^{1/100}, \qquad \sup_{s \in [s', s_{j + 1 + J}]} \frac{|\xi(s)|}{\lambda(s)} \ll \eta_{0}, \qquad \sup_{s \in [s', s_{j + 1 + J}]} |x(s)| \ll T_{\ast}^{1/25}.
\end{equation}
Then by Theorem $\ref{t15.7}$ and Theorem $\ref{t15.8}$,
\begin{equation}\label{15.103}
\int_{s'}^{s_{j + 1 + J}} \| \epsilon(s) \|_{L^{2}}^{2} ds \lesssim 2^{-(j + 1 + J)} \eta_{\ast} + T_{\ast}^{1/50} \eta_{\ast} (\int_{s'}^{s_{j + 1 + J}} \| \epsilon(s) \|_{L^{2}}^{2} ds) \lesssim 2^{-(j + 1 + J)} \eta_{\ast},
\end{equation}
and therefore by definition of $s_{j + 1 + J}$,
\begin{equation}\label{15.104}
\int_{s'}^{s_{j + 1 + J}} \| \epsilon(s) \|_{L^{2}} ds \lesssim 1.
\end{equation}
Then, $(\ref{15.102})$ holds on the interval $[s', s_{j + 1 + 2J}]$, and arguing by induction, for any $k \geq 1$,
\begin{equation}\label{15.105}
\int_{s'}^{s_{j + k}} \| \epsilon(s) \|_{L^{2}}^{2} ds \lesssim 2^{-j - k} \eta_{\ast},
\end{equation}
and
\begin{equation}\label{15.106}
\int_{s'}^{s_{j + k}} \| \epsilon(s) \|_{L^{2}} ds \lesssim 1,
\end{equation}
with implicit constant independent of $k$. Taking $k \rightarrow \infty$,
\begin{equation}\label{15.107}
\int_{s'}^{\infty} \| \epsilon(s) \|_{L^{2}}^{2} ds = 0,
\end{equation}
which implies that $\epsilon(s) = 0$ for all $s \geq s'$. Therefore, $u$ is a soliton solution to $(\ref{1.1})$.
\end{proof}

\subsection{Almost monotone $\lambda(t)$}
In the nonsymmetric case, when $\sup(I) = \infty$, $u$ is equal to a soliton solution, and when $\sup(I) < \infty$, $u$ is the pseudoconformal transformation of the soliton solution.

\begin{theorem}\label{t15.10}
If $u$ satisfies the conditions of Theorem $\ref{t15.1}$, blows up forward in time, and
\begin{equation}\label{15.108}
\sup(I) = \infty,
\end{equation}
then $u$ is equal to a soliton solution.
\end{theorem}
\begin{proof}
As in Theorem $\ref{t14.0}$, for any integer $k \geq 0$, let
\begin{equation}\label{15.109}
I(k) = \{ s \geq 0 : 2^{-k + 2} \leq \tilde{\lambda}(s) \leq 2^{-k + 3} \}.
\end{equation}
As in the proof of Theorem $\ref{t14.0}$, there exists a sequence $k_{n} \nearrow \infty$ such that
\begin{equation}\label{15.110}
|I(k_{n})| 2^{-2k_{n}} \geq \frac{1}{k_{n}^{2}},
\end{equation}
and that $|I(k_{n})| \geq |I(k)|$ for all $k \leq k_{n}$.

\begin{lemma}\label{l15.11}
For $n$ sufficiently large, there exists $s_{n} \in I(k_{n})$ such that
\begin{equation}\label{15.111}
\| \epsilon(s_{n}) \|_{L^{2}} \lesssim k_{n}^{2} 2^{-2k_{n}}.
\end{equation}
\end{lemma}
\begin{proof}
Let $I(k_{n}) = [a_{n}, b_{n}]$. 
By Theorem $\ref{t15.8}$,
\begin{equation}\label{15.112}
\int_{I(k_{n})} \| \epsilon(s) \|_{L^{2}}^{2} ds \lesssim \eta_{\ast},
\end{equation}
Then, using the virial identity in $(\ref{15.99})$,
\begin{equation}\label{15.113}
\int_{a_{n}}^{\frac{3 a_{n} + b_{n}}{4}} (\epsilon_{2}, \frac{Q}{2} + x Q_{x})_{L^{2}} ds = O(\eta_{\ast}) + O(1).
\end{equation}
Therefore, by the mean value theorem, there exists $s_{n}^{-} \in [a_{n}, \frac{3 a_{n} + b_{n}}{4}]$ such that
\begin{equation}\label{15.114}
|(\epsilon_{2}(s_{n}^{-}), \frac{Q}{2} + x Q_{x})_{L^{2}}| \lesssim \frac{1}{|I(k_{n})|}.
\end{equation}
By a similar calculation, there exists $s_{n}^{+} \in [\frac{a_{n} + 3 b_{n}}{4}, b_{n}]$ such that
\begin{equation}\label{15.115}
|(\epsilon_{2}(s_{n}^{+}), \frac{Q}{2} + x Q_{x})_{L^{2}}| \lesssim \frac{1}{|I(k_{n})|}.
\end{equation}
Therefore, by Theorem $\ref{t15.7}$, $(\ref{15.114})$ and $(\ref{15.115})$ imply
\begin{equation}\label{15.116}
\int_{s_{n}^{-}}^{s_{n}^{+}} \| \epsilon(s) \|_{L^{2}}^{2} ds \lesssim \frac{1}{|I(k_{n})|}.
\end{equation}
Indeed, rescale so that $\lambda(s_{n}^{-}) = \frac{1}{\eta_{1}}$. Then by Galilean transformation, suppose $\xi(s_{n}^{-}) = 0$ and by translation in space $x(s_{n}^{+}) = 0$. For all $s \in [s_{n}^{-}, s_{n}^{+}]$, by $(\ref{15.39})$ and Theorem $\ref{t15.8}$,
\begin{equation}\label{15.117}
\frac{|\xi(s)|}{\lambda(s)} \lesssim \eta_{1} \eta_{\ast}, \qquad |x(s)| \ll T_{\ast}^{1/25}.
\end{equation}
Therefore, by Theorem $\ref{t15.7}$ and $(\ref{15.39})$,
\begin{equation}
\int_{s_{n}^{-}}^{s_{n}^{+}} \| \epsilon(s) \|_{L^{2}}^{2} ds \lesssim \frac{1}{|I(k_{n})|} + \eta_{\ast} T_{\ast}^{1/25} (\int_{s_{n}^{-}}^{s_{n}^{+}} \| \epsilon(s) \|_{L^{2}}^{2} ds) \lesssim \frac{1}{|I(k_{n})|}.
\end{equation}
\begin{remark}
To make these computations completely rigorous, partition $[s_{n}^{-}, s_{n}^{+}]$ into a dyadic integer number of subintervals of length $\sim \frac{1}{\eta_{\ast}}$ and then following the arguments proving Theorem $\ref{t15.8}$, it is possible to prove that $(\ref{15.117})$ holds on subintervals of length $\sim \frac{J}{\eta_{\ast}}$, and then by induction, $(\ref{15.117})$ holds on $[s_{n}^{-}, s_{n}^{+}]$, which by Theorem $\ref{t15.7}$ implies that $(\ref{15.116})$ holds.
\end{remark}
Then by the mean value theorem,
\begin{equation}\label{15.118}
\| \epsilon(s_{n}) \|_{L^{2}}^{2} \lesssim \frac{1}{|I(k_{n})|^{2}}.
\end{equation}
Since $|I(k_{n})| \geq 2^{2k_{n}} k_{n}^{-2}$, the proof of Lemma $\ref{l15.11}$ is complete.
\end{proof}

Make a Galilean transformation so that $\xi(s_{n}) = 0$. Then by $(\ref{15.39})$, since $\lambda(s) \gtrsim 2^{-k_{n}}$ for all $s \in [0, s_{n}]$,
\begin{equation}\label{15.119}
\frac{|\xi(s)|}{\lambda(s)} \lesssim 2^{k_{n}} \eta_{\ast}.
\end{equation}
Now let $m$ be the smallest integer such that
\begin{equation}\label{15.120}
\frac{2^{2k_{n}}}{k_{n}^{2}} 2^{m} \geq |I(k_{n})|.
\end{equation}
Since $|I(k)| \leq |I(k_{n})|$ for all $0 \leq k \leq k_{n}$, $(\ref{15.120})$ implies that
\begin{equation}\label{15.121}
|s_{n}| \leq 2^{2k_{n} + m + 1}.
\end{equation}
Let $r_{n}$ be the smallest integer that satisfies
\begin{equation}\label{15.122}
2^{\frac{2k_{n} + m + 1}{3}} 2^{k_{n}} \frac{1}{\eta_{1}} \leq 2^{r_{n}}.
\end{equation}
Then, as in the proof of Theorem $\ref{t15.4}$, setting $t_{n} = s^{-1}(s_{n})$, $(\ref{15.119})$ and induction on frequency implies
\begin{equation}\label{15.123}
\| P_{\geq r_{n}} u \|_{U_{\Delta}^{2}([0, t_{n}] \times \mathbb{R})} \lesssim \eta_{\ast},
\end{equation}
and
\begin{equation}\label{15.124}
\| P_{\geq r_{n} + \frac{k_{n}}{4} + \frac{m}{4}} u \|_{U_{\Delta}^{2}([0, t_{n}] \times \mathbb{R})} \lesssim k_{n}^{2} 2^{-2k_{n}} 2^{-m}.
\end{equation}
Furthermore,
\begin{equation}\label{15.125}
E(P_{\leq r_{n} + \frac{k_{n}}{4} + \frac{m}{4}} u(t_{n})) \lesssim (k_{n}^{2} 2^{-2k_{n}} 2^{-m} 2^{r_{n} + \frac{k_{n}}{4} + \frac{m}{4}})^{2} \sim (k_{n}^{2} 2^{-\frac{k_{n}}{12} - \frac{5m}{12}} \eta_{1}^{-1})^{2}.
\end{equation}
and
\begin{equation}\label{15.126}
\sup_{t \in [0, t_{n}]} E(P_{\leq r_{n} + \frac{k_{n}}{4} + \frac{m}{4}} u(t)) \lesssim (k_{n}^{2} 2^{-\frac{k_{n}}{12} - \frac{5m}{12}} \eta_{1}^{-1})^{2}.
\end{equation}
By $(\ref{15.54})$, if $\xi_{n}(s)$ is the $\xi(s)$ in $(\ref{15.23})$ for which $\xi_{n}(s_{n}) = 0$,
\begin{equation}
\sup_{0 \leq s \leq s_{n}} \frac{|\xi_{n}(s)|^{2}}{\lambda(s)^{2}} \lesssim (k_{n}^{2} 2^{-\frac{k_{n}}{12} - \frac{5m}{12}} \eta_{1}^{-1})^{2},
\end{equation}
which implies that $\xi(s)$ converges to some $\xi_{\infty}$ as $s \rightarrow \infty$. Making a Galilean transformation that maps $\xi_{\infty}$ to the origin and taking $n \rightarrow \infty$, since $m \geq 0$, $(\ref{15.26})$ implies that $E(u_{0}) = 0$. Therefore, by the Gagliardo--Nirenberg inequality, $u_{0}$ is a soliton.
\end{proof}

When $\sup(I) < \infty$, suppose without loss of generality that $\sup(I) = 0$, and
\begin{equation}\label{15.127}
\sup_{-1 < t < 0} \| \epsilon(t) \|_{L^{2}} \leq \eta_{\ast}.
\end{equation}
Then decompose
\begin{equation}\label{15.128}
u(t,x) = \frac{e^{-i \gamma(t)} e^{-ix \frac{\xi(t)}{\lambda(t)}}}{\lambda(t)^{1/2}} Q(\frac{x - x(t)}{\lambda(t)}) + \frac{e^{-i \gamma(t)} e^{-ix \frac{\xi(t)}{\lambda(t)}}}{\lambda(t)^{1/2}} \epsilon(t,\frac{x - x(t)}{\lambda(t)}).
\end{equation}
Then apply the pseudoconformal transformation to $u(t,x)$. For $-\infty < t < -1$,
\begin{equation}\label{15.129}
\aligned
v(t,x) = \frac{1}{t^{1/2}} \overline{u(\frac{1}{t}, \frac{x}{t})} e^{i x^{2}/4t} = \frac{1}{t^{1/2}} \frac{e^{i \gamma(1/t)} e^{ix \frac{\xi(\frac{1}{t})}{\lambda(\frac{1}{t})}}}{\lambda(1/t)^{1/2}} Q(\frac{x - t x(\frac{1}{t})}{t \lambda(1/t)}) e^{i x^{2}/4t} \\ + \frac{1}{t^{1/2}} \frac{e^{i \gamma(1/t)} e^{ix \frac{\xi(\frac{1}{t})}{\lambda(\frac{1}{t})}}}{\lambda(1/t)^{1/2}}\overline{\epsilon(\frac{1}{t}, \frac{x - t x(\frac{1}{t})}{t \lambda(1/t)})} e^{i x^{2}/4t}.
\endaligned
\end{equation}
Since the $L^{2}$ norm is preserved by the pseudoconformal transformation,
\begin{equation}\label{15.130}
\aligned
\lim_{t \searrow -\infty} \|  \frac{1}{t^{1/2}} \frac{e^{i \gamma(1/t)} e^{ix \frac{\xi(\frac{1}{t})}{\lambda(\frac{1}{t})}}}{\lambda(1/t)^{1/2}}\overline{\epsilon(\frac{1}{t}, \frac{x - t x(\frac{1}{t})}{t \lambda(1/t)})} e^{i x^{2}/4t} \|_{L^{2}} = 0.
\endaligned
\end{equation}

Next,
\begin{equation}\label{15.131}
 \frac{1}{t^{1/2}} \frac{e^{i \gamma(1/t)} e^{ix \frac{\xi(\frac{1}{t})}{\lambda(\frac{1}{t})}}}{\lambda(1/t)^{1/2}} Q(\frac{x - t x(\frac{1}{t})}{t \lambda(1/t)}) e^{ix \frac{x(\frac{1}{t})}{2}} e^{-i \frac{t}{4} x(\frac{1}{t})^{2}},
\end{equation}
is of the form
\begin{equation}\label{15.132}
e^{-i \tilde{\gamma}(t)} e^{-ix \frac{\tilde{\xi}(t)}{\tilde{\lambda}(t)}} \tilde{\lambda}(t)^{-1/2} Q(\frac{x - \tilde{x}(t)}{\tilde{\lambda}(t)}),
\end{equation}
where
\begin{equation}\label{15.133}
\tilde{\gamma}(t) = \gamma(\frac{1}{t}) - \frac{x(\frac{1}{t})^{2} t}{4}, \qquad \tilde{\xi}(t) = \xi(\frac{1}{t}) + \frac{x(\frac{1}{t})}{2} t \lambda(\frac{1}{t}), \qquad \tilde{\lambda}(t) = t \lambda(\frac{1}{t}), \qquad \tilde{x}(t) = t x(\frac{1}{t}).
\end{equation}
Also,
\begin{equation}\label{15.134}
\aligned
\|  \frac{1}{t^{1/2}} \frac{e^{i \gamma(1/t)} e^{ix \frac{\xi(\frac{1}{t})}{\lambda(\frac{1}{t})}}}{\lambda(1/t)^{1/2}} Q(\frac{x - t x(\frac{1}{t})}{t \lambda(1/t)}) e^{i\frac{x^{2}}{4t}} - \frac{1}{t^{1/2}} \frac{e^{i \gamma(1/t)} e^{ix \frac{\xi(\frac{1}{t})}{\lambda(\frac{1}{t})}}}{\lambda(1/t)^{1/2}} Q(\frac{x - t x(\frac{1}{t})}{t \lambda(1/t)}) e^{ix \frac{x(\frac{1}{t})}{2}} e^{i t x(\frac{1}{t})^{2}} \|_{L^{2}} \\
= \| \frac{1}{t^{1/2} \lambda(\frac{1}{t})^{1/2}} Q(\frac{x - t x(\frac{1}{t})}{t \lambda(1/t)}) (e^{i\frac{(x - t x(\frac{1}{t}))^{2}}{4t}} - 1) \|_{L^{2}}.
\endaligned
\end{equation}
As in $(\ref{14.25})$ and $(\ref{14.26})$,
\begin{equation}\label{15.135}
\lim_{t \searrow -\infty} \| \frac{1}{t^{1/2} \lambda(\frac{1}{t})^{1/2}} Q(\frac{x - t x(\frac{1}{t})}{t \lambda(1/t)}) (e^{i\frac{(x - t x(\frac{1}{t}))^{2}}{4t}} - 1) \|_{L^{2}} = 0.
\end{equation}
Therefore, by time reversal symmetry, $v$ satisfies the conditions of Theorem $\ref{t15.1}$, and $v$ is a solution that blows up backward in time at $\inf(I) = -\infty$, so therefore, by Theorem $\ref{t15.10}$, $v$ must be a soliton. Therefore, $u$ is the pseudoconformal transformation of a soliton, which proves Theorem $\ref{t1.2}$.

\section*{Acknowledgement}
The author was partially supported by NSF Grant DMS--$1764358$. The author was also greatly helped by many stimulating discussions with Frank Merle at Cergy-Pontoise and University of Chicago, as well as his constant encouragement to pursue this problem. The author would also like to recognize the many helpful discussions that he had with Svetlana Roudenko and Anudeep Kumar Arora, both at George Washington University and Florida International University.

The author was also greatly assisted by discussions of the nonlinear Schr{\"o}dinger equation with his PhD student, Dr. Zehua Zhao. The author also gratefully acknowledges discussions with Chenjie Fan and Jason Murphy at Oberwolfach, and Robin Neumayer at the Institute for Advanced Study.

The author would also like to thank Jonas L{\"u}hrmann and Cristian Gavrus for many helpful discussions on the related mass-critical generalized KdV equation.

\bibliography{biblio}

\newcommand{\etalchar}[1]{$^{#1}$}
\begin{thebibliography}{CKS{\etalchar{+}}02}

\bibitem[BL78]{MR512091}
Henri Berestycki and Pierre-Louis Lions.
\newblock Existence d'ondes solitaires dans des probl\`emes non-lin\'{e}aires
  du type {K}lein-{G}ordon.
\newblock {\em C. R. Acad. Sci. Paris S\'{e}r. A-B}, 287(7):A503--A506, 1978.

\bibitem[BLP81]{berestycki1981ode}
H~Berestycki, PL~Lions, and LA~Peletier.
\newblock An {ODE} approach to the existence of positive solutions for
  semilinear problems in ${R}^{N}$.
\newblock {\em Indiana University Mathematics Journal}, 30(1):141--157, 1981.

\bibitem[Bou98]{bourgain1998refinements}
Jean Bourgain.
\newblock Refinements of {S}trichartz' inequality and applications to 2d-{NLS}
  with critical nonlinearity.
\newblock {\em International Mathematics Research Notices}, 1998(5):253--283,
  1998.

\bibitem[CKS{\etalchar{+}}02]{colliander2002almost}
J~Colliander, M~Keel, Gigliola Staffilani, H~Takaoka, and T~Tao.
\newblock Almost conservation laws and global rough solutions to a nonlinear
  {S}chr{\"o}dinger equation.
\newblock {\em Mathematical Research Letters}, 9(5):659--682, 2002.

\bibitem[CL82]{cazenave1982orbital}
Thierry Cazenave and Pierre-Louis Lions.
\newblock Orbital stability of standing waves for some nonlinear
  {S}chr{\"o}dinger equations.
\newblock {\em Communications in Mathematical Physics}, 85(4):549--561, 1982.

\bibitem[CW90]{cazenave1990cauchy}
Thierry Cazenave and Fred~B Weissler.
\newblock The {C}auchy problem for the critical nonlinear {S}chr{\"o}dinger
  equation in ${H}^{s}$.
\newblock {\em Nonlinear Analysis: Theory, Methods \& Applications},
  14(10):807--836, 1990.

\bibitem[Dod15]{dodson2015global}
Benjamin Dodson.
\newblock Global well-posedness and scattering for the mass critical nonlinear
  {S}chr{\"o}dinger equation with mass below the mass of the ground state.
\newblock {\em Advances in mathematics}, 285:1589--1618, 2015.

\bibitem[Dod16a]{dodson2016global2}
Benjamin Dodson.
\newblock Global well-posedness and scattering for the defocusing,
  ${L}^{2}$-critical, nonlinear {S}chr{\"o}dinger equation when $d = 2$.
\newblock {\em Duke Mathematical Journal}, 165(18):3435--3516, 2016.

\bibitem[Dod16b]{dodson2016global}
Benjamin Dodson.
\newblock Global well-posedness and scattering for the defocusing,
  ${L}^{2}$-critical, nonlinear {S}chr{\"o}dinger equation when d=1.
\newblock {\em American Journal of Mathematics}, 138(2):531--569, 2016.

\bibitem[Dod19]{dodson2019defocusing}
Benjamin Dodson.
\newblock {\em Defocusing {N}onlinear {S}chr{\"o}dinger {E}quations}, volume
  217.
\newblock Cambridge University Press, 2019.

\bibitem[Dod20]{dodson20202}
Benjamin Dodson.
\newblock The ${L}^{2}$ sequential convergence of a solution to the one
  dimensional, mass-critical {NLS} above the ground state.
\newblock {\em arXiv preprint arXiv:2011.02569}, 2020.

\bibitem[Fan18]{fan20182}
Chenjie Fan.
\newblock The ${L}^{2}$ weak sequential convergence of radial focusing mass
  critical {NLS} solutions with mass above the ground state.
\newblock {\em International Mathematics Research Notices}, 2018.

\bibitem[Gla77]{glassey1977blowing}
Robert~T Glassey.
\newblock On the blowing up of solutions to the {C}auchy problem for nonlinear
  {S}chr{\"o}dinger equations.
\newblock {\em Journal of Mathematical Physics}, 18(9):1794--1797, 1977.

\bibitem[GV79]{ginibre1979class}
Jean Ginibre and Giorgio Velo.
\newblock On a class of nonlinear {S}chr{\"o}dinger equations. {I}. the
  {C}auchy problem, general case.
\newblock {\em Journal of Functional Analysis}, 32(1):1--32, 1979.

\bibitem[GV85]{ginibre1985global}
Jean Ginibre and Giorgio Velo.
\newblock The global {C}auchy problem for the non linear {S}chr{\"o}dinger
  equation revisited.
\newblock In {\em Annales de l'Institut Henri Poincare (C) Non Linear
  Analysis}, volume~2, pages 309--327. Elsevier, 1985.

\bibitem[GV92]{ginibre1992smoothing}
Jean Ginibre and Giorgio Velo.
\newblock Smoothing properties and retarded estimates for some dispersive
  evolution equations.
\newblock {\em Communications in mathematical physics}, 144(1):163--188, 1992.

\bibitem[HHK09]{hadac2009well}
Martin Hadac, Sebastian Herr, and Herbert Koch.
\newblock Well-posedness and scattering for the {KP}-{II} equation in a
  critical space.
\newblock In {\em Annales de l'Institut Henri Poincare (C) Non Linear
  Analysis}, volume~26, pages 917--941. Elsevier, 2009.

\bibitem[Kat87]{kato1987nonlinear}
Tosio Kato.
\newblock On nonlinear {S}chr{\"o}dinger equations.
\newblock In {\em Annales de l'IHP Physique th{\'e}orique}, volume~46, pages
  113--129, 1987.

\bibitem[Ker06]{keraani2006blow}
Sahbi Keraani.
\newblock On the blow up phenomenon of the critical nonlinear {S}chr{\"o}dinger
  equation.
\newblock {\em Journal of Functional Analysis}, 235(1):171--192, 2006.

\bibitem[KLVZ09]{killip2009characterization}
Rowan Killip, Dong Li, Monica Visan, and Xiaoyi Zhang.
\newblock Characterization of minimal-mass blowup solutions to the focusing
  mass-critical {NLS}.
\newblock {\em SIAM journal on mathematical analysis}, 41(1):219--236, 2009.

\bibitem[KT07]{koch2007priori}
Herbert Koch and Daniel Tataru.
\newblock A priori bounds for the 1{D} cubic {NLS} in negative {S}obolev
  spaces.
\newblock {\em International Mathematics Research Notices},
  2007(9):rnm053--rnm053, 2007.

\bibitem[Kwo89]{kwong1989uniqueness}
Man~Kam Kwong.
\newblock Uniqueness of positive solutions of ${\Delta} u - u+ u^{p}$= 0 in
  ${R}^{n}$.
\newblock {\em Archive for Rational Mechanics and Analysis}, 105(3):243--266,
  1989.

\bibitem[Mer92]{merle1992uniqueness}
Frank Merle.
\newblock On uniqueness and continuation properties after blow-up time of
  self-similar solutions of nonlinear {S}chr{\"o}dinger equation with critical
  exponent and critical mass.
\newblock {\em Communications on pure and applied mathematics}, 45(2):203--254,
  1992.

\bibitem[Mer93]{merle1993determination}
Frank Merle.
\newblock Determination of blow-up solutions with minimal mass for nonlinear
  {S}chr{\"o}dinger equations with critical power.
\newblock {\em Duke Mathematical Journal}, 69(2):427--454, 1993.

\bibitem[Mer01]{merle2001existence}
Frank Merle.
\newblock Existence of blow-up solutions in the energy space for the critical
  generalized {K}d{V} equation.
\newblock {\em Journal of the American Mathematical Society}, 14(3):555--578,
  2001.

\bibitem[MM02]{martel2002stability}
Yvan Martel and Frank Merle.
\newblock Stability of blow-up profile and lower bounds for blow-up rate for
  the critical generalized {K}d{V} equation.
\newblock {\em Annals of mathematics}, pages 235--280, 2002.

\bibitem[MR05]{merle2005blow}
Frank Merle and Pierre Raphael.
\newblock The blow-up dynamic and upper bound on the blow-up rate for critical
  nonlinear {S}chr{\"o}dinger equation.
\newblock {\em Annals of mathematics}, pages 157--222, 2005.

\bibitem[NS11]{nakanishi2011invariant}
Kenji Nakanishi and Wilhelm Schlag.
\newblock {\em Invariant manifolds and dispersive Hamiltonian evolution
  equations}, volume~14.
\newblock European Mathematical Society, 2011.

\bibitem[PV09]{planchon2009bilinear}
Fabrice Planchon and Luis Vega.
\newblock Bilinear virial identities and applications.
\newblock In {\em Annales scientifiques de l'Ecole normale sup{\'e}rieure},
  volume~42, pages 261--290, 2009.

\bibitem[Str77a]{strauss1977existence}
Walter~A Strauss.
\newblock Existence of solitary waves in higher dimensions.
\newblock {\em Communications in Mathematical Physics}, 55(2):149--162, 1977.

\bibitem[Str77b]{strichartz1977restrictions}
Robert~S Strichartz.
\newblock Restrictions of {F}ourier transforms to quadratic surfaces and decay
  of solutions of wave equations.
\newblock {\em Duke Mathematical Journal}, 44(3):705--714, 1977.

\bibitem[Tao06]{tao2006nonlinear}
Terence Tao.
\newblock {\em Nonlinear dispersive equations: local and global analysis}.
\newblock Number 106. American Mathematical Soc., 2006.

\bibitem[TVZ08]{tao2008minimal}
Terence Tao, Monica Visan, and Xiaoyi Zhang.
\newblock Minimal-mass blowup solutions of the mass-critical {NLS}.
\newblock In {\em Forum Mathematicum}, volume~20, pages 881--919. De Gruyter,
  2008.

\bibitem[Wei83]{weinstein1983nonlinear}
Michael~I Weinstein.
\newblock Nonlinear {S}chr{\"o}dinger equations and sharp interpolation
  estimates.
\newblock {\em Communications in Mathematical Physics}, 87(4):567--576, 1983.

\bibitem[Wei86]{weinstein1986structure}
Michael~I Weinstein.
\newblock On the structure and formation of singularities in solutions to
  nonlinear dispersive evolution equations.
\newblock {\em Communications in Partial Differential Equations},
  11(5):545--565, 1986.

\bibitem[Yaj87]{yajima1987existence}
Kenji Yajima.
\newblock Existence of solutions for {S}chr{\"o}dinger evolution equations.
\newblock {\em Communications in Mathematical Physics}, 110(3):415--426, 1987.

\end{thebibliography}
\bibliographystyle{alpha}

\end{document}